\documentclass[11pt,letterpaper]{amsart}

\usepackage{amssymb,amsmath,amscd}
\usepackage{mathtools}
\usepackage{amsfonts}
\usepackage{graphicx}
\usepackage{xcolor}
\usepackage{mathrsfs}
\usepackage{stmaryrd}
\usepackage{epsfig,color}
\usepackage{blindtext}
\usepackage{enumerate}

\usepackage{geometry}
\usepackage[colorlinks,linktocpage]{hyperref}
\hypersetup{linkcolor=[rgb]{0,0,0.715}}
\hypersetup{citecolor=[rgb]{0,0.715,0}}

\newtheorem{thm}{Theorem}[section]
\newtheorem{cor}[thm]{Corollary}
\newtheorem{prop}[thm]{Proposition}
\newtheorem{lem}[thm]{Lemma}

\newtheorem{quest}[thm]{Question}

\newtheorem{mainthm}{Theorem}

\theoremstyle{definition}
\newtheorem{defn}[thm]{Definition}

\newtheorem{notn}[thm]{Notation}

\usepackage{comment}

\theoremstyle{remark}
\newtheorem{rem}[thm]{Remark}
\newtheorem{rems}[thm]{Remarks}

\newtheorem{conv}[thm]{Convention}

\newcommand{\mf}{\mathbf}

\newcommand{\mc}{\mathcal}

\newcommand{\mk}{\mathfrak}

\let\phi\varphi

\newcommand{\N}{\mathbb{N}}

\newcommand{\R}{\mathbb{R}}
\newcommand{\Z}{\mathbb{Z}}
\renewcommand{\subset}{\subseteq}
\newcommand{\defeq}{\mathrel{\mathop:}=}

\newcommand{\haus}{\mathcal{H}}

\newcommand{\tran}{\mathcal{T}}

\newcommand{\vol}{\mathrm{vol}}
\newcommand{\dvol}{\mathrm{dvol}}

\newcommand{\Ric}{\mathrm{Ric}}
\newcommand{\dist}{d}

\newcommand{\meas}{\mathfrak{m}}

\newcommand{\di}{\mathop{}\!\mathrm{d}}
\DeclareMathOperator{\RCD}{RCD}

\makeatletter
\let\c@equation\c@thm
\makeatother
\numberwithin{equation}{section}

\title[Nonnegative Ricci curvature and linear volume growth]{On the topology of manifolds with nonnegative Ricci curvature and linear volume growth}

\author{Dimitri Navarro}
\address[Dimitri Navarro]{Department of Mathematics, University of California, Santa Cruz, CA, USA.}
\email{dnavar17@ucsc.edu}

\author{Jiayin Pan}
\address[Jiayin Pan]{Department of Mathematics, University of California, Santa Cruz, CA, USA.}
\email{jpan53@ucsc.edu}

\author{Xingyu Zhu}
\address[Xingyu Zhu]{Michigan State University, East Lansing, MI, USA.}
\email{zhuxing3@msu.edu}

\begin{document}

\begin{abstract}
   Understanding the relations between geometry and topology is a central theme in Riemannian geometry. We establish two results on the fundamental groups of open (complete and noncompact) $n$-manifolds with nonnegative Ricci curvature and linear volume growth. First, we show that the fundamental group of such a manifold contains a subgroup $\mathbb{Z}^k$ of finite index, where $0\le k\le n-1$. Second, we prove that if the Ricci curvature is positive everywhere, then the fundamental group is finite. The proofs are based on an analysis of the equivariant asymptotic geometry of successive covering spaces and a plane/halfplane rigidity result for RCD spaces.
\end{abstract}


\maketitle

\tableofcontents

\parskip=5pt plus 0.5pt

\section{Introduction}\label{sec:intro}
About half a century ago, Calabi \cite{Calabi75} and Yau \cite{Yau76} independently proved that an open (complete and noncompact) Riemannian manifold with nonnegative Ricci curvature must have at
least linear volume growth. These manifolds of minimal volume growth were studied extensively by Sormani \cite{Sor98,Sor00a,Sor00c,Sor00b} in the late 90s.

\begin{defn}\label{def:linear_vol}
Let $M$ be an open manifold with $\Ric\ge 0$. We say that $M$ has \emph{linear (or minimal) volume growth}, if for some point (hence for all points) $p\in M$, we have the following:
$$\limsup_{r\to\infty} \dfrac{\vol B_r(p)}{r} <\infty.$$
\end{defn}

While open manifolds with nonnegative Ricci curvature and linear volume growth may have infinite topological types \cite{Menguy00,Jiang21}, Sormani proved that their fundamental groups are always finitely generated \cite{Sor00b}. In contrast, for open manifolds with nonnegative Ricci curvature in general, their fundamental groups need not be finitely generated, as shown by the recent counterexamples to the Milnor conjecture by Bruè--Naber--Semola \cite{BNS_7,BNS_6}.  

When the fundamental group of an open manifold with nonnegative Ricci curvature is finitely generated, the works of Milnor \cite{Milnor68} and Gromov \cite{Gromov81} imply that it is virtually nilpotent, that is, it contains a nilpotent subgroup of finite index; also, see the work by Kapovitch--Wilking \cite{KW}, where an index bound $C(n)$ is achieved. Conversely, Wei \cite{Wei88} and Wilking \cite{Wilking00} showed that every finitely generated virtually nilpotent group can be realized as the fundamental group of some open manifold of positive Ricci curvature.

In this paper, we establish two results concerning the fundamental group of open manifolds with linear volume growth. First, we show that the fundamental group is always virtually abelian, ruling out the possibility of torsion-free nilpotent (non-abelian) groups arising under these conditions. Second, we prove that if the Ricci curvature is positive everywhere, then the fundamental group must be finite.

\begin{mainthm}\label{thm:vir_abel}
If $M^n$ is an open manifold with $\Ric\ge 0$ and linear volume growth, then $\pi_1(M)$ contains a subgroup $\mathbb{Z}^k$ of finite index, where $0\le k\le n-1$.
\end{mainthm}

\begin{mainthm}\label{thm:finite}
If $M^n$ is an open manifold with $\Ric>0$ and linear volume growth, then $\pi_1(M)$ is finite.
\end{mainthm}

We point out that Theorem \ref{thm:vir_abel} also holds for $\RCD(0,N)$ spaces with linear volume growth, as the proof carries over to $\RCD(0,N)$ spaces with some clear modifications.

\begin{thm}\label{thm:rcd_abel}
    If $(X,\dist,\meas)$ is a noncompact $\RCD(0,N)$ space ($N\in [1,\infty)$) with linear volume growth, then $\pi_1(X)$ contains a subgroup $\mathbb{Z}^k$ of finite index, where $0\le k \le \lfloor N\rfloor-1$.
\end{thm}

\begin{rems}
   Here are some remarks on Theorems \ref{thm:vir_abel} and \ref{thm:finite}.\\
   (1) As an application of their celebrated splitting theorem, Cheeger--Gromoll proved that the fundamental group of a closed manifold $M^n$ with nonnegative Ricci curvature must have a subgroup $\mathbb{Z}^k$ of finite index, where $0\le k\le n$ \cite{CG_split}. Our Theorem \ref{thm:vir_abel} can be viewed as an extension of their result to open manifolds whose largeness, in terms of volume growth, is minimal.\\
   (2) Similarly, Theorem \ref{thm:finite} can be seen as an extension of a classical result by Bonnet and Myers, which asserts that every closed manifold with positive Ricci curvature has a finite fundamental group.\\
   (3) The estimate $k\le n-1$ in Theorem \ref{thm:vir_abel} follows directly from a volume argument by Anderson \cite[Theorem 1.1]{Anderson90}. When $k=n-1$, $M$ is flat with an $n-1$ dimensional soul \cite[Theorem 5]{Ye24}.\\
   (4) In \cite{Pan21,Pan22}, the second-named author proved that if a complete manifold has nonnegative Ricci curvature and small escape rate, then the fundamental group is virtually abelian. In the proof of Theorem \ref{thm:vir_abel}, we show that manifolds with linear volume growth have zero escape rate (see Theorem \ref{thm:zero_escape_rate}).\\
   (5) In Theorem \ref{thm:finite}, it is not possible to uniformly bound the order of $\pi_1(M)$ by a constant depending only on the dimension. To construct counterexamples, first consider a complete Riemannian metric $g_m$ on $\mathbb{R}^m$ with positive sectional curvature and linear volume growth, where $m\ge 2$. For instance, we can use a warped product 
  $$[0,\infty)\times_f S^{m-1},\quad g_m = dr^2 +f^2(r) ds_{m-1}^2,$$
  where $f(r)=\arctan r$ and $ds_{m-1}^2$ is the standard metric on $\mathbb{S}^{m-1}$. Now, we construct $M$ as the metric product of $(\mathbb{R}^m,g_m)$ and a lens space $\mathbb{S}^3/\mathbb{Z}_p$ of constant curvature. Then $M$ has positive Ricci curvature, linear volume growth, and $\pi_1(M)=\mathbb{Z}_p$. As $p$ can be arbitrarily large, this demonstrates that no uniform bound can be imposed on the order of $\pi_1(M)$. Consequently, by taking the metric product of this $M$ with a flat torus $\mathbb{T}^k$, one cannot find a uniform bound on the index of the $\Z^k$ subgroup provided by Theorem \ref{thm:vir_abel} as well.
\end{rems}

For a closed manifold $M$ with nonnegative Ricci curvature, Cheeger--Gromoll splitting theorem \cite{CG_split} implies that any torsion-free element in $\pi_1(M)$ gives rise to a splitting line on the universal cover. It is unclear to the authors whether this property extends to open manifolds with linear volume growth\footnote{This question was recently positively answered by Huang--Huang in \cite{huang2025} after we submitted the first version of this paper.}.

\begin{quest}\label{quest:torsion_free_line}
   Let $M^n$ be an open manifold with $\Ric\ge 0$ and linear volume growth. If $\pi_1(M)$ contains a torsion-free element, does the Riemannian universal cover $\widetilde{M}$ split off a line isometrically?
\end{quest}

We also raise the question below on strictly sub-quadratic volume growth.

\begin{quest}\label{quest:subquadratic}
   Let $M^n$ be an open manifold with $\Ric\ge 0$. Suppose that for some $\delta \in (0,1)$ and some point (hence for all points) $p\in M$, the following holds:
   $$\limsup_{r\to\infty} \dfrac{\vol B_r(p)}{r^{1+\delta}}=0.$$
   (1) Is $\pi_1(M)$ finitely generated?\\
   (2) Is $\pi_1(M)$ virtually abelian?\\
   (3) If we further assume $M$ has positive Ricci curvature, is $\pi_1(M)$ finite?
\end{quest}

\begin{rem}
    Regarding Question \ref{quest:subquadratic}(3), there are examples with positive Ricci curvature, quadratic volume growth, and $\pi_1(M)=\mathbb{Z}$. Indeed, Nabonnand \cite{Nabonnand80} constructed doubly warped metrics with positive Ricci curvature on $\R^3\times\mathbb{S}^1$ of the form
$$g=dr^2 + f^2(r) ds_{2}^2 + h^2(r)ds_1^2.
$$
Using the same choice of $f$ and $h$ as in \cite[Proposition 1.1]{PanWei}, one can check $f(r) \sim r^{2/3}$ and $h(r) \sim r^{-1/3}$ as $r\to\infty$. Then $(\R^3\times\mathbb{S}^1,g)$ has positive Ricci curvature and quadratic volume growth since we have the following asymptotic behavior:
$$\mathrm{vol}B_r(p) \sim \int_0^r (s^{2/3})^2\cdot s^{-1/3} ds \sim r^2.$$
\end{rem}

Below we outline our approach to proving Theorems \ref{thm:vir_abel} and \ref{thm:finite}.

Suppose an open manifold $(M,p)$ has nonnegative Ricci curvature and linear volume growth. We consider a Riemannian covering map $(\widehat{M},\hat{p})\to (M,p)$ with a finitely generated and torsion-free nilpotent covering group $\Lambda$. Then $\Lambda$ admits a series of normal subgroups
$$\{e\}=\Lambda_0 \triangleleft \Lambda_1 \triangleleft ... \triangleleft \Lambda_{k-1} \triangleleft \Lambda_k =\Lambda,$$
where $k$ is the rank of $\Lambda$, such that each $\Lambda_{j+1}/\Lambda_j$ is isomorphic to $\mathbb{Z}$. Denoting $\widehat{M}_j \coloneqq \widehat{M}/\Lambda_j$, we obtain a tower of successive covering spaces
$$ \widehat{M}=\widehat{M}_0\to \widehat{M}_{1} \to ... \to \widehat{M}_{k-1} \to \widehat{M}_k=M.$$
Each covering map $\widehat{M}_j \to \widehat{M}_{j+1}$ has covering group $\Lambda_{j+1}/\Lambda_j \simeq \mathbb{Z}$.

A key component of our proof is the Induction Theorem on the asymptotic cones of these covering spaces, which we state below.

\begin{thm}[Induction Theorem]\label{thm:induction}
For each $j=0,1,...,k$, the covering space $\widehat{M}_{k-j}$ has a unique asymptotic cone $(Y_{k-j},y_{k-j})$, isometric to either a Euclidean space $(\mathbb{R}^{j+1},0)$ or a Euclidean halfspace $(\mathbb{R}^j \times [0,\infty),0)$. Moreover, the asymptotic cone $Y_{k-j}$ has a limit renormalized measure as (a multiple of) the Lebesgue measure.
\end{thm}

Theorem \ref{thm:vir_abel} relies on applying the Induction Theorem. Without loss of generality, we begin by assuming that $\pi_1(M)=\Gamma$ is finitely generated and torsion-free nilpotent. Then the Induction Theorem implies that for any equivariant asymptotic cone $(Y,y,G)$ of $(\widetilde{M},\tilde{p},\Gamma)$, the limit orbit $G\cdot y$ must be a Euclidean subspace of $Y$. By the results of the second-named author \cite{Pan21}, it follows that $M$ has zero escape rate and $\pi_1(M)$ is virtually abelian.

For Theorem \ref{thm:finite}, we argue by contradiction and consider a suitable $\mathbb{Z}$-folding cover $\widehat{M}$ of $M$. Using the Induction Theorem (with $j=k=1$), we can estimate that the volume growth of $\widehat{M}$ is close to quadratic growth. Then, applying Anderson's results on positive Ricci curvature and at most cubic volume growth \cite{Anderson90}, we arrive at a desired contradiction.

As the name suggests, we prove the Induction Theorem by induction. The proof of the inductive step relies primarily on two key ingredients. The first is the equivariant asymptotic geometry of a $\mathbb{Z}$-folding covering map $\widehat{M}\to M$ whose escape rate is not $1/2$ (non-maximal). This topic was explored by the second-named author in \cite{Pan23}. In the inductive step, we must first verify the escape rate of the covering group $\mathbb{Z}$-action indeed is not $1/2$; this requires Abresch--Gromoll excess estimate \cite{AG90,GM14} and an argument by Sormani \cite{Sor00b}. Once this is confirmed, we apply the result from \cite{Pan23} to study the equivariant asymptotic geometry 
\begin{equation}\label{eq:induction_intro}
\begin{CD}
(r_i^{-1} \widehat{M}_{k-(j+1)},\hat{p}_{k-(j+1)},\mathbb{Z}) @>GH>> (Y_{k-(j+1)},y_{k-(j+1)},G) \\
	@VV\pi V @VV \pi V\\
	(r_i^{-1} \widehat{M}_{k-j},\hat{p}_{k-j}) @>GH>> (Y_{k-j},y_{k-j}).
\end{CD}
\end{equation}
In particular, by \cite[Proposition C(1)]{Pan23}, the limit orbit $G\cdot y_{k-(j+1)}$ is always homeomorphic to $\mathbb{R}$. Building on this and the inductive assumptions, we further investigate the asymptotic geometry of $\widehat{M}_{k-(j+1)}$. One particular goal is to analyze a limit renormalized measure on $Y_{k-(j+1)}$; the readers may refer to Section \ref{subsec:outline} for an outline of this step.

The analysis of equivariant asymptotic geometry leads to an asymptotic cone $Y_{k-(j+1)}=\mathbb{R}^j\times Z$ of $\widehat{M}_{k-(j+1)}$, where $Z$ satisfies specific symmetry and measure conditions; moreover, $Z$ is $\RCD(0,n-j)$ by Gigli's splitting theorem \cite{Gigli13,Gigli14}. The second key ingredient is a plane/halfplane rigidity result below for $\RCD(0,N)$ spaces, which is formulated with the purpose to understand such a space $Z$. Applying this rigidity result to $Z$, we conclude that $Z$ is isometric to either a Euclidean plane or a Euclidean halfplane.

\begin{thm}[Plane/Halfplane Rigidity]\label{thm:plane_halfplane_rigid}
Let $(Y,y,\dist,\mathfrak{m})$ be an $\RCD(0,N)$ space such that:\\
(1) $Y$ has a measure-preserving isometric $G$-action, where $G=\mathbb{R}\times K$ is a closed subgroup of $\mathrm{Isom}(Y)$, $K$ fixes $y$, and the quotient metric space $(Y/G,\bar{y})$ is isometric to a ray $([0,\infty),0)$,\\
(2) there is a constant $c>0$ such that $\mathfrak{m}(\Omega_r(v))=crv$ for all $r,v>0$ (see Definition \ref{def:omega_set} below for the subset $\Omega_r(v)\subseteq Y$). \\
Then $(Y,y,\dist,\mathfrak{m})$ is isomorphic (up to a constant) to a Euclidean plane $(\mathbb{R}^2,0)$ or a Euclidean halfplane $(\mathbb{R}\times [0,\infty),0)$ equipped with the Lebesgue measure.
\end{thm}

\begin{defn}\label{def:omega_set}
Let $r,v>0$. In the context of Theorem \ref{thm:plane_halfplane_rigid}(1), we define a subset $\Omega_r(v)\subset Y$ by
$$\Omega_r(v)\coloneqq\{ (w,h) \cdot \sigma(t)\ |\ w\in[-v,v], h\in K, t\in[0,r] \},$$
where $\sigma$ is a lift of the unit speed ray in $Y/G$ at $y$ and $(w,h)$ represents an element in $\mathbb{R}\times K=G$. We remark that the definition of $\Omega_r(v)$ is independent of the choice of $\sigma$ and the splitting $\mathbb{R}\times K=G$.
\end{defn}

The plane/halfplane rigidity (Theorem \ref{thm:plane_halfplane_rigid}) can be interpreted as a form of splitting in the presence of boundaries. The ideas behind the proof are closely aligned with those in Kasue's splitting theorem \cite[Theorem C]{KasueSplitting} and Croke--Kleiner \cite[Theorem 2]{CrokeKleiner}. In the nonsmooth setting, these results are unified as a type of local functional splitting and extended to $\RCD$ spaces using the disintegration techniques by Ketterer--Kitabeppu--Lakzian; see \cite{KKL23,Ketterer_23}. The corresponding global functional splitting theorem was first studied by Gigli \cite{Gigli13,Gigli14} and later formulated by Antonelli--Bru\`e--Semola \cite{ABS19}. From the local functional splitting theorem, it follows that to establish a splitting, it is sufficient to identify a harmonic function with a constant gradient. Distance functions are natural candidates for this purpose. Indeed, up to taking translation copies, the set $\Omega_r(v)$ is a union of level sets of the distance function to the orbit $Gy$. The assumption that the volume of $\Omega_r(v)$ grows linearly in $r$ suggests, heuristically, that the level sets of this distance function have zero mean curvature. The mean curvature of these level sets corresponds to the Laplacian of the distance function. 

We formalize this intuition by using the disintegration techniques, especially those from Cavalletti--Mondino \cite{CM_newformula}. The subtlety of the nonsmooth setting is that the splitting involves the intrinsic metric of the level sets of distance functions rather than the ambient metric. Given the abundance of symmetry in our case, we prove that the intrinsic metric of the level sets coincides with the ambient metric.
\smallskip

\noindent\textbf{Organization of the paper.} In Section \ref{sec:pre}, we recall the concepts of Ricci limit and $\RCD$ spaces, equivariant Gromov--Hausdorff convergence, as well as key results on manifolds with linear volume growth; we conclude this section with a reminder on escape rates and key results of the second-named author. Most of Section \ref{sec:asymgeo} is devoted to the proof of the induction Theorem \ref{thm:induction}. The section starts with a detailed sketch of the proof of the induction Theorem. It concludes with the proof of Theorem \ref{thm:vir_abel}, Theorem \ref{thm:finite}, and Theorem \ref{thm:rcd_abel} in Subsection \ref{subsec:pi1}, modulo the plane-half plane rigidity Theorem \ref{thm:plane_halfplane_rigid}, which we prove later in Section \ref{sec:rigid}. Finally, in Appendix \ref{sec:app}, we show that $\lim_{r\to\infty} \vol(B_r(p))/r$ always exists under the linear volume growth assumption both for manifolds with $\Ric\ge0$ and $\RCD(0,N)$ spaces (a fact that we use in the proof of Theorem \ref{thm:rcd_abel}).

%... organization the paper...

\emph{Acknowledgements.}

This material is based upon work supported by the National Science Foundation under Grant No. DMS-1928930, while two of the authors, J. Pan and X. Zhu, were in residence at the Simons Laufer Mathematical Sciences Institute (SLMath) in Berkeley, California, during the fall semester of 2024.

We are grateful to Guofang Wei for helpful discussions and her suggestion to work on Theorem \ref{thm:finite}. We would like to thank Jie Zhou for explaining the work \cite{ZZ} with details. We would like to thank Zhu Ye for pointing out Proposition \ref{prop:rcd_vol_ratio}. We would also like to thank our reviewers for their valuable suggestions and proposed alternatives for Section \ref{sec:rigid}.

J. Pan is partially supported by the National Science Foundation DMS-2304698 and Simons Foundation Travel Support for Mathematicians. D. Navarro and X. Zhu are partially supported by the AMS-Simons Travel Grant.

\section{Preliminaries}\label{sec:pre}
\subsection{Ricci limits and RCD spaces} 

\begin{defn}
Let $M^n$ be a complete manifold with $\Ric\ge0$. An \emph{asymptotic cone} (or a \emph{tangent cone at infinity}) of $M$ is a pointed Gromov--Hausdorff subsequential limit
    \begin{equation}\label{eq:asym_cone}
    (r_i^{-1}M,p) \overset{GH}\longrightarrow (X,x),
    \end{equation}
where $p\in M$ is fixed and $r_i\to\infty$ is a sequence of positive real numbers.
\end{defn}

\begin{rem}
    If $M$ is a complete manifold with $\Ric\ge0$, then, for every $\lambda>0$, the rescaled manifold $\lambda M$ also has $\Ric\geq0$. As a result of Gromov's precompactness theorem (see \cite[Theorem 5.3]{Gromov_98}), for any sequence $r_i\to\infty$ and $p\in M$, the blow-down sequence $\{(r_i^{-1}M,p)\}$ always admits a converging subsequence. Note, however, that the asymptotic cone may not be unique and depends on the subsequence and choice of $r_i\to\infty$.
\end{rem}

Asymptotic cones of manifolds with $\Ric\ge0$ are particular examples of Ricci limit spaces, which were extensively studied by Cheeger--Colding in their seminal papers \cite{CC96,CCI,CCII,CCIII}. Keeping track of the Riemannian measures associated with the approximating sequence is fundamental to retaining good stability properties in the limit space. A common way to formalize this idea (introduced in \cite{Fukaya_87}) is the measured Gromov--Hausdorff topology, endowing Riemannian manifolds with their normalized volume measure.

\begin{defn}\label{defn:renorm_meas}
    Let $M$ be a complete manifold with $\Ric\ge0$ and let $(X,x)$ be an asymptotic cone of $M$ coming from (\ref{eq:asym_cone}). For every $i\in\N$, we define the renormalized measure $\meas_i$ on $r_i^{-1}M$ as follows:
    \begin{equation*}
        \meas_i= \frac{\dvol}{\vol(B_{r_i}(p))}.
    \end{equation*}
    A \emph{limit renormalized measure} on $X$ is any measure $\nu$ on $X$ such that (passing to a subsequence if necessary) $\{(r_i^{-1}M,p,\meas_i)\}$ converges to $(X,x,\nu)$ in the (pointed) measure Gromov--Hausdorff topology.
\end{defn}

\begin{rem}If $(X,x)=\lim_{i\to\infty}(r_i^{-1}M,p)$ is an asymptotic cone of an open manifold $M$ with $\Ric\geq0$, then one can always pass to a measured Gromov--Hausdorff convergent subsequence to obtain a limit renormalized measure $\nu$, which is a Radon measure on $X$ (see \cite[Section 1]{CCI}). Moreover, if $\nu$ is a limit renormalized measure, then, for any converging sequence $q_i\in M\to q\in X$, and $R>0$, we have $\nu(B_R(q))=\lim_{i\to\infty}\meas_i(B_R(q_i))$. However, limit renormalized measures on $X$ are not unique in general.
\end{rem}

Let us introduce isomorphisms of metric measure spaces (m.m.s. for short) up to a constant.

\begin{defn}\label{def:isomorphic_up_to_a_constant}We say two m.m.s. (metric measure spaces) $(X,\dist_X,\meas_X)$ and $(Y,\dist_Y,\meas_Y)$ are \emph{isomorphic up to a constant}, if there is an isometry $\phi\colon(X,\dist_X)\to(Y,\dist_Y)$ and a constant $\lambda>0$ such that $\phi_*\meas_X=\lambda\meas_Y$.    
\end{defn}

A related approach to studying Ricci limit spaces is introducing a new definition of Ricci curvature lower bounds at the more general level of (possibly non-smooth) metric measure spaces. Inspired notably by \cite{Cordero-Erausquin_01}, Lott--Villani \cite{Lott-Villani_09} and Sturm \cite{Sturm_I_06,Sturm_II_06} were the first to propose a definition of Ricci curvature lower bounds using the theory of Optimal Transport in their pioneering papers. Their combined work led to $\mathrm{CD}(K,N)$ spaces. In our paper, we will focus on $\Ric\ge0$; thus, it will be sufficient to consider the case $K=0$.

\begin{defn}%[$\mathrm{CD}(0,N)$ spaces]
    A m.m.s. $(X,\dist,\meas)$ satisfies the \emph{$\mathrm{CD}(0,N)$ condition} if, given any pair of probability measures $\mu_0,\mu_1$ that are absolutely continuous w.r.t. $\meas$, there exists a $W_2$-geodesic $\{\mu_t\}_{0\le t\le1}$ from $\mu_0$ to $\mu_1$ such that, for every $N'\ge N$, we have the following property:
        \begin{equation*}
        \forall 0\le t\le 1, \quad \mathcal{S}_N'(\mu_t\mid\meas)\le t\mathcal{S}_N'(\mu_1\mid\meas)+(1-t)\mathcal{S}_N'(\mu_0\mid\meas),
    \end{equation*}
    where $\mathcal{S}_N'(\cdot\mid\meas)$ denotes the R\'{e}nyi entropy with parameter $N'$ associated with $\meas$.
\end{defn}

We refer the reader to \cite{Villani_09} for an introduction to Optimal Transport, particularly for definitions of Wasserstein geodesics and entropy functionals.

While Ricci limit spaces are examples of $\mathrm{CD}$ spaces, it is also the case of Finsler manifolds with an adequate curvature lower bound (see \cite{Ohta_09}). %Since the original purpose of $\mathrm{CD}$ was to study Ricci limit spaces, and 
Since Finsler manifolds arise as Ricci limit spaces only when they are Riemannian, the $\mathrm{CD}$ had to be strengthened. To stay as close as possible to the Riemannian situation, Ambrosio--Gigli--Savar\'{e} introduced $\RCD$ spaces in \cite{Ambrosio-Gigli-Savare_14} ruling out non-Riemannian Finsler spaces. We follow the definition of $\RCD$ spaces from Gigli's work \cite{Gigli15}, which requires the Sobolev space $H^{1,2}$ to be a Hilbert space. This property is referred to as \emph{infinitesimal Hilbertianity}.

\begin{defn}\cite{Gigli15}
  An $\RCD(0,N)$ space is an infinitesimally Hilbertian $\mathrm{CD}(0,N)$ space.
\end{defn}

The $\RCD(0,N)$ condition is stable under (pointed) measured Gromov--Hausdorff convergence (see \cite{GMS15}). In particular, for an open manifold $M$ with $\Ric\ge 0$, any asymptotic cone of $M$ with any limit renormalized measure is $\RCD(0,N)$. 

When dealing with manifolds with nonnegative Ricci curvature, a fundamental result is the Cheeger--Gromoll splitting theorem \cite{CG_split}. The splitting theorem was extended to Ricci limit spaces by Cheeger--Colding \cite{CC96}, and finally to $\RCD(0,N)$ spaces by Gigli \cite{Gigli13,Gigli14}.

\begin{thm}\label{thm:meas_split}\cite{Gigli13,Gigli14}
If $(X,\dist,\meas)$ is an $\RCD(0,N)$ space (where $N<\infty$) that contains a line, then there exists a m.m.s. $(X',\dist',\meas')$ such that $(X,\dist,\meas)$ is isomorphic to $(X',\dist',\meas')\otimes (\R,d_E,\mathcal{L}^1)$, where:\\
$\bullet$ $(X',\dist',\meas')$ is an $\RCD(0,N-1)$ space when $N\ge 2$,\\
$\bullet$ $(X',\dist',\meas')$ is a point when $N<2$,\\
and $\R$ is equipped with Euclidean distance $d_E$ and Lebesgue measure $\mathcal{L}^1$.
\end{thm}

In addition to the splitting theorem, we will rely on the following fundamental geometric properties of $\RCD(0,N)$ spaces, where $N<\infty$:\\
$\bullet$ $\RCD(0,N)$ spaces are non-branching (see \cite{CoNa12,Deng20}).  \\
$\bullet$ The group of measure preserving isometries of an $\RCD(K,N)$ space is a Lie group (see \cite{CoNa12} for the initial proof on Ricci limit spaces and \cite{Sosa18,Guijarro_19} for the $\RCD$ case).\\
$\bullet$ The quotient of an $\RCD(K,N)$ space by a compact subgroup of measure preserving isometries, endowed with the pushforward measure, is still an $\RCD(K,N)$ space (see \cite{Galaz-Garcia_18}).

The local version of functional splitting theorem by Ketterer (see \cite[Theorem 4.11]{Ketterer_23}) will be crucial in our proof of Theorem \ref{thm:plane_halfplane_rigid}. Let us first introduce the following notation.

\begin{notn}
Let $(X,\dist,\meas)$ be an m.m.s. and $\Omega\subset X$ an open subset. We denote $(\tilde{\Omega},\tilde{\dist}_{\Omega})$ the metric completion of $\Omega$ equipped with its extended intrinsic distance $\dist_{\Omega}$ (this extended distance takes the value $\infty$ between different connected components of $\Omega$). Since $\Omega$ is naturally a subset of $\tilde{\Omega}$, the measure $\meas_{\Omega}\coloneqq \meas|_{\Omega}$ is a well defined measure on $\tilde{\Omega}$.
\end{notn}

\begin{thm}\label{thm:ketterer}\cite[Theorem 4.11]{Ketterer_23}
    Let $(X,\dist,\meas)$ be an $\RCD(0,N)$ space (where $N<\infty$), let $\Omega\subset X$ be an open subset of $X$, and assume that $u\colon\Omega\to\R$ is a function such that $\Omega=u^{-1}(0,D)$, for some $D>0$. If $\lvert \nabla u\rvert = 1$ $\meas$-a.e. on $\Omega$ and if $\mathbf{\Delta}_{\Omega}u=0$, then there exists $(Y,\dist_Y,\meas_Y)$, a disjoint union of connected m.m.s., such that $(\tilde{\Omega},\tilde{\dist}_{\Omega},\meas_{\Omega})$ is isomorphic to $(Y,\dist_Y,\meas_Y)\otimes [0,D]$. Moreover, any connected component of $(Y,\dist_Y,\meas_Y)$ is $\RCD(0,N-1)$ if $N\ge 2$, or a point if $N<2$.
\end{thm}

\subsection{Equivariant Gromov--Hausdorff convergence}

We recall some basic facts about (pointed) equivariant Gromov--Hausdorff convergence from \cite{Fukaya86,FY92}.

Throughout the paper, we often denote $(X,x,G)$ a pointed complete length metric space $(X,x)$ with a closed subgroup $G$ of the isometry group $\mathrm{Isom}(X)$. For $R\ge 0$, we denote
$$G(R)=\{g\in G\ |\ d(gx,x)\le R \}.$$

\begin{defn}\label{def:eqgh}\cite{Fukaya86,FY92} 
Given $\epsilon>0$, two spaces $(X,x,G)$ and $(Y,y,H)$ are $\epsilon$-close in the (pointed) \emph{equivariant Gromov--Hausdorff} topology, written as
$$d_{GH}((X,x,G),(Y,y,H))\le \epsilon,$$
if there are $\epsilon$-approximation maps $(f,\psi,\phi)$, that is,
$$f:B_{1/\epsilon}(x)\to Y,\quad  \psi: G(1/\epsilon) \to H(1/\epsilon), \quad \phi:H(1/\epsilon) \to  G(1/\epsilon) $$
with the following properties:\\
(1) $f(x)=y$;\\
(2) the $\epsilon$-neighborhood of $f(B_{1/\epsilon}(x))$ contains $B_{1/\epsilon}(y)$;\\
(3) $|d_Y(f(x_1),f(x_2))-d_X(x_1,x_2)|\le\epsilon$ for all $x_1,x_2\in B_{1/\epsilon}(x)$;\\
(4) if $g\in G(1/\epsilon)$ and $z,gz\in B_{1/\epsilon}(x)$, then $d_Y(f(gz),\psi(g)f(z))\le\epsilon$;\\
(5) if $h\in H(1/\epsilon)$ and $z,\phi(h)z\in B_{1/\epsilon}(x)$, then $d_Y(f(\phi(h)z),h f(z))\le\epsilon$.

We say that a sequence $(X_i,x_i,G_i)$ converges to $(Y,y,H)$ in the (pointed) \emph{equivariant Gromov--Hausdorff} topology, written as
\begin{equation}\label{eq:eqGH}
(X_i,x_i,G_i)\overset{GH}\longrightarrow (Y,y,H),
\end{equation}
if $d_{GH}((X_i,x_i,G_i),(Y,y,H))\to 0$.
\end{defn}

\begin{thm}\cite{Fukaya86,FY92} 
	If $(X_i,x_i)\overset{GH}\longrightarrow (Y,y)$
	is a Gromov--Hausdorff convergent sequence and, for each $i$, $G_i$ is a closed subgroup of $\mathrm{Isom}(X_i)$, then:\\
	(1) passing to a subsequence if necessary, we have an equivariant Gromov--Hausdorff convergence sequence
   $$(X_i,x_i,G_i)\overset{GH}\longrightarrow (Y,y,H),$$
	where $H$ is closed subgroup of $\mathrm{Isom}(Y)$,\\
	(2) the sequence of quotient spaces converges
	$$(X_i/G_i,\bar{x}_i)\overset{GH}\longrightarrow (Y/H,\bar{y}).$$
\end{thm}

For a convergent sequence (\ref{eq:eqGH}) by definition we have a sequence $\epsilon_i\to 0$ and a sequence of $\epsilon_i$-approximation maps $(f_i,\psi_i,\phi_i)$ with the conditions in Definition \ref{def:eqgh}. If a sequence $g_i\in G_i$ and $h\in H$ satisfies
$$d_Y(\psi_i(g_i)z,hz)\to 0$$
for all $z\in Y$, then we say that $h$ is the limit of $g_i$ associated to (\ref{eq:eqGH}). For any sequence $g_i\in G_i$ with uniformly bounded displacement, that is, $d(g_ix_i,x_i)\le D< \infty$, by the method of \cite[Proposition 3.6]{FY92} we can always find a subsequence of $g_i$ converging to some $h\in H$  associated to (\ref{eq:eqGH}).

In Section \ref{sec:asymgeo}, we often take a sequence of closed subsets $S_i\subseteq G_i$ and consider its convergence to some limit closed subset $S \subseteq H$. This method was used in \cite{Pan24,Pan23}. 

\begin{defn}\label{def_conv_symsubset}
  Let (\ref{eq:eqGH}) be an equivariant Gromov--Hausdorff convergent sequence and let $S_i$ be a sequence of closed subsets in $G_i$. We write
   \begin{equation}\label{eq:eqGH_subset}
   (X_i,x_i,S_i)\overset{GH}\longrightarrow (Y,y,S),
   \end{equation}
	for a closed subset $S$ of $H$, if\\
	(1) for every $h\in S$, there exists a sequence $g_i\in S_i$ converging to $h$,\\
	(2) every convergent subsequence $g_i \in S_i$ has its limit $h$ in $S$.
\end{defn}

It follows directly from the proof of \cite[Proposition 3.6]{FY92} that a corresponding precompactness theorem holds for the convergence of subsets $S_i\subset G_i$. More specifically, for any convergent sequence (\ref{eq:eqGH}) and any sequence of closed subsets $S_i\subset G_i$, we can always find a subsequence such that (\ref{eq:eqGH_subset}) holds. In this paper, we mainly consider the convergence of symmetric subsets. Then its limit $S\subset G$ is symmetric as well.

\begin{defn}
   Let $G$ be a group and $S$ be a subset of $G$. We say that $S$ is \emph{symmetric}, if $e\in S$ and $s^{-1}\in S$ whenever $s\in S$.
\end{defn}

To study the fundamental group of an open manifold $M$ with $\Ric\ge 0$, it is natural to consider a suitable Riemannian covering space $\widehat{M}$ of $M$ with covering group $\Gamma$. We blow down the space to understand the asymptotic geometry of this $\Gamma$-action. Given any $r_i\to\infty$, we can pass to a subsequence and obtain convergence
\begin{equation*}
\begin{CD}
(r_i^{-1} \widehat{M},\hat{p},\Gamma) @>GH>> (Y,y,G) \\
	@VV\pi V @VV \pi V\\
	(r_i^{-1} M,p) @>GH>> (X,x)=(Y/G,\bar{y}).
\end{CD}
\end{equation*}
We call the limit space $(Y,y,G)$ an \emph{equivariant asymptotic cone} of $(\widetilde{M},\Gamma)$. In general, the limit space $(Y,y,G)$ is not unique and depends on the scaling $r_i\to\infty$. Because the isometry group $\mathrm{Isom}(Y)$ is a Lie group \cite{CoNa12} and $G$ is a closed subgroup of $\mathrm{Isom}(Y)$, $G$ is also a Lie group. Lastly, we note that, by construction, the $G$-action preserves any limit renormalized measure $\nu$ on $Y$.

\subsection{Manifolds with linear volume growth}

We recall the relevant results on open manifolds with nonnegative Ricci curvature and linear volume growth. Theorem \ref{thm:sormani} below follows directly from Sormani's works \cite{Sor98,Sor00a}. For a detailed proof of this exact statement, see \cite[Theorem 2.4]{Zhu23}.

\begin{thm}\cite{Sor98,Sor00a}\label{thm:sormani}
   Let $M$ be an open manifold with $\Ric\ge0$ and linear volume growth. Then one of the following holds:\\
   (1) Every asymptotic cone of $M$ is isometric to a line; this happens if and only if $M$ splits isometrically as $\mathbb{R}\times N$, where $N$ is compact.\\
   (2) Every asymptotic cone of $M$ is isometric to a ray $([0,\infty),0)$.
\end{thm}

In \cite{Sor00b}, Sormani proved that manifolds with $\Ric\ge0$ and small linear diameter growth have finitely generated fundamental groups. Her result applies to manifolds with linear volume growth since they either split isometrically as $\mathbb{R}\times N$ with some compact factor $N$ or have sublinear diameter growth. In fact, Sormani proved a much stronger result: $\pi_1(M)$ is finitely generated if every asymptotic cone of $M$ is sufficiently close to a polar space; see \cite[Theorem 11]{Sor00b}.

\begin{thm}\cite{Sor00b}\label{thm:finite_gen}
   If $M$ is an open manifold with $\Ric\ge0$ and linear volume growth, then $\pi_1(M)$ is finitely generated.
\end{thm}

When $M$ has nonnegative Ricci curvature and $\pi_1(M)$ is finitely generated, Milnor proved that $\pi_1(M)$ has at most polynomial growth \cite{Milnor68}. Combining this with Gromov's result that any finitely generated group of polynomial growth is virtually nilpotent \cite{Gromov81}, we obtain the following theorem.

\begin{thm}\label{thm:vir_nil}\cite{Milnor68,Gromov81}
   If $M$ is a complete manifold with $\Ric\ge0$, then every finitely generated subgroup of $\pi_1(M)$ has a nilpotent subgroup of finite index.
\end{thm}

\begin{rem}\label{rem:vir_nil_torsionfree}
We can further require that the nilpotent subgroup in Theorem \ref{thm:vir_nil} is torsion-free. That is because every finitely generated nilpotent group contains a torsion-free nilpotent subgroup of finite index \cite{KM_book}.
\end{rem}

Recently, Zhou--Zhu pointed out in their work \cite{ZZ} that for manifolds with nonnegative Ricci curvature and linear volume growth, the limit of $\vol B_r(p)/r$ as $r\to\infty$ always exists and is always positive.

\begin{thm}\label{thm:volume_ratio_limit}\cite[Remark 2.1]{ZZ}
   If $M$ is an open manifold with $\Ric\ge0$ and linear volume growth, then the limit $\lim_{r\to\infty} \vol (B_r(p))/r$ exists.
\end{thm}

For readers' convenience, we include the proof of Theorem \ref{thm:volume_ratio_limit} in the appendix (see Proposition \ref{prop:linear2}), following the method in \cite{ZZ}. The corollary below results from Theorem \ref{thm:sormani} and a direct calculation of limit renormalized measure by Theorem \ref{thm:volume_ratio_limit}.
 
\begin{cor}\label{cor:induction_base}
   Let $M$ be an open manifold with $\Ric\ge0$ and linear volume growth. Then $M$ has a unique asymptotic cone, as a pointed metric-measure space, isomorphic (up to a constant) to either a line $(\mathbb{R},0)$ or a ray $([0,\infty),0)$ with the Lebesgue measure (see Definition \ref{def:isomorphic_up_to_a_constant}).
\end{cor}

\begin{proof}
    The isometric type is discussed in Theorem \ref{thm:sormani}. We show in each case there, the limit measure $\nu$ is a constant multiple of the Lebesgue measure on $\R$.

    If $M$ does not split, then the unique asymptotic cone is the ray $([0,\infty),0)$. For any rescaling sequence $r_i\to \infty$ and any radius $r\ge 0$, by the definition of the renormalized limit measure and Theorem \ref{thm:volume_ratio_limit} we have that 
    \[
    \nu([0,r))=\lim_{i\to \infty}\frac{\vol_g(B_{rr_i}(p))}{\vol_g(B_{r_i}(p))}=r\lim_{i\to \infty}\frac{\vol_g(B_{rr_i}(p))}{rr_i}\frac{r_i}{\vol_g(B_{r_i}(p))}=r.
    \]
    By varying $r$ we have that for any $b>a\ge 0$, $\nu([a,b))=b-a$. It follows that $\nu$ is the Lebesgue measure $\mathcal{L}^1$. 

    If $(M,g)$ splits as $\R\times N$, we still denote the limit measure by $\nu$. The unique asymptotic cone is $(\R,0)$. For any radius $r\ge0$, the same rescaling argument yields that $\nu((-r,r))=r$. Note that $M=\R\times N$ is translation invariant along $\R$, to compute $\nu ((a-r,a+r))$ for any $a\in \R$, we can apply the translation, which is a measure preserving isometry on $(M,r_i^{-2}g)$, along $\R$ that takes $\{ar_i\}\times N$ to $\{0\}\times N$. This yields that $\nu ((a-r,a+r))=\nu((-r,r)) =r$ for any $a\in \R$. It then follows that $\nu=\frac{1}{2}\mathcal{L}^1$.
\end{proof}

\subsection{Escape rate}

The notion of escape rate was introduced and studied extensively by the second-named author \cite{Pan21,Pan22,Pan24,Pan23}. In any open manifold with positive Ricci curvature, if $\pi_1(M,p)$ is an infinite group, then it follows from the Cheeger--Gromoll splitting theorem \cite{CG_split} that the minimal representing loops of $\pi_1(M,p)$ must escape from any bounded subsets of $M$ as we exhaust the elements in $\pi_1(M,p)$. Motivated by this observation, the second named author introduced the notion of escape rate \cite{Pan21}, which measures how fast this escape phenomenon is by comparing the size of minimal representing loops to their length. 

\begin{defn}\label{def:escape_rate_pi1}\cite{Pan21}
   Let $(M,p)$ be a complete manifold with an infinite fundamental group $\pi_1(M,p)$. For each element $\gamma \in \pi_1(M,p)$, let $c_\gamma$ be a geodesic loop of minimal length at $p$ representing $\gamma$. We define the escape rate of $(M,p)$ by
   $$E(M,p)=\limsup_{|\gamma|\to\infty} \dfrac{\mathrm{size}(c_\gamma)}{|\gamma|},$$
   where $|\gamma|=\mathrm{length}(c_\gamma)$ and $\mathrm{size}(c_\gamma)=\inf\{R>0 | c_\gamma \subseteq \overline{B_R}(p) \}$. When choosing $c_\gamma$, if there are multiple geodesic loops representing $\gamma\in \pi_1(M,p)$ of minimal length, we choose the one with the smallest size.
\end{defn}

In \cite[Definition 2.17]{Pan23}, the second-named author further defined the escape rate of an isometric $G$-action on a length metric space $(X,x)$. For the purpose of this paper, we will consider the covering group action on a covering space and its escape rate. 

\begin{defn}\label{def:escape_rate}\cite[Definition 2.17]{Pan23}
   Let $(\widehat{M},\hat{p})\to (M,p)$ be a Riemannian covering map with an infinite covering group $\Gamma$. For each $\gamma\in \Gamma$, let $\sigma_\gamma$ be a minimal geodesic from $\hat{p}$ to $\gamma\cdot \hat{p}$. We define the escape rate of $(\widehat{M},\hat{p},\Gamma)$ by $$E(\widehat{M},\hat{p},\Gamma)=\limsup_{|\gamma|\to\infty} \dfrac{\mathrm{size}(\sigma_\gamma)}{|\gamma|},$$
   where $|\gamma|=d(\hat{p},\gamma \hat{p})=\mathrm{length}(\sigma_\gamma)$ and $\mathrm{size}(\sigma_\gamma)=\inf\{R>0 | \sigma_\gamma \subseteq \overline{B_R}(\Gamma\cdot \hat{p}) \}$. When choosing $\sigma_\gamma$, if there are multiple minimal geodesics from $\hat{p}$ to $\gamma \cdot \hat{p}$, we choose the one with the smallest size.
\end{defn} 

\begin{rems}
    (1) By path-lifting, it is clear that $E(M,p)$ in Definition \ref{def:escape_rate_pi1} is the same as the escape rate of $\pi_1(M,p)$-action on the universal cover $(\widetilde{M},\tilde{p})$ in Definition \ref{def:escape_rate}, that is,
$$E(M,p)=E(\widetilde{M},\tilde{p},\pi_1(M,p)).$$
   (2) If $\Gamma$ is a finite group, then, as a convention, we set $E(\widehat{M},\hat{p},\Gamma)=0$.\\
   (3) In Definition \ref{def:escape_rate}, since $\mathrm{size}(\sigma_\gamma)\le |\gamma|/2$ for all $\gamma\in \Gamma$, it always holds that $E(\widehat{M},\hat{p},\Gamma)\le 1/2$.
\end{rems} 

The second-named author studied equivariant asymptotic geometry with $\Ric\ge 0$ and escape rate strictly less than $1/2$ in \cite{Pan23}. In particular, \cite[Proposition C(1)]{Pan23} shows that, for a finitely generated nilpotent covering group $\mathcal{N}$-action and a $\mathbb{Z}$-subgroup in the last nontrivial subgroup of the central series of $\mathcal{N}$, every asymptotic orbit $Gy$ of this $\mathbb{Z}$-action must be homeomorphic to $\mathbb{R}$, provided that the escape rate of $\mathcal{N}$-action is not $1/2$. We obtain the following in the particular case where $\mathcal{N}=\mathbb{Z}$.

\begin{prop}\cite[Proposition C(1)]{Pan23}\label{prop:orbit_R}
   Let $(\widehat{M},\hat{p})\to (M,p)$ be a Riemannian covering map with covering group $\Gamma\simeq\mathbb{Z}$. Suppose that $\Ric_M\ge 0$ and $E(\widehat{M},\hat{p},\Gamma)\not=1/2$, then in every equivariant asymptotic cone $(Y,y,G)$ of $(\widehat{M},\hat{p},\Gamma)$, the orbit $Gy$ is homeomorphic to $\mathbb{R}$.
\end{prop}

\section{Equivariant asymptotic geometry of successive covering spaces}\label{sec:asymgeo}
In this section, we prove Theorems \ref{thm:vir_abel} and \ref{thm:finite} by assuming Theorem \ref{thm:plane_halfplane_rigid}, the plane/halfplane rigidity. We will prove Theorem \ref{thm:plane_halfplane_rigid} in Section \ref{sec:rigid}.

This section is organized as follows. 

In Section \ref{subsec:nonmax_escape}, we prove that if $M$ is polar at infinity, then the escape rate  $E(\widehat{M},\hat{p},\Gamma)$ is non-maximal, where $(\widehat{M},\hat{p})$ is a covering space of $(M,p)$ with covering group $\Gamma$ (Proposition \ref{prop:non_max_escape}); the proof is based on the Abresch-Gromoll excess estimate \cite{AG90,GM14} and an argument by Sormani \cite{Sor00b}. 

Building on this property of non-maximal escape rate and Proposition \ref{prop:orbit_R}, in Section \ref{subsec:orbit_R} we study the equivariant asymptotic geometry of a $\mathbb{Z}$-orbit in a covering space
\begin{equation}
(r_i^{-1}\widehat{M},\hat{p},\Gamma=\mathbb{Z})\overset{GH}\longrightarrow (Y,y,G),
\end{equation}
where $G=\mathbb{R}\times K$ with $K$-action fixing $y$.

In Section \ref{subsec:cylinder}, we construct appropriate symmetric subsets $S_i\subset \Gamma=\mathbb{Z}$ such that $S_i$ converges to a cylinder subset $[-v,v]\times K$ in $\mathbb{R}\times K =G$ associated with the above blow-down sequence. We remark that both Sections \ref{subsec:orbit_R} and \ref{subsec:cylinder} only depend on the non-maximal escape rate rather than the stronger linear volume growth condition.

In Section \ref{subsec:outline}, we state the induction theorem (Theorem \ref{thm:induction_plus}) and give an outline of its proof. Besides the plane/halfplane rigidity, whose proof is postponed to Section \ref{sec:rigid}, another key technical part in the proof of the induction theorem is computing a limit renormalized measure on the asymptotic cone of the covering space. We explain in Section \ref{subsec:outline} that we can approximate cube regions in the limit space using approximate domains, whose volumes can be estimated.

In Section \ref{subsec:cube}, we construct these appropriate domains and verify their convergence by using the results from Section \ref{subsec:cylinder} and the inductive assumptions.

With the above preparations, we complete the proof of the Induction Theorem in Section \ref{subsec:induction}, modulo the plane/halfplane rigidity.

In Section \ref{subsec:pi1}, we apply the Induction Theorem to prove Theorems \ref{thm:vir_abel} and \ref{thm:finite}.

\subsection{Polar asymptotic cones and non-maximal escape rate}\label{subsec:nonmax_escape}

In this subsection, we consider a Riemannian covering map $(\widehat{M},\hat{p})\to (M,p)$ with covering group $\Gamma$. We show that if $M$ has nonnegative Ricci curvature and is polar at infinity, then the escape rate $E(\widehat{M},\hat{p},\Gamma)$ is not $1/2$ (Proposition \ref{prop:non_max_escape}). This result will serve as a starting point in the inductive step of the Induction Theorem so that Proposition \ref{prop:orbit_R} applies. The proof of Proposition \ref{prop:non_max_escape} is an adaptation of an argument by Sormani \cite{Sor00b} to the asymptotic cones.

\begin{defn}\label{def:polar}
   Let $M$ be a complete manifold with $\mathrm{Ric}\ge 0$. We say that $M$ is polar at infinity if every asymptotic cone $(X,x)$ of $M$ is polar at $x$, that is, for every point $z\in X-\{x\}$, there is a ray emanating from $x$ through $z$.
\end{defn}

\begin{prop}\label{prop:non_max_escape}
   Let $(M,p)$ be a complete manifold with $\Ric\ge 0$ and let $(\widehat{M},\hat{p})$ be a covering space of $(M,p)$ with covering group $\Gamma$. Suppose that $M$ is polar at infinity, then $E(\widehat{M},\hat{p},\Gamma)<1/2$.
\end{prop}

\begin{proof}
    We argue by contradiction and suppose that $E(\widehat{M},\hat{p},\Gamma)=1/2$, then there is a sequence of elements $\gamma_i\in \Gamma$ and a sequence of minimal geodesics $\sigma_i$ from $\hat{p}$ to $\gamma_i\hat{p}$ such that
   $$\frac{\mathrm{size}(\sigma_i)}{\mathrm{length}(\sigma_i)}\to \frac{1}{2},\quad r_i:=\mathrm{length}(\sigma_i)=d(\hat{p},\gamma_i\hat{p})\to\infty.$$
   For each $i$, there exist points $q_i\in \sigma_i$ and $z_i\in \Gamma\cdot \hat{p}$ such that 
   $$\mathrm{size}(\sigma_i)=d(q_i,z_i)=d(q_i,\Gamma\cdot \hat{p}).$$ After passing to a convergent subsequence, we consider the blow-down sequence
   $$(r_i^{-1}\widehat{M},\hat{p},\Gamma,q_i,z_i,\gamma_i,\sigma_i)\overset{GH}\longrightarrow (Y,y,G,q,z,g,\sigma),$$
   where $g\in G$ with $d(gy,y)=1$ and $\sigma$ is a minimal geodesic from $y$ to $gy$. By construction, we have $q\in \sigma$, $z\in Gy$, and
   \begin{equation}\label{eq:midpoint}
   d_Y(q,Gy)=d_Y(q,z)=\frac{1}{2} \mathrm{length}(\sigma) = \frac{1}{2} d_Y(y,gy)=\dfrac{1}{2}.
   \end{equation}
   This implies that $q$ is the midpoint of $\sigma$: in fact, if either $d_Y(q,y)$ or $d_Y(q,gy)$ is strictly less than $1/2$, then we would have $d_Y(q,Gy)<1/2$, a contradiction.

   On $M$, we have the corresponding convergence 
   $$(r_i^{-1}M,p)\overset{GH}\longrightarrow (X,x)=(Y/G,\bar{y}).$$ 
   Let $\bar{q}\in X$ be the projection of $q\in Y$. Because $X$ is polar at $x$, there is a unit speed ray ${c}$ starting at ${c}(0)=x$ and passing through $\bar{q}$. Note that by (\ref{eq:midpoint}) we have ${c}(1/2)=\bar{q}$. Let $w=c(S+1/2) \in X$, where $S=S(n)\in (0,1)$ is a small constant to be determined later. We lift $c|_{[1/2,S+1/2]}$ to a minimal geodesic $\hat{c}$ in $Y$ starting at $q$. Let $\hat{w}=\hat{c}(S+1/2)$ be the endpoint of $\hat{c}$. We set $q_-=y$ and $q_+=gy$. Then 
   $$d_Y(q_\pm,\hat{w})\ge d_X(x,w)=S+1/2.$$
   Observe that 
   $$d_Y(\sigma,\hat{w})=d_Y(q,\hat{w})=S;$$
   in fact, if a point $q'\in \sigma-\{q\}$ satisfies 
   $$d_Y(q',\hat{w})<d_Y(q,\hat{w})=S,$$
   then 
   $$d_Y(Gy,\hat{w})\le d_Y(Gy,q')+d_Y(q',\hat{w})<\dfrac{1}{2}+S,$$
   which is a contradiction to $d_X(x,w)=S+1/2$. Since $d_Y(\sigma,\hat{w})<d_Y(q_{\pm},\hat{w})$, we can apply Abresch--Gromoll excess estimate in RCD$(0,N)$ spaces \cite{AG90,GM14} to $\hat{w}$ and $\sigma$. Then there is a constant $C(n)$ such that 
   $$2S\le d_Y(q_+,\hat{w})+d_Y(q_-,\hat{w})-d_Y(q_-,q_+)\le C(n) \cdot l^{\frac{n}{n-1}} =C(n)\cdot S ^{\frac{n}{n-1}},$$
   which implies 
   $$S\ge \dfrac{2^{n-1}}{C(n)^{n-1}}.$$
   Choosing $S=S(n)<(2/C(n))^{n-1}$ in the beginning results in a desired contradiction.
\end{proof}

\subsection{Consequences of asymptotic \texorpdfstring{$\R$}{R}-orbits}\label{subsec:orbit_R}

In this subsection, we consider a Riemannian covering map $(\widehat{N},\hat{q})\to (N,q)$ with covering group $\Gamma\simeq \mathbb{Z}$ and escape rate $E(\widehat{N},\hat{q},\Gamma)\not= 1/2$, where $(N,q)$ is an open manifold with $\Ric\ge 0$. Let $(Y,y,G)$ be an equivariant asymptotic cone of $(\widehat{N},\hat{q},\Gamma)$:
$$(r_i^{-1} \widehat{N},\hat{q},\Gamma)\overset{GH}\longrightarrow (Y,y,G).$$
Proposition \ref{prop:orbit_R} assures that the orbit $Gy$ is homeomorphic to $\mathbb{R}$. We study the immediate consequences of these asymptotic $\mathbb{R}$-orbits.

Given that $G$ is a closed and abelian Lie subgroup of $\mathrm{Isom}(Y)$ and its orbit $Gy$ is homeomorphic to $\mathbb{R}$, we can always write $G=\mathbb{R}\times K$, where $K$ is the isotropy subgroup of $G$ at $y$. If the compact Lie group $K$ has positive dimension, then $\mathbb{R}$-subgroups of $G$, and thus the splittings $G=\mathbb{R}\times K$, are not unique. However, a different choice of the $\mathbb{R}$-subgroup does not change the corresponding orbit points: in fact, any $\mathbb{R}$-subgroup in $\mathbb{R}\times K$ has the form $t\mapsto (t,e(t))$, where $e(t)$ is either a one-parameter subgroup in $K$ or is a constant map to the identity element $e$ of $K$, then it holds that
$$(t,e(t))\cdot y = (t,e) \cdot y.$$
We remark that since $G=\mathbb{R}\times K$ is a closed subgroup of $\mathrm{Isom}(Y)$, the orbit $Gy$ is closed in $Y$. Consequently, $d((t,e)\cdot y,y)\to \infty$ as $t\to\infty$.

\begin{defn}\label{def:one_para_orbit}
   For any orbit point $z\in Gy-\{y\}$, we can write $z=(v,e)\cdot y$, where $v\in \mathbb{R}-\{0\}$ and $e$ is the identity element in $K$. For $t\in \mathbb{R}$, we denote
   $tz := (tv,e)\cdot y.$
\end{defn}

One consequence of the limit $\mathbb{R}$-orbits is a quantitative distance control on these orbits.

\begin{lem}\cite[Lemma 5.13]{Pan23}\label{lem:boomer_control}
   Under the assumptions of Proposition \ref{prop:orbit_R}, we denote $\gamma$ a generator of $\Gamma\simeq \mathbb{Z}$. Then there is a constant $C=C(\widehat{N},\gamma)$, depending on the action by $\gamma$ on $\widehat{N}$, such that for any equivariant asymptotic cone $(Y,y,G)$ of $(\widehat{N},\hat{q},\Gamma)$ and any orbit point $z\in Gy-\{y\}$, we have
	$$d(tz,y)\le C\cdot d(z,y)$$
	for all $t\in [0,1]$.
\end{lem}

One may view Lemma \ref{lem:boomer_control} as a quantitative version of Proposition \ref{prop:orbit_R}. The proof is by contradiction. Roughly speaking, if there were no such uniform distance control for $d(tz,y)$, where $t\in[0,1]$, then one can apply a contradicting argument and find a suitable equivariant asymptotic cone $(Y',y',G')$ of $(\widehat{N},\hat{p},\Gamma)$ such that the orbit $G'y'$ is not homeomorphic to $\mathbb{R}$; this is a contradiction to Proposition \ref{prop:orbit_R}. See \cite[Lemma 5.13]{Pan23} for the detailed proof.

\begin{cor}\label{cor:boomer_control}
   Under the assumptions of Proposition \ref{prop:orbit_R}, for any equivariant asymptotic cone $(Y,y,G)$ of $(\widehat{N},\hat{q},\Gamma)$ and any orbit point $z\in Gy-\{y\}$, we have
   $$d(tz,y)\ge C^{-1} d(z,y)$$
   for all $t\ge 1$, where $C$ is the constant in Lemma \ref{lem:boomer_control}.
\end{cor}

\begin{proof}
   Suppose that there exist an equivariant asymptotic cone $(Y,y,G)$, an orbit point $z\in Gy-\{y\}$, and $t_0\ge 1$ such that
   $$ d(t_0z,y)<C^{-1} d(z,y). $$
   Then $t_0^{-1} \in (0,1]$ satisfies
   $$d(t_0^{-1}(t_0z),y) = d(z,y) > C\cdot d(t_0z,y);$$
   a contradiction to Lemma \ref{lem:boomer_control}.
\end{proof}

The corollary below also follows immediately.

\begin{cor}\label{cor:bdd_sym_fix}
    Let $S$ be a closed symmetric subset of $G$. Suppose that the set $Sy=\{sy \mid s\in S\}$ satisfies the following:\\
	(1) $Sy$ is closed under multiplication, that is, if $s_1,s_2\in S$, then $s_1s_2y\in Sy$;\\
	(2) $Sy$ is bounded.\\
	Then $Sy=\{y\}$.
\end{cor}

\begin{proof}
   Let $H=\overline{\langle S \rangle}$, the closure of the subgroup generated by $S$. By conditions (1,2), we have that $Hy=Sy$ is bounded. Hence, $H$ is a compact subgroup of $G$. Given that $G=\mathbb{R}\times K$, where $K$ is the isotropy subgroup at $y$, the compact subgroup $H$ must be contained in $K$. The result follows.
\end{proof}

\begin{rem}
  We remark that Corollary \ref{cor:bdd_sym_fix} was derived without using Proposition \ref{prop:orbit_R} in \cite[Corollary 5.12]{Pan23}.
\end{rem}

\subsection{Convergence of appropriate symmetric subsets to cylinders}\label{subsec:cylinder}

We continue to assume that $(\widehat{N},\hat{q})\to (N,q)$ is a covering map with covering group $\Gamma\simeq \mathbb{Z}$ and escape rate $E(\widehat{N},\hat{q},\Gamma)\not= 1/2$, where $(N,q)$ is an open manifold with $\Ric\ge 0$. Let $(Y,y,G)$ be an equivariant asymptotic cone of $(\widehat{N},\hat{q},\Gamma)$. We have seen that $G=\mathbb{R}\times K$, where $K$ fixes $y$.

The main results of this section are Propositions \ref{prop:sym_gh} and \ref{prop:sym_saturated}. They concern appropriate symmetric subsets of $\Gamma$ and show that they approximate a \emph{cylinder}-like subset $[-v,v]\times K$ in $G$, where $v>0$. The results in this subsection will help us construct appropriate domains in $\widehat{N}$ relating to $\Omega_r(v)$ in the asymptotic cones in a later section (see Definition \ref{def:omega_set} and Section \ref{subsec:cube}).

As a convention, for an equivariant asymptotic cone $(Y,y,G)$ of $(\widehat{N},\hat{q},\langle \gamma \rangle)$ coming from
\begin{equation}\label{eq:orbit_R}
(r_i^{-1}\widehat{N},\hat{q},\langle \gamma \rangle)\overset{GH}\longrightarrow (Y,y,G),
\end{equation}
where $\gamma$ is a fixed generator of $\Gamma\simeq \mathbb{Z}$, we will always choose an $\mathbb{R}$-subgroup of $G$ (thus a splitting $G=\mathbb{R}\times K$) as follows. 

\begin{conv}\label{conv:R_subgroup}
Associated to the convergent sequence (\ref{eq:orbit_R}), for each $i$, we choose $m_i$ as the smallest positive integer $m$ such that
$$d(\gamma^m \hat{q},\hat{q})\ge r_i.$$
Then $\gamma^{m_i}\overset{GH}\to g\in G$ associated to (\ref{eq:orbit_R}) with $d(gy,y)=1$. As the isotropy subgroup $K$ is compact Lie and abelian, we can write $K=\mathbb{T}\times F$, where $\mathbb{T}$ is a torus and $F$ is a finite group. With this, we write the $F$-component of $g$ as $f$, which can be viewed as an element in $G$. Now, we choose a one-parameter subgroup through $gf^{-1}\in \R\times \mathbb{T}=G_0$. (There are multiple ones if the isotropy subgroup $K$ has positive dimension; we can choose any of them here.) We obtain a splitting $G=\mathbb{R}\times K$, where $\mathbb{R}\times \{e\}$ is the chosen one-parameter subgroup ($e$ is the identity element). Relabeling the elements in this $\mathbb{R}$-subgroup, we write $g=(1,e')\in \mathbb{R}\times K$, where $e'=(e,f)\in \mathbb{T}\times F$.

We remark that the choices of $m_i$ and $g=(1,e')$ are indeed no more than conventions. One may label any $g\in G$ with $d(gy,y)=1$ as $(1,e')$ and may choose any sequence $m_i$ with $\gamma^{m_i}\overset{GH}\to g$, then the results below in this subsection still hold, but the proofs should be modified correspondingly.
\end{conv}

From now on in this subsection, we always consider a convergent sequence (\ref{eq:orbit_R}) and let ${m_i}\to\infty$ be a sequence as chosen in Convention \ref{conv:R_subgroup}. For each $l>0$, we denote 
\begin{equation}\label{eq:sym_gamma}
S(l)=\{ e,\gamma^{\pm 1},...,\gamma^{\pm \lfloor l \rfloor} \},
\end{equation}
where $\lfloor \cdot \rfloor$ is the floor function. Regarding the sequence $S(m_i)$, we pass to a subsequence and obtain convergence
\begin{equation}\label{eq:orbit_R_sym}
(r_i^{-1}\widehat{N},\hat{q},\langle \gamma \rangle,S(m_i),\gamma^{m_i})\overset{GH}\longrightarrow (Y,y,G=\mathbb{R}\times K,S_\infty,(1,e')),
\end{equation}
where $S_\infty$ is closed symmetric subset of $G$. By the choice of $m_i$, it is clear that 
\begin{equation}\label{eq:limit_sym_1_ball}
S_\infty \cdot y \subseteq \overline{B_1}(y).
\end{equation}
We also observe that $S_\infty\cdot y$ is connected by construction. 

Lemma \ref{lem:isotropy_elements} below addresses the isotropy elements in $G$.

\begin{lem}\label{lem:isotropy_elements}
   Let $l_i$ be a sequence of integers such that $\gamma^{l_i}\overset{GH}\to h \in G$ associated to (\ref{eq:orbit_R}). Then $h\in K$ if and only if $|l_i|\ll m_i$. In particular, $S_\infty$ in (\ref{eq:orbit_R_sym}) contains the isotropy subgroup $K$.
\end{lem}

\begin{proof}
   After passing to a subsequence, we assume that $l_i$ has a fixed sign. Below, we consider the case where $l_i$ is a sequence of positive integers. The negative case can be handled similarly.

   Let $l_i \ll m_i$; we shall show that $h$, the limit of $\gamma^{l_i}$, belongs to $K$. Suppose the contrary $d(hy,y)>0$, then we can write
   $$h=(\delta,\beta)\in \mathbb{R}\times K =G,$$
   where $\delta\not =0$ and $\beta\in K$. Let $k$ be a large integer such that 
   $$10<d((k\delta,\beta^k)\cdot y,y)=d(h^k \cdot y,y).$$
   Then $\gamma^{kl_i}\overset{GH}\to (k\delta,\beta^k)$ with $\gamma^{kl_i}\in S(m_i)$, which contradicts (\ref{eq:limit_sym_1_ball}). Hence $h\in K$.

   Next, we prove the other direction. Let $l_i$ be a sequence of positive integers such that $h$, the limit of $\gamma^{l_i}$ associated to (\ref{eq:orbit_R}), belongs to $K$.

   \textit{Case 1.} $l_i/m_i\to c>0$ for a subsequence. We can pass to a subsequence and choose an integer $k$ such that $kl_i>m_i$ for all $i$. Then associated to (\ref{eq:orbit_R}), $\gamma^{kl_i}$ converges to $h^k\in K$. We consider the convergence
   $$(r_i^{-1}\widehat{N},\hat{q},S(m_i),S(kl_i),\gamma^{m_i},\gamma^{kl_i})\overset{GH}\longrightarrow (Y,y,S_\infty, T_\infty,(1,e'),h^k).$$
   Since $0<m_i<kl_i$, it is clear that $T_\infty \supseteq S_\infty$. In particular, then the set $T_\infty \cdot y$ contains the point $(1,e')\cdot y\not= y$. Also, $T_\infty \cdot y$ is bounded; this is because for $i$ large,
   $$kl_i\le k\cdot 2c m_i \le C m_i$$
   for some large integer $C$, and thus $T_\infty y \subseteq \overline{B_C}(y).$
   
   We claim that $T_\infty \cdot y$ is closed under multiplication (see Corollary \ref{cor:bdd_sym_fix}(1)); if this holds, then Corollary \ref{cor:bdd_sym_fix} implies $T_\infty \cdot y =\{y\}$, a contradiction. Now we verify the claim. Let $\alpha,\beta\in T_\infty$; they are the limit of $\gamma^{a_i}$ and $\gamma^{b_i}$, respectively, where $a_i,b_i$ are integers within $[-kl_i,kl_i]$. When $a_i+b_i\in [-kl_i,kl_i]$, it is clear that $\alpha\beta = \lim \gamma^{a_i+b_i} \in T_\infty$. When $a_i+b_i\in [kl_i,2kl_i]$, we write $a_i+b_i = kl_i+o_i$, where $o_i\in \mathbb{Z}\cap [0,kl_i]$. Then
   $$\alpha \beta y = \lim \gamma^{o_i}\gamma^{kl_i}\hat{q}= g_o \cdot h^k \cdot y = g_o y \in T_\infty y,$$
   where $g_o=\lim \gamma^{o_i}\in T_\infty.$ A similar argument covers the remaining case $a_i+b_i\in [-2kl_i,-kl_i]$. This verifies the claim and thus rules out Case 1.

   \textit{Case 2.} $l_i/m_i\to \infty$ for a subsequence. Let 
   $$d_i =\max \{ d(\gamma^m\hat{q},\hat{q}) \mid m=0,1,...,l_i \}.$$
   We observe that $d_i\gg r_i$. In fact, associated to (\ref{eq:orbit_R_sym}), $\gamma^{km_i}$ has limit $(k,(e')^k)\in G$ for each positive integer $k$. Since $d_Y((k,(e')^k)\cdot y,y)\to\infty$ as $k\to\infty$ and $km_i\ll l_i$ for each $k$, we see that $d_i\gg r_i$.

   We blow-down $\widehat{N}$ by $d_i\to\infty$:
   $$(d_i^{-1}\widehat{N},\hat{q},\langle \gamma \rangle,S(l_i),\gamma^{l_i})\overset{GH}\longrightarrow (Y',y',G',S'_\infty,h').$$
   It follows from the construction that 
   $$\{y'\} \subsetneqq S'_\infty \cdot y' \subseteq \overline{B_1}(y');$$
   moreover, $h'y'=y'$ since $d(\gamma^{l_i}\hat{q},\hat{q})\ll r_i \ll d_i$. Then, one can follow a similar argument in Case 1 to show that $S'_\infty\cdot y'$ is closed under multiplication. Then Corollary \ref{cor:bdd_sym_fix} implies $S'_\infty \cdot y'=\{y'\}$, a contradiction. Hence, Case 2 cannot happen.

   We are left with $l_i\ll m_i$ as the only possibility; this completes the proof.
\end{proof}

\begin{lem}\label{lem:egh_Q}
   (1) Let $t$ be a rational number, then associated to (\ref{eq:orbit_R_sym}) and after passing to a convergent subsequence if necessary, we have $$\gamma^{\lfloor m_i t \rfloor} \overset{GH} \to (t,h)\in \mathbb{R}\times K=G$$
   for some $h\in K$.\\
   (2) Let $t_1<t_2$ be rational numbers of the same sign and let $w\in(t_1,t_2)$, then we can find a sequence of integers $k_i\in [\lfloor m_i t_1 \rfloor,\lfloor m_i t_2 \rfloor]$ such that $\gamma^{k_i}$ converges to $(w,h)\in \mathbb{R}\times K$ for some $h\in K$.
\end{lem}

\begin{proof}
   (1) Below we assume that $t>0$; the case $t<0$ can be handled similarly. Let $t=a/b\in\mathbb{Q}_+$, where $a,b$ are positive and co-prime integers. It is clear that
  $$m_i a -b \le b\cdot \lfloor \frac{m_ia}{b} \rfloor \le m_i a,$$
  where $\lfloor \cdot \rfloor$ is the floor function. Hence we can write $b\lfloor m_ia/b \rfloor =m_i a +j_i$ for some $j_i\in \{-b,...,-1,0\}$. For each $j\in \{-b,...,-1,0\}$, associated to (\ref{eq:orbit_R_sym}), we have convergence
  $$\gamma^{m_i a + j} \overset{GH}\to (a,(e')^a\beta_j) \in \mathbb{R}\times K,$$
  where $\beta_j\in K$ is the limit of $\gamma^{j}$. Hence, we can pass to a subsequence and obtain
  $$\gamma^{b\cdot \lfloor \frac{m_ia}{b} \rfloor} \overset{GH}\to (a,\beta),$$
  where $\beta \in \{(e')^a\beta_{j} \mid j=-b,...,-1,0\}$. Then for
  any subsequential limit $\lambda\in G$ of $\gamma^{\lfloor \frac{m_ia}{b}\rfloor}$, $\lambda$ satisfies $\lambda^b = (a,\beta)$. Therefore, $\lambda=(a/b,h)=(t,h)$ for some $h\in K$ with $h^b=\beta$. 

  (2) We assume that $t_1$ and $t_2$ are both positive. The other case can be handled similarly. 

  We consider a sequence of subsets
  $$T_i=\{ \gamma^m \mid m=\lfloor m_i t_1\rfloor,\lfloor m_i t_1\rfloor+1,..., \lfloor m_i t_2\rfloor \}$$
  and its convergence
  $$(r_i^{-1}\widehat{N},\hat{q},\langle \gamma \rangle,T_i, \gamma^{\lfloor m_i t_1\rfloor},\gamma^{\lfloor m_i t_2\rfloor})\overset{GH}\longrightarrow (Y,y,G,T_\infty,(t_1,h_1),(t_2,h_2)),$$
  where the limits $(t_1,h_1)$ and $(t_2,h_2)$ are guaranteed by (1). We observe that $T_\infty\cdot y$ is connected because $T_\infty \cdot y$ is bounded and $r_i^{-1}d(\gamma\hat{q},\hat{q})\to 0$. Together with Proposition \ref{prop:orbit_R} that $Gy$ is homeomorphic to $\R$, we conclude that $T_\infty \cdot y$ contains the orbit point $(w,e)\cdot y$. In particular, there is some sequence $\gamma^{k_i}\in T_i$ whose limit is $(w,h)$ for some $h\in K$.
\end{proof}

Now we show that for each $v>0$, the symmetric subsets $S(m_i v)$ converge to the cylinder-like subset $[-v,v]\times K$ in $G$.

\begin{prop}\label{prop:sym_gh}
  For each $v>0$, associated to (\ref{eq:orbit_R_sym}) and after passing to a convergent subsequence if necessary, we consider the convergence
  \begin{equation}\label{eq:sym_subsets}
  (r_i^{-1}\widehat{N},\hat{q},\langle \gamma \rangle,S(m_i v))\overset{GH}\longrightarrow (Y,y,G=\mathbb{R}\times K,S^v_\infty),
  \end{equation}
  where $S^v_\infty$ is a closed symmetric subset of $G$. Then $S^v_\infty= [-v,v]\times K \subseteq G$.
\end{prop}

\begin{proof}
  \textbf{Step 1.} We first prove that 
  $$S^v_\infty \cdot y \supseteq \{ tz \mid t\in [-v,v] \},$$
  where $z=(1,e)\cdot y$ and the notation $tz$ is introduced in Definition \ref{def:one_para_orbit}. From Lemma \ref{lem:egh_Q} we see that for any rational number $t\in(-v,v)$, there is an element $\lambda=(t,h)\in S_\infty^v$. In particular, 
  $$ S^v_\infty \cdot y \ni \lambda\cdot y = (t,h)\cdot y = tz. $$
  Given that $S^v_\infty\cdot y$ is a closed subset of $Y$ and $t$ can be any arbitrary rational number within $(-v,v)$, we have completed Step 1.

  \textbf{Step 2.} $S_\infty^v \supseteq [-v,v] \times K$. In Step 1, we have shown that for any rational number $t\in (0,v)$, there is a sequence of elements $\gamma^{\lfloor m_it \rfloor}\overset{GH}\to (t,\beta_t)\in \mathbb{R}\times K$ associated to (\ref{eq:orbit_R_sym}). For any $\beta \in K$, by Lemma \ref{lem:isotropy_elements}, there is a sequence $l_i$ such that $\gamma^{l_i}\overset{GH}\to \beta$ and $|l_i|\ll m_i$. Since $\lfloor m_it \rfloor + l_i < m_i v$ for all $i$ large and $S^v_\infty$ is closed in $G$, we see that $S^v_\infty$ includes all the elements in $[0,v]\times K$. The other half $[-v,0]\times K$ can be proved similarly.

  \textbf{Step 3.} $S_\infty^v \subseteq [-v,v] \times K$. Suppose otherwise, then we would have a sequence $\gamma^{l_i}\in S(m_iv)$, where $l_i>0$, such that 
  $$\gamma^{l_i}\overset{GH}\to (w,\beta) \in \mathbb{R}\times K $$
  associated to (\ref{eq:orbit_R_sym}), where $w>v$ or $w<-v$.
  
  When $w>v$, let $k$ be an integer such that $kv>w$. We choose rational numbers $t_1,t_2$ such that $v<t_1<w<t_2<kv$. By Lemma \ref{lem:egh_Q}, there is a sequence of integers $k_i \in [\lfloor m_i t_1 \rfloor,\lfloor m_i t_2 \rfloor]$ such that
  $$\gamma^{k_i}\overset{GH}\longrightarrow (w,h) \in \mathbb{R}\times K$$
  for some $h\in K$. Consequently,
  $$\gamma^{k_i-l_i}\overset{GH}\longrightarrow (0, h \beta^{-1})\in K.$$
  However, it is clear that the sequence $k_i-l_i$ does not satisfy $k_i-l_i\ll m_i$  by our construction, thus we end in a contradiction to Lemma \ref{lem:isotropy_elements}. 

  When $w<-v$, we will derive a contradiction similarly. In fact, we shall show that $w<0$ cannot happen by using $l_i>0$. We suppose that $w<0$ and look for a contradiction. Let $t_1<t_2<0$ be two negative rational numbers such that $w\in (t_1,t_2)$. Then similarly by Lemma \ref{lem:egh_Q} a sequence of negative integers $k_i\in [\lfloor m_i t_1 \rfloor,\lfloor m_i t_2 \rfloor]$ such that 
  $$\gamma^{k_i}\overset{GH}\longrightarrow (w,h) \in \mathbb{R}\times K.$$
  Then, $\gamma^{l_i-k_i}$ converges to an isotropy element, but $l_i-k_i \ll m_i$ does not hold; a contradiction to Lemma \ref{lem:isotropy_elements}. That completes Step 3 and the proof of the proposition.
\end{proof}

\begin{rem}\label{rem:limit_sym_sign}
   We remark that in Step 3 of the proof above, we showed that any limit of a sequence $\gamma^{l_i}\in S(m_iv)$ associated to (\ref{eq:sym_subsets}) belongs to $[0,v]\times K$, if $l_i$ is a sequence of positive numbers.
\end{rem}

Next, we show that the convergence $S(m_i v)\overset{GH}\to [-v,v]\times K$ in Proposition \ref{prop:sym_gh} is \emph{saturated}, in the sense that we cannot have a sequence of elements outside $S(m_i v)$ converging to the interior $(-v,v)\times K$. 

\begin{prop}\label{prop:sym_saturated}
   Let $v>0$ and $w\in (-v,v)$. Suppose that $\gamma^{l_i}$ is a sequence converging to $(w,h)\in \mathbb{R}\times K=G$ associated to (\ref{eq:sym_subsets}). Then $\gamma^{l_i} \in S(m_i v)$ for all $i$ large. 
\end{prop}

\begin{proof}
   For $w=0$, we have already proved in Lemma \ref{lem:isotropy_elements} that $\gamma^{l_i}\overset{GH}\to (e,h)$ implies $|l_i|\ll m_i$. Below, we assume that $w>0$. The case $w<0$ can be proved similarly. The proof is similar to Step 3 in the proof of Proposition \ref{prop:sym_gh}.

   Suppose that $\gamma^{l_i}\overset{GH}\to (w,h)$ associated to (\ref{eq:orbit_R_sym}) with $w\in (0,v)$. We choose rational numbers $t_1,t_2$ such that $0<t_1<w<t_2<v$.
   Then by Lemma \ref{lem:egh_Q}, we have a sequence of integers $k_i\in [\lfloor m_i t_1 \rfloor,\lfloor m_i t_2 \rfloor]$ such that 
  $$\gamma^{k_i}\overset{GH}\longrightarrow (w,\beta) \in \mathbb{R}\times K.$$
  Then $\gamma^{l_i-k_i}$ converges to an isotropy element. According to Lemma \ref{lem:isotropy_elements}, we obtain $|l_i-k_i|\ll m_i$. Hence $\gamma^{l_i}\in S(m_i v)$ for all $i$ large.
\end{proof}

We close this subsection with a distance estimate on $\widehat{N}$, which is related to Corollary \ref{cor:boomer_control}.

\begin{lem}\label{lem:larger_power_dist}
   Let $l_i$ be a sequence of integers with $l_i\ge m_i$. Then
   $$d(\gamma^{l_i} \hat{q},\hat{q})\ge (2C)^{-1} r_i$$
   for all $i$ large, where $C$ is the constant in Lemma \ref{lem:boomer_control}. 
\end{lem}

\begin{proof}
     Suppose the contrary, then, after passing to a subsequence if necessary, we would have $\gamma^{l_i}\overset{GH}\to h\in G$ associated to (\ref{eq:orbit_R_sym}) with $d(hy,y)\le (2C)^{-1}$. Since $l_i\ge m_i$, for each $i$, we can choose an integer $k_i\ge 1$ such that
     $$k_i m_i \le  l_i \le (k_i +1) m_i.$$

     If the sequence $k_i$ is uniformly bounded, we pass to a subsequence and assume $k_i=k$ is a constant. We further write $l_i=km_i+o_i$, where $0\le o_i \le m_i$. Then associated to (\ref{eq:orbit_R_sym}), we have
   $$h=\lim \gamma^{l_i}= \lim (\gamma^{m_i})^k \cdot \lim \gamma^{o_i}= (k,(e')^k) \cdot g_o,$$
   where $g_o \in S_\infty$ is the limit of $\gamma^{o_i}$.
   By Proposition \ref{prop:sym_gh} with $v=1$ and Remark \ref{rem:limit_sym_sign}, $g_o=(w,\beta)$ for some $w\in [0,1]$ and $\beta \in K$. Let $z=(1,e)\cdot y$, then
   $$hy = (k,(e')^k)\cdot g_o y = (k+w)z\in \{ tz \mid t\in [k,k+1] \}.$$
   Then $d(hy,y)\le (2C)^{-1}$ contradicts Corollary \ref{cor:boomer_control}.

   If $k_i\to \infty$ for some subsequence, we shall follow a similar argument in Lemma \ref{lem:isotropy_elements} Case 2 to derive a contradiction. Let 
   $$d_i =\max \{ d(\gamma^m \hat{q},\hat{q}) \mid m=0,1,...,l_i \}.$$ Then $d_i\gg r_i$. After blowing down $\widehat{N}$ by $d_i^{-1}$:
   $$(d_i^{-1}\widehat{N},\hat{q},\langle \gamma \rangle,S(l_i),\gamma^{l_i})\overset{GH}\longrightarrow (Y',y',G',S'_\infty,h').$$
   Similar to the proof in Lemma \ref{lem:isotropy_elements} Case 2, one can verify that
   $$h'y'=y',\quad \{y'\} \subsetneqq S'_\infty \cdot y' \subseteq \overline{B_1}(y'),$$
   and $S'_\infty \cdot y'$ is closed under multiplication; a contradiction to Corollary \ref{cor:bdd_sym_fix}.
\end{proof}

\subsection{Statement of the Induction Theorem and an outline}\label{subsec:outline}

In this section, we state a strengthened version of the Induction Theorem and explain the ideas behind its proof.

We first recall the setup. Let $M$ be an open manifold with $\Ric\ge0$ and linear volume growth. Let $(\widehat{M},\hat{p})$ be a covering space of $(M,p)$ with a finitely generated torsion-free nilpotent covering group $\Lambda$. We consider a series of normal subgroups of $\Lambda$:
\begin{equation}\label{eq:series_subgroups}
\{e\}=\Lambda_0 \triangleleft \Lambda_1 \triangleleft ... \triangleleft \Lambda_{k-1} \triangleleft \Lambda_k =\Lambda,
\end{equation}
such that each $\Lambda_{j+1}/\Lambda_j$ is isomorphic to $\mathbb{Z}$ (see \cite[5.5.20]{book_Robinson}). This corresponds to a tower of successive covering spaces:
\begin{equation}\label{eq:tower_covers}
\widehat{M}=\widehat{M}_0\to \widehat{M}_{1} \to ... \to \widehat{M}_{k-1} \to \widehat{M}_k=M,
\end{equation}
where $\widehat{M}_j = \widehat{M}/\Lambda_j$. It follows from the construction that each covering map $\widehat{M}_j \to \widehat{M}_{j+1}$ has covering group $\Lambda_{j+1}/\Lambda_j \simeq \mathbb{Z}$.

\begin{thm}[Induction Theorem+]\label{thm:induction_plus}
Let $j=0,1,...,k$ and let $r_i\to\infty$ be a sequence. After passing to a convergent subsequence if necessary, we consider the convergence
\begin{equation}\label{eq:induction}
(r_i^{-1}\widehat{M}_{k-j},\hat{p}_{k-j}) \overset{GH}\longrightarrow (Y_{k-j},y_{k-j}).
\end{equation}
for all $j$. Then, the following holds.\\
(1) $(Y_{k-j},y_{k-j})$ is isometric to either a Euclidean space $(\mathbb{R}^{j+1},0)$ or a Euclidean halfspace $(\mathbb{R}^j \times [0,\infty),0)$.\\
(2) $Y_{k-j}$ has a limit renormalized measure as (a multiple of) the Lebesgue measure.\\
(3) When $Y_{k-j}=\mathbb{R}^j\times [0,\infty)$, for every $s>0$, there is a sequence of closed domains $D_{i}(s)\subseteq \widehat{M}_{k-j}$ such that\\
(3A) $D_{i}(s)\overset{GH}\to [-s,s]^j\times [0,s] \subset Y_{k-j}$ associated to (\ref{eq:induction});\\
(3B) for every $u\in (-s,s)^j \times [0,s)$ and every sequence $u_i\overset{GH}\to u$ associated to (\ref{eq:induction}), it holds that $u_i\in D_i(s)$ for all $i$ large.
%moreover, there is a constant $c(j)\in (0,1]$ such that $D_i(1)\supseteq B_{c(j)r_i}(q)$ for all $i$ large.
\end{thm}

\begin{rems} 
    We give some remarks on Theorem \ref{thm:induction_plus}.\\
   (1) Theorem \ref{thm:induction_plus}(1) implies that the asymptotic cone of $\widehat{M}_{k-j}$ is unique. In fact, if both $(\mathbb{R}^{j+1},0)$ and $(\mathbb{R}^j \times [0,\infty),0)$ appear as asymptotic cones of $\widehat{M}_{k-j}$, then it would contradict the connectedness of the set of asymptotic cones under (pointed) Gromov--Hausdorff topology (see, for example, \cite[Proposition 2.1]{Pan19}).\\
   (2) Theorem \ref{thm:induction_plus}(2) asserts that $Y_{k-j}$ has some limit renormalized measure as (a multiple of) the Lebesgue measure. While it does not show every limit renormalized measure on $Y_{k-j}$ has this property, it is sufficient for our purpose.\\
   (3) Theorem \ref{thm:induction_plus}(1,2) is exactly the Induction Theorem (Theorem \ref{thm:induction}) in the introduction. Here, we added (3) as an auxiliary property, which facilitates the proof of Theorem \ref{thm:induction_plus}(1,2).
   %(4) The constant $c(j)$ in Theorem \ref{thm:induction_plus}(3) indeed can be chosen arbitrarily close to $1$. However, a small constant $c(j)\in (0,1]$ is sufficient for our proof of Theorem \ref{thm:induction_plus}(1,2), so we do not seek for an optimal $c(j)$ here.
\end{rems}

The base step $j=0$ in Theorem \ref{thm:induction_plus}(1,2) follows immediately from Corollary \ref{cor:induction_base}. For Theorem \ref{thm:induction_plus}(3) with $j=0$, we can simply choose $D_i(s):=\overline{B_{r_i s}}(p)\subseteq M$. %with constant $c(0)=1$.

Sections \ref{subsec:cube} and \ref{subsec:induction} will be devoted to proving the inductive step in Theorem \ref{thm:induction_plus} (modulo Theorem \ref{thm:plane_halfplane_rigid}), that is, proving the statement with $j+1$ by assuming it holds for $j$. In Sections \ref{subsec:cube} and \ref{subsec:induction}, we usually refer to one of the statements in Theorem \ref{thm:induction_plus} with $j$ as \emph{inductive assumption} ($\alpha$), where $\alpha\in \{1,2,3,3A,3B\}$

We give an outline of the inductive step and explain how Section \ref{subsec:cylinder} and the inductive assumptions are involved. For convenience, we write 
\begin{equation}\label{eq:induction_space}
(\widehat{N},\hat{q},\Gamma):=(\widehat{M}_{k-(j+1)},\hat{p}_{k-(j+1)},\mathbb{Z}),\quad (N,q):=(\widehat{M}_{k-j},\hat{p}_{k-j}).
\end{equation}
For any sequence $r_i\to\infty$, after passing to a subsequence, we have convergence
\begin{equation}\label{eq:induction_notation}
\begin{CD}
(r_i^{-1} \widehat{N},\hat{q},\Gamma\simeq \mathbb{Z}) @>GH>> (Y,y,G) \\
	@VV\pi V @VV \pi V\\
	(r_i^{-1} N,q) @>GH>> (X,x)=(Y/G,\bar{y}),
\end{CD}
\end{equation}
and we will continue to use this notation in the next two subsections. We have the inductive assumption that Theorem \ref{thm:induction_plus} holds on the bottom row of (\ref{eq:induction_notation}), and we need to prove the theorem for the top row. Because $X$ is polar at $x$ by inductive assumption (1), Propositions \ref{prop:non_max_escape} and \ref{prop:orbit_R} together imply that the orbit $Gy$ is homeomorphic to $\mathbb{R}$; in particular, we can write $G=\mathbb{R}\times K$, where $K$ is the isotropy subgroup of $G$ at $y$, as we did in Sections \ref{subsec:orbit_R} and \ref{subsec:cylinder}. 

The nontrivial case in the inductive step is when $X=\mathbb{R}^j \times [0,\infty)$. In this case, we have $Y=\mathbb{R}^j \times Z$ with $Z/G=[0,\infty)$. As indicated in the introduction, we shall investigate a limit renormalized measure $\meas_Y =\mathcal{L}^j \otimes \meas_Z$, in particular, the value of $\meas_Z(\Omega_s(v))$ (see Definition \ref{def:omega_set}), so that we can apply the plane/halfplane rigidity (Theorem \ref{thm:plane_halfplane_rigid}) to $Z$. 

To achieve this, for each $s,v>0$ we will construct a sequence of appropriate domains $\mathcal{D}_i(s,v)\subseteq \widehat{N}$ such that associated to the top row of (\ref{eq:induction_notation}), it holds that
\begin{align}\label{eq:induction_meas_outline}
\mathcal{D}_i(s,v) &\overset{GH}\to [-s,s]^j \times \Omega_s(v),\\
\meas_i(\mathcal{D}_i(s,v)) &\to \meas_Y([-s,s]^j \times \Omega_s(v))=(2s)^j \cdot \meas_Z(\Omega_s(v)),
\end{align}
where $\meas_i$ is the renormalized measure on $r_i^{-1}\widehat{N}$ (see Definition \ref{defn:renorm_meas}). If we have these domains and can calculate their renormalized measure, then we can obtain $\meas_Z(\Omega_s(v))$. 

The construction of $\mathcal{D}_i(s,v)$ with limit $[-s,s]^j \times \Omega_s(v)$ has two components. The first one is related to the $s$-direction. By the inductive assumption (3A), we have a sequence of domains $D_i(s)\subseteq N$ with limit $[-s,s]^j \times [0,s]$ associated to the bottom row of (\ref{eq:induction_notation}). We lift $D_i(s)$ to the Dirichlet domain $F$ centered at $\hat{q}$ in $\widehat{N}$. Then, roughly speaking, $F_i(s):={F\cap \pi^{-1}(D_i(s))}$ covers the $s$-direction in $[-s,s]^j \times \Omega_s(v)$. The second component is related to the $v$-direction. Here, we use portions of the $\Gamma$-action to translate $F_i(s)$. By Proposition \ref{prop:sym_gh}, the sequence of symmetric subsets $S(m_i v)$ has limit $[-v,v]\times K$, covering the desired $v$-direction. Hence $S(m_i v)\cdot F_i(s)$ should be the right choice for $\mathcal{D}_i(s,v)$. Proposition \ref{prop:sym_saturated} and inductive assumption (3B) will assure the measure convergence in (\ref{eq:induction_meas_outline}).

\subsection{Convergence of appropriate domains to cubes}\label{subsec:cube}

In this subsection, we follow the outline in Section \ref{subsec:outline} to study the convergence of domains $S(m_i v)\cdot F_i(s)$. We continue to use the notations in (\ref{eq:induction_notation}) and assume that the inductive assumptions hold on the bottom row of (\ref{eq:induction_notation}) with $X=\mathbb{R}^j \times [0,\infty)$, which is the nontrivial case in the inductive step. We observe that $Y$ in (\ref{eq:induction_notation}) splits isometrically as $\mathbb{R}^j \times Z$ with $G$ acting trivially on the $\mathbb{R}^j$-factor. Moreover, by Theorem \ref{thm:meas_split} any limit renormalized measure $\meas_Y$ splits as $\mathcal{L}^j \otimes \meas_Z$, where $\mathcal{L}^j$ is the Lebesgue measure on $\mathbb{R}^j$.

\begin{rem}
   As mentioned in Section \ref{subsec:outline}, we expect this sequence of domains $S(m_iv)\cdot F_i(s)$ to converge to the limit $[-s,s]^j \times \Omega_s(v)$. Eventually, the set $\Omega_s(v)$ is either $[-v,v]\times [-s,s]$ in a Euclidean plane or $[-v,v]\times [0,s]$ in a Euclidean halfplane, so we call these subsets $[-s,s]^j \times \Omega_s(v)$ \emph{cubes} in the title of this subsection.
\end{rem}

For readers' convenience, we recall some notations:\\
$\bullet$ Domains $D_i(s)\subseteq N$ from the inductive assumption (3) of Theorem \ref{thm:induction_plus},\\
$\bullet$ $F\subset \widehat{N}$ the (closed) Dirichlet domain centered at $\hat{q}$, that is,
$$F=\{ q'\in \widehat{N}\mid d(q',\hat{q})\le d(q',\gamma \hat{q}) \text{ for all } \gamma\in \Gamma-\{e\} \}.$$
We note that its boundary $\partial F$ has zero Lebesgue measure (see \cite[page 43]{Anderson90}).\\
$\bullet$ $F_i(s)={F\cap \pi^{-1}(D_i(s))}\subseteq \widehat{N}$,\\
$\bullet$ $m_i\to +\infty$ chosen in Convention \ref{conv:R_subgroup}.\\
$\bullet$ $S(m_i v)=\{ \gamma^m \mid m=0,\pm 1,... \pm \lfloor m_iv\rfloor \}$ defined in (\ref{eq:sym_gamma}).

We start with a lemma on horizontal rays in $Z$, which follows from the non-branching property of $\RCD$ spaces \cite{Deng20}. Recall that $(Y,y)=(\mathbb{R}^j\times Z,(0,z))$ has quotient
   $$Y/G= \mathbb{R}^j\times (Z/G)=\mathbb{R}^j\times [0,\infty)=X.$$

\begin{lem}\label{lem:horizontal_rays}
   Let $\sigma:[0,\infty)\to Z$ be a unit speed horizontal ray emanating at $z$. Then\\
   (1) $\sigma|_{[0,t]}$ is the unique minimal geodesic from $Gz$ to $\sigma(t)$.\\
   (2) Any unit speed horizontal ray in $Z$ emanating at $z$ has the form $h\cdot \sigma$, where $h\in K$.
\end{lem}

\begin{proof}
   (1) Suppose that there is a different minimal geodesic $\alpha$ from $Gz$ to $\sigma(t)$. If $\alpha$ starts at $z$, then we clearly end in a contradiction to the non-branching property \cite{Deng20} because $\sigma$ is a ray. Suppose that $\alpha$ starts at a point in $Gz$ other than $z$, say $gz\not=z$. Since $\sigma$ is a horizontal ray, for any $t'>t$ we have
   \begin{align*}
   t' &= d_Z(Gz,\sigma(t')) \le d_Z(gz,\sigma(t')) \\
   & \le d_Z(gz,\sigma(t))+d_Z(\sigma(t),\sigma(t')) \\
   &\le t+(t'-t)=t'.
   \end{align*}
   This shows that $\alpha$ joining $\sigma|_{[t,t']}$ is a minimal geodesic from $gz$ to $\sigma(t')$. Again, we end in a contradiction to the non-branching property.

   (2) Let $\beta$ be any unit speed horizontal ray in $Z$ emanating at $z$. Since $\beta(1)$ and $\sigma(1)$ project to the same point $1\in [0,\infty)=Z/G$, there is an element $g\in G$ such that $g\cdot \sigma(1)=\beta(1)$. We observe that $g\cdot \sigma$ is a horizontal ray at $g\cdot z$ and shares a common point $\beta(1)$ with $\beta$. Hence by (1), $g \cdot \sigma$ and $\beta$ coincide on $[0,t]$. In particular, $gz=z$, so we have $g\in K$. Also, $g \cdot \sigma$ and $\beta$ must coincide on $[t,\infty)$ as well; otherwise, we would result in a branching point at $\beta(1)$. This completes the proof.
\end{proof}

Now we show that $S(m_iv)\cdot F_i(s)$ indeed converges to $[-s,s]^j \times \Omega_s(v)$, as desired.

\begin{prop}\label{prop:domains_gh}
   Let $s,v>0$. After passing to a convergent subsequence, if necessary, we consider the convergence of domains
   \begin{equation}\label{eq:domains_gh}
   S(m_i v)\cdot F_i(s)\overset{GH}\to \mathcal{D}_\infty \subset Y
   \end{equation}
   associated to top row of (\ref{eq:induction_notation}) with $(Y,y)=(\mathbb{R}^j\times Z,(0,z))$, where $\mathcal{D}_\infty$ is a closed subset in $Y$. Then 
   $\mathcal{D}_\infty= [-s,s]^j\times \Omega_s(v)$.
\end{prop}

\begin{proof}
    Let $\sigma$ be a unit speed horizontal ray emanating at $z$. We first observe every point in $Y=\mathbb{R}^j\times Z$ can be written as $(u,g\cdot \sigma(t))$ for some $u\in \mathbb{R}^j$, $g\in G$, and $t\in [0,\infty)$. 

   Passing to a convergent sequence, we have the convergence of the Dirichlet domain $F$:
   $$(r_i^{-1}\widehat{N},\hat{q},F)\overset{GH}\longrightarrow (Y,y,F_\infty)$$ 
   for some closed subset $F_\infty$ in $Y.$

   \textbf{Claim 1.} $F_\infty \subset \mathbb{R}^j \times (K\cdot \sigma)$, where $K$ is the isotropy subgroup of $G$ at $y$, as usual.
   
   Recall that the (closed) Dirichlet domain $F$ centered at $\hat{q}$ is defined by
   $$F=\{ q'\in \widehat{N}\mid d(q',\hat{q})\le d(q',\gamma \hat{q}) \text{ for all } \gamma\in \Gamma-\{e\} \}.$$
   Passing this property to the limit, we see that the limit $F_\infty$ is contained in the set
   \begin{equation}\label{eq:domains_gh_set}
   \{ (u,g\cdot \sigma(t))\in Y \mid d((u,g\cdot \sigma(t)),(0,z))=d((u,g\cdot \sigma(t)),G\cdot (0,z)) \}.
   \end{equation}
   We calculate the distances involved by
   \begin{align*}
   d((u,g\cdot \sigma(t)),(0,z))&=\left( |u|^2 + d_Z(g\cdot \sigma(t),z) \right)^{1/2}\\
   &= \left( |u|^2 + d_Z(\sigma(t),g^{-1}\cdot z) \right)^{1/2},\\
   d((u,g\cdot \sigma(t)),G\cdot (0,z))&=\left( |u|^2 + d_Z(g\cdot \sigma(t),G\cdot z) \right)^{1/2}\\
   &= \left( |u|^2 + d_Z(\sigma(t), G \cdot z) \right)^{1/2}.
   \end{align*}
   Hence a point $(u,g\cdot \sigma(t))$ belongs to the set (\ref{eq:domains_gh_set}) if and only if it satisfies
   $$d_Z(\sigma(t),g^{-1}z)=d_Z(\sigma(t), G \cdot z).$$
    Together with Lemma \ref{lem:horizontal_rays}(1), we conclude that $g^{-1}z=z$, that is, $g\in K$. Claim 1 follows.

   \textbf{Claim 2.1.} For each non-zero vector $u\in\mathbb{R}^j$, the horizontal ray $t\mapsto (tu,z)\in \mathbb{R}^j \times Z$ is contained in the set $F_\infty$.
   
   \textbf{Claim 2.2.} For each $u\in \mathbb{R}^j$, there is $h_u\in K$ such that the horizontal ray
   $$\sigma_u(t)= (tu,h_u\cdot \sigma(t)) \subseteq \mathbb{R}^j \times Z, \quad t\ge 0$$
   is contained in $F_\infty$.
   
   We prove these two claims together. Let $\beta$ be an arbitrary ray in $X=\mathbb{R}^j\times [0,\infty)$ starting at $0$. By a standard diagonal argument, along the sequence $(r_i^{-1}N,q)$, we can choose a sequence of minimal geodesics $\beta_i$ starting at $q$ such that
   $$(r_i^{-1}N,q,\beta_i)\overset{GH}\longrightarrow (X,0,\beta).$$
   Next, we lift each $\beta_i$ to a horizontal geodesic $\hat{\beta}_i$ in $\hat{N}$ starting at $\hat{q}$. By construction, $\hat{\beta}_i$ is contained in the Dirichlet domain $F$ centered at $\hat{q}$. Then associated to (\ref{eq:induction_notation}), we consider
   \begin{equation*}
   \begin{CD}
   (r_i^{-1}{F},\hat{q},\hat{\beta}_i) @>GH>> (F_\infty, (0,z), \hat{\beta} ) \\
	@VV\pi V @VV \pi V\\
	(r_i^{-1} N,q,\beta_i) @>GH>> (\mathbb{R}^j\times [0,\infty),0,\beta),
   \end{CD}
   \end{equation*}
   where $\hat{\beta}$ is a horizontal ray starting at $(0,z)$ such that $\hat{\beta}\subseteq F_\infty$ and $\pi(\hat{\beta})=\beta$. 
   
   Taking $\beta$ in the form of $t\mapsto (tu,0)\in \mathbb{R}^j \times [0,\infty)$, where $u\in \mathbb{R}^j-\{0\}$, we obtain that the horizontal ray $\hat{\beta}(t)=(tu,z)$ is contained in $F_\infty$. This proves Claim 2.1.
   
   Taking $\beta$ in the form of $t\mapsto (tu,t)\in \mathbb{R}^j\times [0,\infty)$, where $u$ is a vector in $\mathbb{R}^j$, we obtain a corresponding ray $\hat{\beta}$ in $F_\infty$ such that 
   $\pi(\hat{\beta})=\beta$. It is clear that the $\mathbb{R}^j$-component of $\hat{\beta}$ is the same as $\beta$. For its $Z$-component, as a horizontal ray starting at $z$, by Lemma \ref{lem:horizontal_rays}(2) it has the form $h_u\cdot \sigma$ for some $h_u\in K$. This proves Claim 2.2.

   For convenience,  we denote below by $\mathcal{R}$ the union of these two families of horizontal rays from Claims 2.1 and 2.2. 

   Recall that by inductive assumption (3A) we have a sequence of domains $D_i(s)$ in $N$ converging to $[-s,s]^j\times [0,s] \subseteq X$
   associated to the bottom row of (\ref{eq:induction_notation}) with $X=\mathbb{R}^j\times [0,\infty)$. Hence 
   $$(r_i^{-1}(\pi^{-1}(D_i(s)),\hat{q}))\overset{GH}\longrightarrow ([-s,s]^j \times (G\cdot \sigma|_{[0,s]}),(0,z))$$
   associated to the top row of (\ref{eq:induction_notation}) with $Y=\mathbb{R}^j\times Z$. For the sequence $F_i(s)={F\cap \pi^{-1}(D_i(s))}$, it has a limit $F_\infty^s \subset Y$ from
   $$(r_i^{-1}\widehat{N},\hat{q},F_i(s))\overset{GH}\longrightarrow (Y,y,F_\infty^s).$$

   \textbf{Claim 3.} $F_\infty^s = F_\infty \cap \left([-s,s]^j \times (G\cdot \sigma|_{[0,s]})\right).$

   The inclusion $\subseteq $ is immediate. We verify the other direction $\supseteq$. For any point $$\hat{p}:=(u,g\cdot \sigma(t))\in F_\infty \cap \left((-s,s)^j \times (G\cdot \sigma|_{[0,s)})\right),$$
   where $u\in (-s,s)^j$ and $t\in [0,s)$, there is a sequence $\hat{p}_i\in F$ such that $\hat{p_i}\overset{GH}\to \hat{p}$ associated to the top row of (\ref{eq:induction_notation}). Let $p_i=\pi(\hat{p_i})\in N$. Then by construction $$p_i\overset{GH}\to (u,t)\in (-s,s)^j\times [0,s) \subseteq X.$$
   Using inductive assumption (3B), we see that $p_i\in D_i(s)$ for all $i$ large. Hence $\hat{p_i}\in F\cap \pi^{-1}(D_i(s))$ and consequently $\hat{p}\in F^s_\infty$. Because $F_\infty^s$ is closed, it contains the closure of $F_\infty \cap \left((-s,s)^j \times (G\cdot \sigma|_{[0,s)})\right)$. This verifies the other direction $\supseteq$, and thus completes the proof of Claim 3.
   
   Claims 1 and 3 yield 
   \begin{equation}\label{eq:limit_set_1}
   F_\infty^s\subseteq [-s,s]^j \times (K\cdot \sigma|_{[0,s]});
   \end{equation}
   Claims 2 and 3 yield
   \begin{equation}\label{eq:limit_set_2}
   F_\infty^s \supseteq \mathcal{R} \cap \left([-s,s]^j \times (G\cdot \sigma|_{[0,s]})\right).
   \end{equation}

   Let $v>0$. After passing to a subsequence, the sequence of symmetric subsets $S(m_i v)\subseteq \Gamma$ converges to a closed symmetric subset $S^v_\infty$ of $G$, that is,
   $$(r_i^{-1}\widehat{N},\hat{q},\Gamma,S(m_i v))\overset{GH}\longrightarrow (Y,y,G,S^v_\infty).$$
   Proposition \ref{prop:sym_gh} implies $$S^v_\infty = [-v,v]\times K \subseteq \mathbb{R}\times K=G.$$
   Applying (\ref{eq:limit_set_1}), we infer that $\mathcal{D}_\infty$, the limit of $S(m_i v)\cdot F_i(s)$ in (\ref{eq:domains_gh}), satisfies
   $$\mathcal{D}_\infty \subseteq [-s,s]^j \times \left( S^v_\infty \cdot (K\cdot \sigma|_{[0,s]}) \right) \subseteq [-s,s]^j \times \Omega_s(v).$$
   For the other direction, we apply (\ref{eq:limit_set_2}) and claim
   $$ \mathcal{D}_\infty \supseteq S^v_\infty \cdot (\text{RHS of } \ref{eq:limit_set_2}) \supseteq [-s,s]^j \times \Omega_s(v).$$
   We verify the last $\supseteq$ claimed above. We consider an arbitrary point
   $$\hat{u}=(u,(w,h)\cdot \sigma(r))\in [-s,s]^j \times \Omega_s(v),$$
   where $u\in [-s,s]^j$, $h\in K$, $w\in [-v,v]$, and $r\in [0,s]$. If $r=0$, then $\sigma(r)=z$ and
   $$\hat{u}= (w,h) \cdot (u,z),$$
   where $(w,h)\in S^v_\infty$ and $(u,z)\in (\text{RHS of } \ref{eq:limit_set_2})$. Hence $\hat{u} \in S^v_\infty \cdot (\text{RHS of } \ref{eq:limit_set_2})$. If $r\not= 0$, then the vector $u'=u/r\in\mathbb{R}^j$ gives rise to a horizontal ray $\sigma_{u'}(t)=(tu',h_{u'}\cdot \sigma(t))$ in $\mathcal{R}$, where $h_{u'}\in K$. It is also clear that $\sigma_{u'}(r)=(u,h_{u'}\cdot \sigma(r))$ belongs to $[-s,s]^j \times (G\cdot \sigma|_{[0,s]})$. As a result, $\sigma_{u'}(r) \in (\text{RHS of } \ref{eq:limit_set_2})$. Lastly, we observe that
   $$\hat{u}= (w,hh_{u'}^{-1}) \cdot \sigma_{u'}(r) \in S^v_\infty \cdot (\text{RHS of } \ref{eq:limit_set_2}).$$
   This verifies the claimed $\supseteq$ relation and completes the proof of this proposition.
\end{proof}

Let $s,v>0$. We denote 
$$\Omega^\circ _s(v) =\{ (w,h)\cdot \sigma(t)\mid w\in(-v,v), h\in K,  t\in [0,s) \}$$
the interior of $\Omega_s(v)$. We also denote $\partial \Omega_s(v)=\Omega_s(v)-\Omega^\circ _s(v)$ the boundary.

Next, we show that the convergence of domains $S(m_i v)\cdot F_i(s) \overset{GH}\to [-s,s]^j\times \Omega_s(v)$ is \emph{saturated}, in the sense that we cannot have a sequence of points outside $S(m_iv)\cdot F_i(s)$ converging to the interior $(-s,s)^j \times \Omega^\circ_s(v)$.

\begin{prop}\label{prop:domains_saturated}
   Let $s,v>0$ and let $\hat{u}\in (-s,s)^j \times \Omega^\circ_s(v)$. Suppose that $\hat{u}_i$ is a sequence in $\widehat{N}$ converging to $\hat{u}$ associated to (\ref{eq:domains_gh}) with $(Y,y)=(\mathbb{R}^j\times Z,(0,z))$. Then $\hat{u}_i\in S(m_iv)\cdot F_i(s)$ for all $i$ large.
\end{prop}

\begin{proof}
   We write $\hat{u}=(u,(w,h)\cdot \sigma(t))$, where $u\in (-s,s)^j$, $w\in (-v,v)$, $h\in K$, and $t\in [0,s)$. Let $u'_i\in {F}$ and $\gamma^{l_i}\in \Gamma$ such that $\hat{u_i}= \gamma^{l_i}\cdot u'_i$. Since $\pi(u'_i)$ converges to $(u,t)\in (-s,s)^j\times [0,s)$ associated to the bottom row of (\ref{eq:induction_notation}), by the inductive assumption (3B), $\pi(u'_i)\in D_i(s)$ for all $i$ large. Consequently, $u'_i\in F_i(s)$.

   It remains to show that $\gamma^{l_i}\in S(m_i v)$ for all $i$ large. After passing to a subsequence, we have a convergence
   $$u'_i\overset{GH}\to u'\in Y,\quad \gamma^{l_i}\overset{GH}\to (w',h')\in \mathbb{R}\times K=G$$
   associated to the top row of (\ref{eq:induction_notation}). By Claim 1 in the proof of Proposition \ref{prop:domains_gh}, $u'\in \mathbb{R}^j \times (K\cdot \sigma)$.
   Together with
   $$(w',h')\cdot u' = \hat{u} = (u,(w,h)\cdot \sigma(t)),$$
   we infer that $w'=w \in (-v,v)$. Lastly, by Proposition \ref{prop:sym_saturated}, $\gamma^{l_i}\in S(m_i v)$ for all $i$ large.
\end{proof}

Then, it follows from Proposition \ref{prop:domains_saturated} and a standard covering argument that we have the desired measure convergence (\ref{eq:induction_meas_outline}) in the outline.

\begin{cor}\label{cor:domain_meas_conv}
   Let $\meas_i$ be the renormalized measure on $(r_i^{-1}\widehat{N},\hat{q})$ (see Definition \ref{defn:renorm_meas}). After passing to a subsequence, we suppose that $\meas_i$ converges to a limit renormalized measure $\meas_Y$ on $Y$ associated to the top row of (\ref{eq:induction_intro}) with $(Y,y)=(\mathbb{R}^j\times Z,(0,z))$. Then
   $$\lim_{i\to\infty}\meas_i(S(m_iv)\cdot F_i(s))= \meas_Y([-s,s]^j \times \Omega_s(v)).$$
\end{cor}

\begin{proof}
  For convenience, below, we write
  $$\mathcal{D}_i=S(m_iv)\cdot F_i(s),\quad \mathcal{D}_\infty = [-s,s]^j \times \Omega_s(v).$$
  $\mathcal{D}_\infty$ has interior $\mathcal{D}_\infty^\circ =(-s,s)^j \times \Omega^\circ_s(v)$.

  Let $\epsilon>0$. By applying a standard Vitali covering argument to $\mathcal{D}_\infty^\circ$ (see \cite{Hein_book}), we can find a finite collection of disjoint closed balls 
  $\mathcal{B}=\{ \overline{B}_{r_\alpha} (x_\alpha)\}_{\alpha=1}^{A}$
  such that every ball of $\mathcal{B}$ is contained in $\mathcal{D}^\circ_\infty$ and
  $$\sum_{\alpha=1}^{A} \meas_Y (B_{r_\alpha}(x_\alpha)) \le \meas_Y(\mathcal{D}^\circ_\infty) \le \sum_{\alpha=1}^{A} \meas_Y (B_{r_\alpha}(x_\alpha))+\epsilon.$$
  For each ball $B_{r_\alpha}(x_\alpha)$ in $\mathcal{B}$, we choose a sequence $x_{i,\alpha}$ in $r_i^{-1}\widehat{N}$ converging to $x_\alpha$ associated to the top row of (\ref{eq:induction_notation}). Then along the sequence $r_i^{-1}\widehat{N}$,
  $$B_{(1-\epsilon)r_\alpha}(x_{i,\alpha})\overset{GH}\to \overline{B}_{(1-\epsilon)r_\alpha}(x_\alpha) \subset B_{r_\alpha}(x_\alpha) \subset \mathcal{D}_\infty^\circ.$$
  By Proposition \ref{prop:domains_saturated}, $B_{(1-\epsilon)r_\alpha}(x_{i,\alpha})$ is contained in $\mathcal{D}_i$ for all $i$ large. Moreover, from the Gromov-Hausdorff convergence, we see that for all $i$ large, the collection $\{{B}_{(1-\epsilon)r_\alpha}(x_{i,\alpha})\}_{\alpha=1}^{A}$ is disjoint.
  Hence, by Bishop--Gromov relative volume comparison,
  $$ (1-\epsilon)^n\sum_{\alpha=1}^A \meas_i({B}_{r_\alpha}(x_{i,\alpha}))\le \sum_{\alpha=1}^A \meas_i({B}_{(1-\epsilon)r_\alpha}(x_{i,\alpha}))\le \meas_i(\mathcal{D}_i).$$
  Let $i\to\infty$, we obtain
  $$\liminf_{i\to\infty} \meas_i(\mathcal{D}_i)\ge (1-\epsilon)^n \left(\sum_{\alpha=1}^{A} \meas_Y (B_{r_\alpha}(x_\alpha)) -\epsilon\right) \ge (1-\epsilon)^n (\meas_Y(\mathcal{D}^\circ_\infty)-\epsilon).$$
  Let $\epsilon\to 0$; this proves $\liminf \meas_i(\mathcal{D}_i)\ge \meas_Y(\mathcal{D}^\circ_\infty)$.

  For the other direction, we can cover $\mathcal{D}_\infty$ by a finite collection of open balls $\mathcal{B}=\{B_{r_\alpha}(x_\alpha)\}_{\alpha=1}^A$, abusing the notation, such that
  $$\sum_{\alpha=1}^A \meas_Y(B_{r_\alpha}(x_\alpha))-\epsilon \le  \meas_Y(\mathcal{D}_\infty) \le \sum_{\alpha=1}^A \meas_Y(B_{r_\alpha}(x_\alpha)).$$
  For each $\alpha$, we choose a sequence of points $x_{i,\alpha}\in r_i^{-1}\widehat{N}$ converging to $x_\alpha$. From the Gromov--Hausdorff convergence and Proposition \ref{prop:domains_gh}, the collection $\{ 
B_{(1+\epsilon)r_\alpha}(x_\alpha) \}$ is a cover of $\mathcal{D}_i$ for all $i$ large. Hence 
$$\meas_i(\mathcal{D}_i) \le \sum_{\alpha=1}^A \meas_i(B_{(1+\epsilon)r_\alpha}(x_{i,\alpha}))\le (1+\epsilon)^n \sum_{\alpha=1}^A \meas_i(B_{r_\alpha}(x_{i,\alpha})).$$
Let $i\to\infty$ and then let $\epsilon\to 0$. We obtain $\limsup \meas_i(\mathcal{D}_i)\le \meas_Y(\mathcal{D}_\infty)$.

It remains to show that $\meas_Y(\partial \mathcal{D}_\infty)=0$. Since
$$\partial \mathcal{D}_\infty = (\partial[-s,s]^j\times \Omega_s(v)) \cup ([-s,s]^j\times \partial\Omega_s(v)),\quad \meas_Y=\mathcal{L}^j\otimes \meas_Z,$$
$\meas_Z(\partial\Omega_s(v))=0$ from Lemma \ref{lem:bdry_zero_meas} below completes the proof.
\end{proof}

\begin{lem}\label{lem:bdry_zero_meas}
    We have $\meas_Z(\partial \Omega_s(v))=0$, where $\mathcal{L}^j\otimes\meas_Z=\meas_Y$ is any renormalized limit measure on $\R^j\times Z=Y$.
\end{lem}

\begin{proof}
 Observe that the boundary $\partial \Omega_s(v)$ satisfies:
  \begin{align*}
    \partial \Omega_s(v)&=A^+\cup A^-\cup I,\text{ where:}\\
    A^{\pm}&\coloneqq\{(\pm v,h)\cdot\sigma(t) \mid h\in K,t\in[0,s]\},\\
    I&\coloneqq\{(w,h)\cdot\sigma(s) \mid w\in[-v,v],h\in K\}.
  \end{align*}
Now, looking for a contradiction, we assume that $\meas_Z(A^{+})>0$. We recall that $\R=\R\times\{e\}\subset G$ acts by measure-preserving isometries on $Z$; consequently, for every $w\in\R$, we have:
\begin{equation*}
    \meas_Z(\{(w,h)\cdot\sigma(t)\mid h\in K,t\in[0,s]\})=\meas_Z((w,e)\cdot A^+)=\meas_Z(A^+)>0.
\end{equation*}
We note that by Lemma \ref{lem:horizontal_rays}(2) the set $\{(w,h)\cdot\sigma(t)\mid h\in K,t\in \R \}$ consists of horizontal rays starting from $(w,e)\cdot y$. Hence
for distinct $w\not=w' \in \R$, we have disjoint subsets
$$\{(w,h)\cdot\sigma(t)\mid h\in K,t\in[0,s]\}\cap \{(w',h)\cdot\sigma(t)\mid h\in K,t\in[0,s]\}=\emptyset;$$
otherwise, we would have two distinct minimal geodesics from $Gz$ to some $\sigma(t)$, which contradicts Lemma \ref{lem:horizontal_rays}(1).
Then it follows that
$$\meas_Z(\Omega_s(v))=\meas_Z(\{(w,h)\cdot\sigma(t)\mid w\in[-v,v],h\in K,t\in[0,s]\})=\infty,$$
which is not possible because $\meas_Z$ is Radon and $\Omega_s(v)$ is compact. The same argument also shows that $\meas_Z(A^-)=0$.    

It remains to show that $\meas_Z(I)=0$. Suppose that $\meas_Z(I)>0$. Note that all points in $I$ are in the same $G$-orbit. Together with that, $I$ has positive $\meas_Z$-measure and the regular set of $Z$ has full $\meas_Z$-measure (see \cite{BrueSemola20}), we see that every point in $I$ must be regular. In addition, since $I$ has positive $\meas_Z$-measure, $I$ admits density points w.r.t. $\meas_Z$ (see \cite[Theorem 1.8]{Hein_book}). We choose a density point $z'$ of $I$, then $z'$ satisfies
$$\lim_{t\to 0} \dfrac{\meas_Z(I\cap B_t(z'))}{\meas_Z(B_t(z'))}=1.$$
Because $I$ is homeomorphic to a manifold with boundary, we can also assume that $z'$ is in an interior point of $I$.

Let $t_i\to 0^+$ be a sequence. We consider a metric measure tangent cone of $Z$ at $z'$ and keep track of the $G$-action:
\begin{equation*}
   \begin{CD}
   (t^{-1}_iZ,z',\nu_i,G) @>GH>> (\mathbb{R}^k,0,\nu,G') \\
	@VV\pi V @VV \pi V\\
	(t^{-1}_i[0,\infty),s) @>GH>> (\mathbb{R},0)
   \end{CD},
\end{equation*}
where $\nu_i\coloneqq\frac{\meas_Z}{\meas_Z(B_{t_i}(z'))}$ are the renormalized measures and $\nu$ is (a multiple of) the Lebesgue measure $\mathcal{L}^k$ on $\mathbb{R}^k$ by Theorem \ref{thm:meas_split}. Since $\mathbb{R}^k/G'$ is isometric to a line, we can lift this line back to $\R^k$ and obtain an isometric splitting $\R^k\cong \R \times Q$, where $G'$ acts trivially on $\R$ and transitively on $Q$. Hence the orbit $G'\cdot 0=(0,Q)$ must be a linear subspace in $\mathbb{R}^k$ of codimension $1$. In particular, it holds that $\nu(G'\cdot 0)=0.$ On the other hand, because $z'$ is an interior point of $I$, there is small $r>0$ such that $I\cap B_r(z')=Gz'\cap B_r(z')$, then $I\cap B_{t_i}(z')$ converges to $(G'\cdot 0)\cap B_1(0)$ associated to the above blow-up. Applying a covering argument as in the proof of Corollary \ref{cor:domain_meas_conv}, we have $$\nu((G'\cdot0)\cap B_1(0))\ge\limsup_{i\to\infty}\nu_i(I\cap B_{t_i}(z'))=1,$$ 
which is a contradiction. Therefore, $\meas_Z(I)=0$.

In conclusion, we have $\meas_Z(\partial\Omega_s(v))=0$.
\end{proof}

\begin{rem}\label{rem:meas_D}
  Recall that the inductive assumption (3) provides a sequence of domains $D_i(s)\overset{GH}\to [-s,s]^j\times [0,s]$ associated to the bottom row of (\ref{eq:induction_notation}) with $X=\mathbb{R}^j\times [0,\infty)$. By the inductive assumption (3B) and a similar covering argument in the proof of Corollary \ref{cor:domain_meas_conv}, if we start with a limit renormalized measure $\meas_X$ on $X$ as (a multiple of) the Lebesgue measure, which is provided by inductive assumption (2), then
  $$\nu_i (D_i(s)) \to \meas_X([-s,s]^j\times [0,s])=c\cdot 2^j s^{j+1},$$
  where $\nu_i$ is the renormalized measure on $(r_i^{-1}N,q)$.
\end{rem}

We close this subsection by a lemma comparing the domains $S(m_i)\cdot F_i(1)$ to balls $B_{r_i}(\hat{q})$. We continue to assume (\ref{eq:induction_notation}) with $X=\mathbb{R}^j \times [0,\infty)$ and the inductive assumptions for $j$.

\begin{lem}\label{lem:balls_vs_domains}
   There are constants $C_1,C_2>1$ independent of $i$ such that 
   $$B_{r_i/C_1}(\hat{q}) \subseteq S(m_i)\cdot F_i(1) \subseteq B_{C_2r_i}(\hat{q})$$
   holds for all $i$ large. 
\end{lem}

\begin{proof}
  We first prove $S(m_i)\cdot F_i(1) \subseteq B_{C_2r_i}(\hat{q})$ for some constant $C_2$. Let $\gamma^{l_i} \in S(m_i)$ and $q'_i\in F_i(1)$. By the method we chose $m_i$ in Convention \ref{conv:R_subgroup}, we have
  $$d(\gamma^{l_i} \hat{q},\hat{q}) \le 2r_i.$$
  for all $i$. Also, in $X=\mathbb{R}^j\times [0,\infty)$ with $j+1\le n$, we have relation 
  $$ [-1,1]^j\times [0,1] \subseteq \overline{B_{\sqrt{j+1}}}(0)\subseteq \overline{B_{\sqrt{n}}}(0).$$
  Let $\pi:\widehat{N} \to N$ be the covering map. Recall the inductive assumption (3A)
  $$\pi(q'_i)\in D_i(1)\overset{GH}\to [-1,1]^j \times [0,1],$$
  associated to the bottom row of (\ref{eq:induction_notation}). It follows that for all $i$ large
  $$d(q'_i,\hat{q})=d(q'_i,\Gamma \hat{q})=d_N(\pi(q'_i),q)\le nr_i,$$
  Consequently,
  $$d(\gamma^{l_i} q'_i,\hat{q})\le d(\gamma^{l_i} q'_i,\gamma^{l_i} \hat{q})+d(\gamma^{l_i} \hat{q},\hat{q})\le nr_i+2r_i=(n+2)r_i.$$
  This shows $S(m_i)\cdot F_i(1) \subseteq B_{(n+2)r_i}(\hat{q})$.

  Next, we prove $ B_{r_i/C_1}(\hat{q}) \subseteq S(m_i)\cdot F_i(1)$ for some constant $C_1$. Let $q'_i\in B_{r_i/C_1}(\hat{q})$, where $C_1\ge 2$ will be determined through the proof. We choose $\gamma^{l_i}\in \Gamma$ such that
  $$d(q'_i,\gamma^{l_i} \hat{q}) = d(q'_i,\Gamma \hat{q}).$$
  Then the point $\gamma^{-l_i} q'_i$ satisfies
  $$d(\hat{q},\gamma^{-l_i} q'_i)=d(\gamma^{l_i} \hat{q},q'_i)\le d(\hat{q},q'_i)\le r_i/C_1.$$
  By construction, $\gamma^{-l_i} q'_i$ belongs to the (closed) Dirichlet domain $F$ centered at $\hat{q}$ because $\hat{q}\in \Gamma \hat{q}$ realizes the distance between $\gamma^{-l_i} q'_i$ and $\Gamma \hat{q}$. Associated to the bottom row of (\ref{eq:induction_notation}) with $(X,x)=(\mathbb{R}^j \times [0,\infty),0)$, we have
  $$B_{r_i/C_1}(q)\overset{GH}\to \overline{B_{1/C_1}}(0)\subset (-1,1)^j\times [0,1).$$
  By inductive assumption (3B), $\pi(\gamma^{-l_i}q'_i)\in D_i(1)$ for all $i$ large. Hence $\gamma^{-l_i}q'_i\in F_i(1)$ for all $i$ large.

  Also, the group element $\gamma^{l_i}$ satisfies
  $$d(\gamma^{l_i} \hat{q},\hat{q})\le d(\gamma^{l_i} \hat{q},q'_i)+d(q'_i,\hat{q})\le 2r_i/C_1.$$
  By Lemma \ref{lem:larger_power_dist}, if we require $C_1\ge 6C$, where $C$ is the constant from Lemma \ref{lem:boomer_control}, then the inequality
  $$d(\gamma^{l_i} \hat{q},\hat{q}) \le 2r_i/C_1 \le r_i/(3C)$$
  implies $l_i\le m_i$ for all $i$ large by Lemma \ref{lem:larger_power_dist}. To sum up, we choose $C_1=6C$, then $q'=\gamma^{l_i}(\gamma^{-l_i} q')$ belongs to $S(m_i)\cdot F_i(1)$.
\end{proof}

\subsection{Proof of the Induction Theorem}\label{subsec:induction}

In this subsection, we prove the inductive step in Theorem \ref{thm:induction_plus} by assuming the plane/halfplane rigidity (Theorem \ref{thm:plane_halfplane_rigid}). This completes the proof of Theorem \ref{thm:induction_plus} modulo Theorem \ref{thm:plane_halfplane_rigid}.

\begin{proof}[Proof of the inductive step in Theorem \ref{thm:induction_plus}]

We assume that Theorem \ref{thm:induction_plus} holds for some $j$, and we prove the statement for $j+1$.

Let $r_i\to\infty$ be a sequence. After passing to a subsequence if necessary, we consider the convergence (\ref{eq:induction_notation}). By the inductive assumption, we know that $X$ is isometric to Euclidean $\mathbb{R}^j \times [0,\infty)$ or Euclidean $\mathbb{R}^{j+1}$. Propositions \ref{prop:non_max_escape} and \ref{prop:orbit_R} imply that the orbit $Gy$ is homeomorphic to $\mathbb{R}$; in particular, we can write $G=\mathbb{R}\times K$, where $K$ fixes $y$.

If $X$ is isometric to Euclidean $\mathbb{R}^{j+1}$, then we can lift $j+1$ many linearly independent lines to $Y$. By splitting theorem, $(Y,y)$ is isometric to $(\mathbb{R}^{j+1} \times Z,(0,z))$, where $Z$ is $\RCD(0,n-j-1)$; moreover, $G$ acts trivially on $\mathbb{R}^{j+1}$-factor. Since the orbit $Gy=\{0\}\times Z$ is homeomorphic to $\mathbb{R}$, we conclude that $Z$ is a line. Hence $Y$ is isometric to $\mathbb{R}^{j+2}$. By Theorem \ref{thm:meas_split}, $Y$ must be equipped with (a multiple of) the Lebesgue measure as any limit renormalized measure. This completes the inductive step when $X$ is isometric to Euclidean space $\mathbb{R}^{j+1}$.

Below, we assume that $X$ is isometric to the Euclidean halfspace $\mathbb{R}^j \times [0,\infty)$. In this case, $Y$ also splits isomorphically as
$$(Y,y,\meas_Y)=(\mathbb{R}^j,0,\mathcal{L}^j)\otimes (Z,z,\meas_Z)$$
for any limit renormalized measure $\meas_Y$ by Theorem \ref{thm:meas_split}. Moreover,
$G$ acts trivially on the $\mathbb{R}^j$-factor.

We check that $(Z,z,d_Z,\meas_Z)$ fulfills the conditions of Theorem \ref{thm:plane_halfplane_rigid} for some limit renormalized measure $\meas_Y=\mathcal{L}^j \otimes \meas_Z$. Verifying condition (1) of Theorem \ref{thm:plane_halfplane_rigid} is straightforward. In fact, we have 
$$\mathbb{R}^j\times [0,\infty) =X=Y/G = \mathbb{R}^j \times (Z/G).$$
Hence $Z/G$ is isometric to $[0,\infty)$ with $\bar{z}$ projecting to $0$. Moreover, we have $G=\mathbb{R}\times K$, where $K$ fixes $y$, as we discussed before.

Next, we check $(Z,z,d_Z,\mathfrak{m}_Z)$ satisfies condition (2) of Theorem \ref{thm:plane_halfplane_rigid}. 

By the inductive assumption (2), $X=\mathbb{R}^j\times [0,\infty)$ has a limit renormalized measure as (a multiple of) the Lebesgue measure. We pass to this subsequence to obtain $\meas_X$ as that. Let $s>0$. Inductive assumption (3) provides a sequence of domains $D_i(s)\subset N$ such that
$$(r_i^{-1}D_i(s),q)\overset{GH}\longrightarrow ([-s,s]^j \times [0,s],0)$$
associated to the bottom row of (\ref{eq:induction_notation}) with $(X,x)=(\mathbb{R}^j\times [0,\infty),0)$. Then we have
\begin{equation}\label{eq:vol_D}
\dfrac{\vol D_i(s)}{\vol D_i(1)}=\dfrac{\vol D_i(s)}{\vol B_{r_i}(q)}\dfrac{\vol B_{r_i}(q)}{\vol D_i(1)} \to \dfrac{\meas_X([-s,s]^j\times [0,s])}{\meas_X([-1,1]^j \times [0,1]) } = s^{j+1}
\end{equation}
as $i\to\infty$ (see Remark \ref{rem:meas_D}).

Let $F\subseteq \widehat{N}$ be the Dirichlet domain of $\Gamma$-action centered at $\hat{q}$. We consider $F_i(s) := { F \cap \pi^{-1}(D_i(s)) }$ as in Section \ref{subsec:cube}. As a property of the Dirichlet domain, it holds that
\begin{equation}\label{eq:vol_D_F}
\vol F_i(s) = \vol D_i(s).
\end{equation}

Let $\gamma\in \Gamma$ be a generator of $\Gamma \simeq \mathbb{Z}$. With respect to the top row of (\ref{eq:induction_notation}), we choose a sequence $\gamma^{m_i}$, where $m_i\to\infty$ as in Convention \ref{conv:R_subgroup}. By Lemma \ref{lem:balls_vs_domains}, we have inequality
$$ \vol B_{r_i/C_1}(\hat{q}) \le \vol \left(S(m_i)\cdot F_i(1)\right) \le \vol B_{C_2r_i}(\hat{q})$$
for all $i$ large, where $C_1,C_2>1$ are constants independent of $i$. Together with Bishop--Gromov relative volume comparison, we derive
$$ 1/C_1^n \le \dfrac{\vol(S(m_i)\cdot F_i(1)) }{\vol B_{r_i}(\hat{q})} \le C_2^n.$$
Passing to a subsequence, we obtain
\begin{equation}\label{eq:vol_cube_1}
\dfrac{\vol(S(m_i)\cdot F_i(1)) }{\vol B_{r_i}(\hat{q})} \to \theta \in [1/C_1^n,C_2^n].
\end{equation}

Let $v>0$. According to Proposition \ref{prop:domains_gh}, we have convergence of subsets
\begin{equation}\label{eq:induction_domains_gh}
   (r_i^{-1}(S(m_i v)\cdot F_i(s)),\hat{q}) \overset{GH}\to ([-s,s]^j\times \Omega_s(v),y)
\end{equation}
associated to the top row of (\ref{eq:induction_notation}) with $Y=\mathbb{R}^j \times Z$. Combining Corollary \ref{cor:domain_meas_conv} with (\ref{eq:vol_D}), (\ref{eq:vol_D_F}), and (\ref{eq:vol_cube_1}), this allows us to calculate $\meas_Z(\Omega_s(v))$ by
\begin{align*}
&(2s)^j\cdot \meas_Z(\Omega_s(v)) =(\mathcal{L}^j\otimes\meas_Z)([-s,s]^j\times \Omega_s(v))\\
=&\lim_{i\to\infty}\dfrac{\vol (S(m_i v)\cdot F_i(s))}{\vol B_{r_i}(\hat{q})}=\lim_{i\to\infty} \dfrac{\vol (S(m_i v)\cdot F_i(s))}{\theta^{-1} \vol (S(m_i)\cdot F_i(1))}\\
=&\ \theta \cdot \lim_{i\to\infty} \dfrac{\# S(m_i v) \cdot \vol D_i(s)}{\# S(m_i) \cdot \vol D_i(1)}= \theta \cdot vs^{j+1}.
\end{align*}
Hence $\meas_Z(\Omega_s(v))=c\cdot vs$, where $c=\theta/2^j$ is a constant independent of $s$ and $v$. This verifies condition (2) in Theorem \ref{thm:plane_halfplane_rigid} for $(Z,z,d_Z,\meas_Z)$.

Now we can finally apply the plane/halfplane rigidity (Theorem \ref{thm:plane_halfplane_rigid}) to $Z$ and conclude that $Y=\mathbb{R}^j \times Z$ is isometric to either a Euclidean space $\mathbb{R}^{j+2}$ or a Euclidean halfspace $\mathbb{R}^{j+1} \times [0,\infty)$ with (a multiple of) the Lebesgue measure as a limit renormalized measure. This completes the proof of (1) and (2) in the inductive step.

To complete the inductive step, it remains to verify Theorem \ref{thm:induction_plus}(3) when $Y=\mathbb{R}^{j}\times Z = \mathbb{R}^j \times (\mathbb{R} \times  [0,\infty))$. In this case, $G=\mathbb{R}$ acts as pure translations in the $\mathbb{R}$-factor of $Z$. Hence, it is clear that 
$$
\Omega_s(v)=[-v,v]\times [0,s] \subseteq \mathbb{R}\times [0,\infty)=Z.
$$
By Proposition \ref{prop:domains_gh}, we can choose $v=s$ and $\widehat{D}_i(s):=S(m_i s)\cdot F_i(s)$, then 
$$(r_i^{-1}\widehat{D}_i(s),\hat{q})\overset{GH}\longrightarrow ([-s,s]^{j+1}\times [0,s],0).$$
This proves Theorem \ref{thm:induction_plus}(3A) for $j+1$.
Lastly, applying Proposition \ref{prop:domains_saturated} with $v=s$, we conclude Theorem \ref{thm:induction_plus}(3B) for $j+1$.
 
We complete the inductive step for Theorem \ref{thm:induction_plus}.
\end{proof}

\begin{rem}\label{rem:limit_G_action}
   We give a remark derived from the proof of Theorem \ref{thm:induction_plus}. In the inductive step, if $N=\widehat{M}_{k-j}$ has a (unique) asymptotic cone $\mathbb{R}^{j+1}$, then $\widehat{N}=\widehat{M}_{k-(j+1)}$ also has a (unique) asymptotic cone $\mathbb{R}^{j+2}$; in this case, we have $G=\mathbb{R}$ acting as pure translations in (\ref{eq:induction_notation}). If $N$ has a (unique) asymptotic cone $\mathbb{R}^{j}\times [0,\infty)$, then in (\ref{eq:induction_notation}),
   $$ G=
   \begin{cases}
      \mathbb{R} , & \text{ if } (Y,y)=(\mathbb{R}^{j+1}\times [0,\infty),0)\\
      \mathbb{R}\times \mathbb{Z}_2, & \text{ if } (Y,y)=(\mathbb{R}^{j+2},0),
   \end{cases}$$
   where $\mathbb{Z}_2$ acts as a reflection across a line through $0$. This $\mathbb{Z}_2$ isotropy subgroup appears at most once through all the inductive steps. In fact, once it appears, the covering must have asymptotic cone $\mathbb{R}^{j+2}$. Afterwards, we will always have $G=\mathbb{R}$ in the remaining inductive steps.
\end{rem}

\begin{rem}
  We give an example of the aforementioned $\mathbb{Z}_2$ isotropy subgroup. Let $M^2$ be the flat (noncompact) M\"obius band. It has linear volume growth and $\pi_1(M)=\mathbb{Z}$. Its universal cover has a unique equivariant asymptotic cone
  $$(r_i^{-1} \widetilde{M},\tilde{p},\mathbb{Z})\overset{GH}\longrightarrow (\mathbb{R}^2,0,G=\mathbb{R}\times \mathbb{Z}_2).$$
\end{rem}

\subsection{Structure of fundamental groups}\label{subsec:pi1}

In this subsection, we prove Theorems \ref{thm:vir_abel}, \ref{thm:finite}, and \ref{thm:rcd_abel} by applying the Induction Theorem. We start with Theorem \ref{thm:vir_abel}.

\begin{thm}\label{thm:zero_escape_rate}
  Let $(M,p)$ be an open $n$-manifold with $\Ric\ge 0$ and linear volume growth. Then $E(M,p)=0$.
\end{thm}

\begin{proof}
   If $M$ has a (unique) asymptotic cone $(\mathbb{R},0)$. Then, by Theorem \ref{thm:sormani}, $M$ splits as a metric product $M=\mathbb{R}\times N$, where $N$ is compact. In this case, $E(M,p)=0$ holds clearly because all minimal geodesic loops representing elements in $\pi_1(M,p)$ are contained in a slice of $N$, which is bounded.

   Below we assume that $M$ has a (unique) asymptotic cone $([0,\infty),0)$. We first consider the case that $\pi_1(M)$ is torsion-free nilpotent. Let $(\widetilde{M},\tilde{p})$ be the universal cover of $(M,p)$ with covering group $\Gamma=\pi_1(M,p)$. By the Induction Theorem, for any $r_i\to\infty$, we have asymptotic cones
   \begin{equation*}
   \begin{CD}
   (r_i^{-1} \widetilde{M},\tilde{p},\Gamma) @>GH>> (Y,y,G) \\
	@VV\pi V @VV \pi V\\
	(r_i^{-1} M,p) @>GH>> ([0,\infty),0)=(Y/G,\bar{y}),
   \end{CD}
   \end{equation*}
   where $(Y,y)$ is either Euclidean space $(\mathbb{R}^{k+1},0)$ or Euclidean halfspace $(\mathbb{R}^k\times [0,\infty))$. In either case, $Gy$ is a Euclidean $\mathbb{R}^k \times \{0\}$ inside $Y$. In fact, one can see this from $Gy=\pi^{-1}(0)$ as the successive pre-image of the quotient map $Y_{k-(j+1)}\to Y_{k-j}$, where $Y_{k-j}$ is the asymptotic cone defined in (\ref{eq:induction}). Thanks to \cite[Theorem 1.3]{Pan21}, we conclude that $E(M,p)=0$. 

   In general, $\Gamma=\pi_1(M,p)$ has a torsion-free nilpotent subgroup $\Gamma'$ of finite index by Theorem \ref{thm:finite_gen} and Remark \ref{rem:vir_nil_torsionfree}. For any equivariant asymptotic cone
   \begin{equation}\label{eq:gh_conv_index}(r_i^{-1}\widetilde{M},\tilde{p},\Gamma,\Gamma')\overset{GH}\longrightarrow (Y,y,G,G'),
   \end{equation}
   we saw in the last paragraph that $G'y$ is Euclidean in $Y$. We claim $Gy=Gy'$. If this claim holds, then it follows that $Gy$ is Euclidean and $E(M,p)=0$ by \cite[Theorem 1.3]{Pan21}.

   We verify the claim that $Gy=G'y$. Since $\Gamma'$ has finite index in $\Gamma$, there are finitely many elements $\gamma_0=e,\gamma_1,...,\gamma_k$ in $\Gamma$ such that $\Gamma=\cup_{j=0}^k \Gamma' \gamma_j$. Passing to a subsequence, we assume that every $\gamma_j$ converges to some $g_j \in G$ associated to (\ref{eq:gh_conv_index}). We note that $g_j y=y$ by construction. For any $h\in G-G'$, let $\alpha_i \in \Gamma$ be a sequence converging to $h$ associated to (\ref{eq:gh_conv_index}). We write $\alpha_i=\beta_i\cdot \gamma_{j(i)}$, where $\beta_i\in \Gamma'$. We pass to a subsequence and assume $j(i)$ is a constant; below, we denote $\gamma_{j(i)}$ by $\gamma_1$ for convenience. Then $\beta_i=\alpha_i \cdot \gamma_1^{-1}$ has limit $hg_1^{-1}\in G'$. It follows that
   $$ hy= hg_1^{-1}(g_1 y)=hg_1^{-1} y \in G'y.$$
   This verifies the claim and completes the proof.
\end{proof}

%Next, we prove Theorem \ref{thm:vir_abel}.

\begin{proof}[Proof of Theorem \ref{thm:vir_abel}]
    We have shown that $E(M,p)=0$ in Theorem \ref{thm:zero_escape_rate}. Now we apply \cite[Theorem 1.2]{Pan21} to conclude that $\pi_1(M)$ is virtually abelian. Because $\pi_1(M)$ is finitely generated and virtually abelian, it contains a $\Z^k$-subgroup of finite index. It follows from a result of Anderson \cite[Theorem 1.1]{Anderson90} that $k\le n-1$.
\end{proof}

\begin{rems}
   (1) Alternatively, one can first assume that $\pi_1(M)$ is finitely generated and torsion-free nilpotent without loss of generality, thanks to Theorem \ref{thm:finite_gen} and Remark \ref{rem:vir_nil_torsionfree}. Then \cite[Theorem 1.2]{Pan21} and the first two paragraphs in the proof of Theorem \ref{thm:zero_escape_rate} are sufficient to conclude Theorem \ref{thm:vir_abel}.\\
   (2) One may also prove Theorem \ref{thm:vir_abel} without directly applying the main result in \cite[Theorem 1.2]{Pan21}. In fact, the difficult part of \cite[Theorem 1.2]{Pan21} is showing that any equivariant asymptotic cone $(Y,y,G)$ of the universal cover has orbit $Gy$ isometric to a Euclidean subspace in $Y$ from $E(M,p)=0$. Deducing the virtual abelianness from this property of Euclidean orbits is much less involved (see the last page of \cite[Section 4]{Pan21}). Here, we have already shown that the orbit $Gy$ is always Euclidean in the proof of Theorem \ref{thm:zero_escape_rate}, so one can establish the virtual abelianness from there. 
\end{rems}

Theorem \ref{thm:vir_abel} extends to noncompact $\RCD(0,N)$ spaces with linear volume growth, as we stated in Theorem \ref{thm:rcd_abel}.

\begin{proof}[Proof of Theorem \ref{thm:rcd_abel}]
    The proof is essentially the same as the smooth case. Below, we point out where we used the smooth structure of $M$ and the corresponding results in $\RCD(0,N)$ spaces that we needed.

    First of all, Wang proved that any $\RCD(0,N)$ space is semi-locally simply connected \cite{WangJK24}, so the fundamental group acts as the covering group of the universal cover, just like the case of manifolds. Moreover, its universal cover admits a natural $\RCD(0,N)$ structure by the work of Mondino-Wei \cite{MondinoWei19} such that the fundamental group acts by measure-preserving isometries.

    In the smooth case, we started with Sormani's results \cite{Sor00b} on the asymptotic structure of $M$ with linear volume growth and the finite generation of $\pi_1(M)$. It is well-established that these results hold in the RCD setting as well. In fact, Huang \cite{Huang18,Huang20} showed that $\RCD(0,N)$ with linear volume growth has sublinear diameter growth. Hence $\pi_1(X)$ is finitely generated \cite{KL15,Qian22}. Then we can apply Milnor's argument \cite{Milnor68} and the volume comparison on the universal cover $\widetilde{X}$ to deduce that $\pi_1(X)$ has polynomial growth. By Gromov's result on groups with polynomial growth \cite{Gromov81}, we conclude that $\pi_1(X)$ is virtually nilpotent. After further passing to a subgroup of finite index, we can assume that $\pi_1(X)$ is torsion-free nilpotent. With this, we can consider the tower of $\Z$-covering spaces as in the Induction Theorem.
    
    For the limit volume ratio Theorem \ref{thm:volume_ratio_limit}, we provide a proof for $\RCD(0,N)$ spaces in the Appendix; see Proposition \ref{prop:rcd_vol_ratio}. The equivariant convergence theory with non-maximal escape rate used in Section \ref{subsec:orbit_R} (\cite[Theorem 1.3]{Pan21} and \cite[Proposition C(1)]{Pan23}, or Proposition \ref{prop:orbit_R} of this paper) extends verbatim to $\RCD(0,N)$ spaces because smooth structures are not used in their proofs.
     
    In the proof of the Induction Theorem, we used the Dirichlet domain and the fact that its boundary has zero volume. The Dirichlet domain in any $\RCD$ space $(Y,\dist,\meas)$ has a zero $\meas$-measure boundary as well; see \cite[Theorem 1.7(1)]{Ye23} for a proof.
    
    Lastly, the rank estimate $k\le \lfloor N \rfloor -1$ follows from Anderson's volume argument \cite[Theorem 1.1]{Anderson90} and the fact that the universal cover $\widetilde{X}$ has polynomial volume growth of order at most $N$.
\end{proof}

%\begin{rem}\label{rem:vir_abel}
%   Sormani's results on manifolds with nonnegative Ricci curvature and linear volume growth were extended to $\RCD(0,N)$ spaces with linear volume growth by Huang \cite{Huang18,Huang20}. In particular, Theorem \ref{thm:sormani} extends to the $\RCD(0,N)$ case.\\
%   (3) Sormani's work \cite{Sor00b} extends immediately to $\RCD(0,N)$ spaces by using the Abresch-Gromoll excess estimate in $\RCD(0,N)$ spaces \cite{GM14} (see \cite{KL15,Qian22}). Together with the result of sublinear diameter growth in \cite[Theorem 1.2]{Huang20}, we deduce that any $\RCD(0,N)$ space with linear volume growth has a finitely generated fundamental group.\\
%   (4) Theorem \ref{thm:volume_ratio_limit} on the limit ratio extends to $\RCD(0,N)$ spaces as well. More precisely, the limit $\lim_{r\to \infty} \meas(B_r(y))/r $ always exists if a noncompact $\RCD(0,N)$ space $(Y,y,\dist,\meas)$ has linear volume growth. See Remark \ref{rem:rcd_linear_ratio} for more details on this. \\ 
%\end{rem}

Next, we move on to the proof of Theorem \ref{thm:finite}. We will consider a $\mathbb{Z}$-folding cover of $M$ in the proof. We start with a result that immediately follows from the Induction Theorem.

\begin{cor}\label{cor:Z_cover}
  Let $(M,p)$ be an open $n$-manifold with $\Ric\ge 0$ and linear volume growth. Let $(\widehat{M},\hat{p})$ be a $\mathbb{Z}$-folding cover of $(M,p)$. Then $(\widehat{M},\hat{p},\mathbb{Z})$ has a unique equivariant asymptotic cone as one of\\
  (1) $(\mathbb{R}^2,0,\mathbb{R}\times \mathbb{Z}_2)$,\\
  (2) $(\mathbb{R}^2,0,\mathbb{R})$,\\
  (3) $(\mathbb{R}\times [0,\infty),0,\mathbb{R} )$. 
\end{cor}

\begin{proof}
   The Induction Theorem with $k=j=1$ and Remark \ref{rem:limit_G_action} together imply that any equivariant asymptotic cone of $(\widehat{M},\Z)$ must be one of the spaces listed above. The uniqueness follows from the connectedness of the set of all equivariant asymptotic cones of $(\widehat{M},\Z)$ under the equivariant pointed Gromov-Hausdorff topology (see, for example, \cite[Proposition 2.1]{Pan19}).
\end{proof}

%\begin{rem}
%  If we assume $M$ has positive Ricci curvature in Corollary \ref{cor:Z_cover}, then only (1) and (3) can occur, because $M$ has a (unique) asymptotic cone $([0,\infty),0)$ by Theorem \ref{thm:sormani}.
%\end{rem}

For Lemma \ref{lem:almost_transl} and Proposition \ref{prop:Z_cover_vol} below, we always assume that $(M,p)$ is an open $n$-manifold with nonnegative Ricci curvature and linear volume growth. We study $(\widehat{M},\hat{p})$, a $\mathbb{Z}$-folding cover of $(M,p)$, by using Corollary \ref{cor:Z_cover}. Our goal is to show that $\widehat{M}$ has roughly quadratic volume growth (Proposition \ref{prop:Z_cover_vol}) so that we can apply Anderson's results \cite{Anderson90} on positive Ricci curvature and at most cubic volume growth later to prove Theorem \ref{thm:finite}. The method to prove Proposition \ref{prop:Z_cover_vol} is similar to some of the arguments in \cite[Section 6.2]{Pan23}.

Below we denote $\gamma$ a generator of the covering group $\Gamma \simeq \mathbb{Z}$ and $|\gamma^m|:=d(\gamma^m \hat{p},\hat{p})$, where $m\in \mathbb{Z}$.

\begin{lem}\label{lem:almost_transl}
   For any $\epsilon>0$, there is a constant $R_0=R_0(\epsilon,\widehat{M},\hat{p},\gamma)$ such that for every $\gamma^m \in \Gamma$ with $|\gamma^m|\ge R_0$ and $m\in \mathbb{Z}_+$, it holds that
   $$|\gamma^m| \ge 2^{1-\epsilon} |\gamma^{\lceil m/2 \rceil}|,$$
   where $\lceil \cdot \rceil$ is the ceiling function.
\end{lem}

\begin{proof}
   We argue by contradiction. Suppose the contrary; then there is a sequence $m_i\to\infty$ such that
   $$|\gamma^{m_i}|< 2^{1-\epsilon} |\gamma^{\lceil m_i/2 \rceil}|.$$
   The sequence $r_i=|\gamma^{m_i}|\to \infty$ gives rise to
   a convergent sequence:
   \begin{equation}\label{eq:almost_transl}(r_i^{-1}\widehat{M},\hat{p},\Gamma,\gamma^{m_i},\gamma^{\lceil m_i/2 \rceil})\overset{GH}\longrightarrow (Y,y,G,g,h),
   \end{equation}
   where $(Y,y,G)$ is one of the spaces described in Corollary \ref{cor:Z_cover}. Moreover,
   $$1=d(gy,y)\le 2^{1-\epsilon} d(hy,y) < 2 d(hy,y).$$
   
   We show that $h^2 y=gy$. In fact, we write $G=\mathbb{R}\times K$ as in subsection \ref{subsec:orbit_R}, where $K$ fixes $y$. Associated to (\ref{eq:almost_transl}), we have 
   $$  \gamma\overset{GH}\to (e,\beta)\in G=\R\times K.$$ 
   Since
   $$2\cdot \lceil m_i/2 \rceil - m_i\in \{0,1\},\quad \gamma^{2\lceil m_i/2 \rceil}\overset{GH}\to  h^2,\quad \gamma^{m_i}\overset{GH}\to g,$$
   we conclude that $g$ and $h^2$ differ at most an isotropy element $\beta$. Hence $h^2 y = gy$.

   This allows us to further write $g=(1,\alpha)$ and $h=(1/2, \beta)$, where $\alpha,\beta \in K$. According to Corollary \ref{cor:Z_cover}, the orbit $Gy$ is a line in $Y$. Thus, we have $1=d(gy,y)=2 d(hy,y);$ a contradiction to $1<2 d(hy,y)$.
\end{proof}

\begin{prop}\label{prop:Z_cover_vol}
   $(\widehat{M},\hat{p})$ has volume growth
   $$\lim\limits_{r\to\infty} \dfrac{\vol B_r(\hat{p})}{r^{2+\epsilon}}=0$$
   for all $\epsilon>0$.
\end{prop}

\begin{proof}
   Let $\epsilon>0$. We choose a large integer $P_0$ such that $|\gamma^m|\ge R_0$ for all $m\ge P_0$, where $R_0=R_0(\epsilon,\widehat{M},\hat{p},\gamma)$ is the constant in Lemma \ref{lem:almost_transl}.

   We claim that there is a constant $C_1>0$ such that $|\gamma^m| \ge C_1 m^{1-\epsilon}$ for all $m\ge P_0$. Let $m\ge P_0$. Lemma \ref{lem:almost_transl} implies
   $$|\gamma^m| \ge 2^{1-\epsilon} |\gamma^{\lceil m/2 \rceil}|.$$
   If $\lceil m/2 \rceil < P_0$, we stop here. Otherwise, we apply Lemma \ref{lem:almost_transl} again to derive
   $$|\gamma^m| \ge 4^{1-\epsilon} |\gamma^{\lceil\lceil m/2 \rceil/2\rceil}|.$$
   Repeating this process,  we obtain
   $$|\gamma^m| \ge (2^k)^{1-\epsilon} |\gamma^{\lceil...\lceil m/2\rceil/2.../2\rceil}|,$$
   where $\lceil\cdot /2 \rceil$ is composed $k$ times on the right hand side so that $$\lceil...\lceil m/2\rceil/2.../2\rceil<P_0$$
   for the first time. Let $r_0=\min_{m\in \mathbb{Z}_+}|\gamma^m|>0$ and note that
   $$\dfrac{m}{2^k} \le \lceil...\lceil m/2\rceil/2.../2\rceil \le P_0,$$
   we end in 
   $$|\gamma^m| \ge (2^k)^{1-\epsilon} r_0 = \dfrac{r_0}{P_0^{1-\epsilon}} \cdot m^{1-\epsilon} = C_1 m^{1-\epsilon}.$$
   This verifies the claim with $C_1=r_0/(P_0^{1-\epsilon})$.

   Applying the claim, we see that for $r$ in an interval $[C_1m^{1-\epsilon}, C_1 (m+1)^{1-\epsilon})$, where $m\ge P_0$, the ball $B_r(\hat p)$ contains at most $2m+1$ many orbit points in $\Gamma\hat{p}$. We denote $$\Gamma(r)=\{ \gamma^m\in \Gamma \mid |\gamma^m|\le r \},\quad \frac{1}{1-\epsilon}=1+\delta(\epsilon).$$ 
   Then
   $$\# \Gamma(r) \le 2m+1 \le \dfrac{2}{C_1^{1+\delta(\epsilon)}} r^{1+\delta(\epsilon)} +1 \le C_2 r^{1+\delta(\epsilon)}$$
   for some constant $C_2$ (which depends on $\epsilon$).

   To complete the desired volume growth estimate, we note that
   $$B_r(\hat{p}) \subseteq \Gamma(2r)\cdot (F\cap \pi^{-1}(B_r(p))),$$
   where $F$ is the Dirichlet domain centered at $\hat{p}$. Together with the linear volume growth of $M$, we have 
   $$\vol B_r(\hat{p})\le \#\Gamma(2r)\cdot \vol B_r(p) \le C_3 r^{2+\delta(\epsilon)},$$
   for some constant $C_3$ (which depends on $\epsilon$). Therefore, for any $\eta>0$, we choose $\epsilon>0$ small such that $\delta(\epsilon)<\eta$, then
   $$\lim_{r\to\infty} \dfrac{\vol B_r(\hat{p})}{r^{2+\eta}}\le \lim_{r\to\infty} \dfrac{C_3 r^{2+\delta(\epsilon)}}{r^{2+\eta}}=0.$$
\end{proof}

\begin{rem}
   Inductively, one can similarly show that $\widehat{M}_{k-j}$ in (\ref{eq:tower_covers}) has volume growth
   $$\lim\limits_{r\to\infty} \dfrac{\vol B_r(\hat{p}_{k-j})}{r^{j+1+\epsilon}}=0$$
   for all $\epsilon>0$, where $j=1,...,k$.
\end{rem}

To prove Theorem \ref{thm:finite}, we use the positive Ricci curvature condition and the volume growth of the $\mathbb{Z}$-folding cover derived in Proposition \ref{prop:Z_cover_vol}. This requires two results by Anderson \cite{Anderson90}, where he proved $b_1(M^n)\le n-3$ for open manifolds with positive Ricci curvature.

\begin{thm}\cite[Theorem 2.1]{Anderson90}\label{thm:anderson_no}
   Let $N$ be a complete oriented Riemannian manifold of at most cubic volume growth, that is,
   there is a constant $C$ such that $\vol B_r(\hat{p})\le Cr^3$ for all $r$ large. If $N$ has positive Ricci curvature, then there is no complete area-minimizing hypersurface $S$ which is a boundary of least area in $N$.
\end{thm}

\begin{thm}\cite[Lemma 2.2]{Anderson90}\label{thm:anderson_yes}
   Let $N$ be a complete Riemannian manifold with finitely generated homology $H_1(N,\mathbb{Z})$. Then any non-zero line $\mathbb{R}\cdot \alpha$, where $\alpha \in H_1(N,\mathbb{R})$, gives rise to a complete homologically
area-minimizing hypersurface $S_\alpha$, which is a boundary of least area in a $\mathbb{Z}$-folding cover $\widehat{N}\to N$.
\end{thm}

\begin{proof}[Proof of Theorem \ref{thm:finite}]
   Suppose that there is a complete manifold $M$ with positive Ricci curvature, linear volume growth, and an infinite fundamental group. Passing to a finite cover if necessary, we can assume that $M$ is oriented and $\Gamma=\pi_1(M)$ is finitely generated and torsion-free nilpotent. Then $H_1(M,\Z)=\Gamma/[\Gamma,\Gamma]$ is an abelian group of rank at least $1$ (see \cite[page 13, Corollary 9]{book_Segal}). Theorem \ref{thm:anderson_yes} implies that there is a complete area-minimizing hypersurface $S$, which is a boundary of least area in $\widehat{M}$, a $\mathbb{Z}$-folding cover of $M$. On the other hand, Proposition \ref{prop:Z_cover_vol} implies that $\widehat{M}$ has at most cubic volume growth. Hence, such an area-minimizing hypersurface $S$ should not exist on $\widehat{M}$ according to Theorem \ref{thm:anderson_no}. This contradiction shows that the fundamental group cannot be infinite.
\end{proof}

%\begin{rem}\label{rem:nonnegative_split}
%   Anderson's result \cite[Theorem 2.1]{Anderson90}, whose first part is cited as Theorem \ref{thm:anderson_no} in this paper, has a second part: if $N$ has $\Ric\ge 0$, $\sec\le b<\infty$, and such a complete area-minimizing hypersurface $S$ exists, then $S$ is smooth and the universal cover $\widetilde{N}$ splits isometrically as $\widetilde{S}\times \mathbb{R}$. If one could improve this result by dropping the sectional curvature upper bound condition, then Question \ref{quest:torsion_free_line} would have an affirmative answer.
%\end{rem}

 \section{Plane/halfplane rigidity for \texorpdfstring{$\RCD(0,N)$}{RCD(0,N)} spaces}\label{sec:rigid}
We will derive the plane/halfplane rigidity (Theorem \ref{thm:plane_halfplane_rigid}) from the two propositions below.

\begin{prop}\label{prop:rigid_iso}
Let $(Y,y,d,\mathfrak{m})$ be an $\RCD(0,N)$ space ($N\ge 1$). Suppose that:\\
(1) $Y$ is polar at $y$, that is, for every point $z\in Y-\{y\}$, there is a ray emanating from $y$ through $z$.\\
(2) there is $c>0$ such that $\mathfrak{m}(B_r(y))=cr$ for all $r>0$.\\
Then $(Y,y,d,\mathfrak{m})$ is isomorphic (up to a constant) to a line $(\mathbb{R},0)$ or a ray $([0,\infty),0)$ with Lebesgue measure.
\end{prop}

\begin{prop}\label{prop:rigid_free}
Let $(Y,y,d,\mathfrak{m})$ be an $\RCD(0,N)$ space ($N\ge1$). Suppose that:\\
  (1) $Y$ has a measure-preserving isometric $G$-action, where $G=\mathbb{R}$ is closed in $\mathrm{Isom}(Y)$, such that the quotient $(Y/G,\bar{y})$ is isometric to a ray $([0,\infty),0)$.\\
  (2) there is $c>0$ such that $\mathfrak{m}(\Omega_r(v))=crv$ for all $r,v\ge 0$.\\
Then $(Y,y,d,\mathfrak{m})$ is isomorphic (up to a constant) to a Euclidean halfplane $(\mathbb{R}\times [0,\infty),0)$ with Lebesgue measure.
\end{prop}

\begin{rem}\label{rem:N=1}
The case $N=1$ of Propositions \ref{prop:rigid_iso} and \ref{prop:rigid_free} follows immediately by \cite[Corollary~1.2]{Kitabeppu-Lakzian_16}, which asserts that any $\RCD(0,1)$ space which is not a point is either a line, a ray, a circle, or a compact interval equipped with a multiple of the Lebesgue measure.

\noindent Under the assumptions of Proposition \ref{prop:rigid_iso}, the space is
noncompact, and the conclusion therefore follows immediately. On the other hand, the assumptions of Proposition \ref{prop:rigid_free} cannot occur when $N=1$. Indeed, the quotient assumption implies that $Y$ is noncompact, so $Y$ must be a line or a ray. A ray admits no nontrivial isometric $\R$-action, whereas every nontrivial isometric $\R$-action on a line is transitive and therefore has a one-point quotient. Thus, Proposition \ref{prop:rigid_free} is vacuous when $N=1$.

\noindent Consequently, throughout the proofs of Propositions \ref{prop:rigid_iso} and \ref{prop:rigid_free}, we may assume
that $N>1$.
\end{rem}

\begin{proof}[Proof of Theorem \ref{thm:plane_halfplane_rigid} by assuming Propositions \ref{prop:rigid_iso} and \ref{prop:rigid_free}]

We note that, by \cite{Galaz-Garcia_18}, the quotient space $Y/K$ (equipped with its quotient distance and measure) is an $\RCD(0,N)$ admitting a free $\R$-action that satisfies the assumptions of Proposition \ref{prop:rigid_free}. We deduce that $Y/K$ is isomorphic (up to a constant) to the Euclidean halfplane $\R\times [0,\infty)$ equipped with the Lebesgue measure. In particular, $Y/K$ contains a line. Since that line can be lifted to $Y$, the latter splits off an $\R$ factor. If $N<2$, Theorem \ref{thm:meas_split} would imply that $Y$ is isomorphic to a line, contradicting the fact that $Y/K$ is a Euclidean halfplane. Thus $N\ge 2$ and, by Theorem \ref{thm:meas_split}, we have:
\begin{equation*}
    (Y,y,d,\meas)=(\R,0,d_E,\mc{L}^1)\otimes (Z,z,d_Z,\mu),
\end{equation*}
for some pointed $\RCD(0,N-1)$ space $(Z,z,\dist_Z,\mu)$, where $\mu$ is a Radon measure on $Z$. Observe that $K$ only acts on $Z$ and satisfies $Kz=z$. It follows that $Y/K=\R\times (Z/K)$, and, in particular, $Z/K$ is isomorphic to $([0,\infty),\dist_{\mathrm{E}},\lambda\mathcal{L}^1)$, for some constant $\lambda>0$. Let $\pi\colon Z\to Z/K$ be the quotient map, and observe that, by definition of the quotient measure, we have $\pi_*\mu=\lambda\mc{L}^1$. This last assertion, combined with $Kz=z$, implies that $\pi^{-1}([0,r))=B_r(z)$ and $\mu(B_r(z))=\lambda\mc{L}^1([0,r))=\lambda r$. In addition, since $Kz=z$, then, given any $z'\in Z$, there exists a ray from $z$ through $z'$ in $Z$ (by lifting $Z/K=[0,\infty)$ into a ray through $z'$). Consequently, $Z$ is polar; hence, $(Z,z,\dist_Z,\mu)$ satisfies the assumptions of Proposition \ref{prop:rigid_iso}. Therefore, $Z$ is isomorphic (up to a constant) to a ray or a line equipped with Lebesgue measure. It follows that $Y$ is isomorphic (up to a constant) to Euclidean plane or halfplane equipped with the Lebesgue measure.
\end{proof}

The proof of Propositions \ref{prop:rigid_iso} and \ref{prop:rigid_free} uses extensively the \emph{disintegration} with respect to some $1$-Lipschitz function. Payne and Weinberger were the first to introduce such a disintegration, also called \emph{needle decomposition}, in their study of the optimal Poincar\'{e} inequality on convex domains of $\R^n$ (see \cite{Bebendorf03} for an exposition of their result). Later on, Klartag \cite{Klartag17} generalized Payne and Weinberger's idea to a Riemannian geometry framework and, in particular, extended Paul--Levy's isoperimetric inequality to weighted Riemannian manifolds with Ricci bounded below. Inspired by Klartag, Cavalletti--Mondino \cite{CM17} refined the needle decomposition to the setting of $\RCD$ spaces and extended Paul--Levy's isoperimetric inequality to this more general class. While the disintegration proposed by Cavalletti--Mondino held on probability spaces, they also generalized it to $\sigma$-finite measures in \cite{CM_newformula}.

\subsection{Needle decomposition in RCD spaces} Following \cite{CM_newformula}, we set up the disintegration of a $1$-Lipschitz function on an $\RCD(0,N)$ space $(X,\dist,\meas)$. (In our case, the setup will be largely simplified due to symmetry considerations and choice of Lipschitz function.) Given a $1$-Lipschitz function $u$, we define the ordered transport set associated with $u$ as follows:
\[\Gamma_{u}\defeq \{(x,y)\in X\times X: u(x)-u(y)=\dist(x,y)\},\]
and its transpose as $\Gamma_{u}^{-1}\defeq \{(x,y)\in X\times X:(y,x)\in\Gamma_{u}\}$.

\begin{rem}
The definition of the ordered transport set is sometimes reversed in the literature (e.g. in 
\cite{Ketterer_23}, which we will use later on).
\end{rem}

Then, we define the transport relation $R_u\defeq \Gamma_u\cup \Gamma_u^{-1}$, and the transport set with endpoints and branching points $\tran_{u,e}\defeq P_1(R_u\setminus \{(x,y):x=y\in X\})$, where $P_1:X\times X\to X$ is the projection onto the first factor. Their sections through $x\in X$ are denoted by $\Gamma_u(x)$, $\Gamma_u^{-1}(x)$, and $R_u(x)$ (e.g., $R_u(x)\defeq \{y\in X:(x,y)\in R_u\}$). Since $\Gamma_{u}$ carries an orientation, we introduce the following sets of distinguished points (which should be thought of as initial and final points):
\begin{equation}\label{eq:IniAndFin}
    a_u\defeq\{x\in \tran_u:\Gamma_u^{-1}(x)=\{x\}\}\quad\&\quad\  b_u\defeq\{x\in \tran_u:\Gamma_u(x)=\{x\}\}.
\end{equation}

Similarly, the set of forward and backward branching points, $A_+$ and $A_-$, are respectively defined by: 
\begin{align*}
    A_{\pm} &\defeq\{x\in \tran_{u,e}:\exists z,w\in \Gamma^{\pm1}_u(x), (z,w)\notin R_u\}.
\end{align*}
We will consider the non-branching transport set $\tran_u\defeq \tran_{u,e}\setminus(A_+\cup A_-)$. The set $R_u$ induces an equivalence relation when restricted to $\tran_u\times \tran_u$ (see \cite{Cavalletti14}). Thus, $\tran_u$ can be decomposed into equivalence classes $\{X_\alpha\}_{\alpha\in Q}$, for some index set $Q$. In addition, there exists a measurable selection map $s:\tran_u\to\tran_u$ so that $(x,y)\in R_u$ if and only if $s(x)=s(y)$ (see \cite{Cavalletti14}). As a result, one can identify the index set $Q$ with the image $\mathrm{Im}(s)\subset X$ of $s$, equipped with the measurable structure induced as a subset of $X$. We can equip $Q$ with a probability measure as described in \cite[lemma 3.3]{CM_newformula}. First, there exists a Borel measurable function $f:X\to (0,\infty)$ so that $f\meas|_{\tran_u}$ is a probability measure. Since the quotient map $\mathfrak{Q}: \tran_u\to Q$ induced by $s$ is measurable, we can push forward the normalized measure and obtain a probability measure $\mathfrak{q}\defeq \mathfrak{Q}_*(f\meas|_{\tran_u})$ on $Q$. The following disintegration theorem holds thanks to Cavalletti--Mondino \cite{CM_newformula}.

\begin{thm}\label{thm:needle}\cite[Theorem 4.6]{CM_newformula}
If $(X,\dist,\meas)$ is an $\RCD(0,N)$ space for some $N\in(1,\infty)$, then there exists a disintegration of $\meas$ such that:
\[\meas|_{\tran_u}=\int_Q\meas_\alpha\di \mk{q}(\alpha)\quad \&\quad \mk q(Q)=1.\]
Here, for $\mk q$-a.e.\@ index $\alpha\in Q$, $\meas_\alpha$ is supported on $\bar X_\alpha$ and $\meas_\alpha=h_\alpha\haus^1|_{X_\alpha}$. Moreover, $(\bar X_\alpha,\dist, \meas_\alpha)$ satisfies the $\RCD(0,N)$ condition. Finally, for every compact subset $K\subset X$, there exists $C_K>0$ such that, for $\mk q$-a.e.\@ index $\alpha\in Q$, $\meas_\alpha(K)\leq C_K$.
\end{thm}

\begin{rem}\label{rmk:para}
In fact, for $\mk q$-a.e.\@ $\alpha\in Q$, $(\bar X_\alpha, \dist,\meas_\alpha)$ is isomorphic to $(I_\alpha, |\cdot|, h_\alpha\haus^1)$, where $I_\alpha\subset \R$ is a closed and possibly infinite interval equipped with a $\mathrm{CD}(0,N)$ density $h_{\alpha}$. We parametrize as a unit speed geodesic $I_\alpha=[a(X_\alpha), b(X_\alpha)]$, if $a(X_\alpha)=-\infty$ or $b(X_\alpha)=\infty$, then $I_\alpha=(-\infty,b(X_\alpha)]$ or $[a(X_\alpha),\infty)$, and $I_\alpha=\R$ if both ends are infinity. This notation indicates that we would like to parametrize $I_\alpha$ so that the initial points are at the left end and the final points are at the right end. We will also say that $I_\alpha$ is a parametrization of $\bar X_\alpha$. The claim that $(I_\alpha, |\cdot|, h_\alpha\haus^1)$ is an $\RCD(0,N)$ space means that $h_\alpha$ satisfies a concavity inequality on $I_\alpha$, implying that $h_\alpha$ is continuous on $I_\alpha$, and is positive and log-concave (hence, locally Lipschitz) on the interior of $I_\alpha$.
\end{rem}

In our cases below, we will explicitly choose the selection map $s$, the index set $Q$, the normalization function $f$, and the parametrization of each $\bar X_\alpha$. Each $X_\alpha$ will be a ray, and we will be able to avoid some measure-theoretical consideration. The following monotonicity formula holds for the density $h_\alpha$ when $I_\alpha$ contains a ray. 

\begin{lem}\label{lem:density}
For $\mk q$-a.e $\alpha\in Q$, if $I_\alpha$ contains a ray $(-\infty,\sigma)$ for some $\sigma\in \R$, then for all $s\le t<\sigma$, $h_\alpha(t)\le h_\alpha (s)$.
\end{lem}

\begin{proof}
Let $(\sigma_-,\sigma_+)\subset I_\alpha$, then for $\mk q$-a.e $\alpha\in Q$ and any $\sigma_-<s\le t<\sigma_+$ we have 
\[
\left(\frac{\sigma_+-t}{\sigma_+-s}\right)^{N-1}\le \frac{h_\alpha(t)}{h_\alpha(s)}\le \left(\frac{t-\sigma_-}{s-\sigma_-}\right)^{N-1}.
\]
See for example \cite[(2-11),(2-20)]{CM_newformula}. By our assumption we can take $\sigma_+=\sigma$ and $\sigma_-\to -\infty$, then from the right inequality we have $h_\alpha(t)\le h_\alpha (s)$ for $s\le t<\sigma_+=\sigma$, as desired.
\end{proof}

\subsection{Rigidity with a pole} We prove Proposition \ref{prop:rigid_iso} in this subsection. We take $u=\dist_y\defeq \dist(\cdot,y)$ and consider the disintegration induced by it. Note that $Y$ is polar with the pole $y$, i.e., for any $x\neq y$, there exists a unique ray $\gamma_x$ passing through $x$ emanating from $y$. Replacing $\meas$ by $c^{-1}\meas$, which preserves the $\RCD(0,N)$ condition, we assume that
\[
\forall r>0,\ \meas(B_r(y))=r.
\]
Moreover, thanks to Remark \ref{rem:N=1}, we may assume $N>1$.

We start with the needle decomposition with the presence of a pole. The following observation follows from the non-branching property of $\RCD$ spaces \cite[Theorem~1.3]{Deng20}.

\begin{lem}\label{lem:ray}
We have $A_+\cup A_-\subset \{y\}$ and $R_{\dist_y}$ is an equivalence relation on $Y\setminus\{y\}$. In particular, $Y\setminus\{y\}\subset\tran_{\dist_y}$ and for each $x\neq y$, $R_{\dist_y}(x)$ is the unique ray $\gamma_x$ passing through $x$.
\end{lem}

\begin{proof}
We show that any $x\in Y$ other than $y$ is not a branching point. Indeed, for $x\neq y$ there is a unique ray $\gamma_x$ passing through it, so if any $z,w\in \Gamma_{\dist_y}(x) $ or $\Gamma^{-1}_{\dist_y}(x)$, then we have $|\dist(z,y)-\dist(x,y)|=\dist(x,z)$. This implies that $z,x,y$ are on a common geodesic, since $y,x$ are already on the ray $\gamma_x$ and $Y$ is non-branching, $z$ must be on $\gamma_x$, and so must $w$ by the same argument. Then $|\dist(z,y)-\dist(w,y)|=\dist(z,w)$ holds, so $(z,w)\in R_u$. It follows that $x\notin A_+\cup A_-$. To show $R_{\dist_y}$ is an equivalence relation on $Y\setminus\{y\}$, it suffices to show that if $(x,z), (z,w)\in R_{\dist_y}$ then $(x,w)\in R_{\dist_y}$. This follows from the exact same argument. First $(x,z)\in R_{\dist_y}$ implies $z\in \gamma_x$ and $\gamma_z=\gamma_x$ by non-branching property. Then $(z,w)\in R_{\dist_y}$ implies $w\in \gamma_z=\gamma_x$ so $(x,w)\in R_{\dist_y}$.
\end{proof}

\begin{rem} As a result of the proof above, $Y$ coincides with the transport set with endpoints and branching points $\mathcal{T}_{\dist_y,e}$. However, note that the inclusion $\mathcal{T}_{\dist_y}\subset Y$ may either be an equality or a strict inclusion. For example, if $Y$ is equal to a ray starting from $y$, then $A_+=A_-=\varnothing$, and $Y=\mathcal{T}_{\dist_y}$. Conversely, if $Y$ is equal to a line through $y$, then $y\in A_-$, and $\mathcal{T}=Y\backslash\{y\}$.
\end{rem}

It follows from Lemma \ref{lem:ray} that the map: 
\begin{equation*}
s\colon x\in Y\setminus\{y\}\to s(x)=\gamma_x\cap\{\dist_y=1\}\in Y\backslash\{y\},
\end{equation*}
is a selection map. Indeed, the map $s$ is well-defined since a ray emanating from $y$ intersects every geodesic sphere centered at $y$ at a unique point. In addition, the map $s$ is continuous (the proof relies on the same argument as the proof of Lemma \ref{lem:s_is_continuous}); hence, it is measurable. Thus, we will identify the quotient set $Q$ with $\{\dist_y=1\}$, equipped with the induced measurable structure coming from $Y$.  Finally, we observe that the function $f=e^{-\dist_y}$ is a normalizing function since, according to Cavalieri's formula, the following equation holds:
\[
\int_Y f(x)\di\meas(x)=-\int_0^\infty f'(r)\meas (B_r(y))\di r=\int_0^\infty re^{-r}\di r=1.
\] 
Lemma \ref{lem:ray} also implies that $Y\setminus\{y\}$ is decomposed into equivalence classes $\gamma_\alpha \setminus \{y\}$, for $\alpha\in Q=\{\dist_y=1\}$. Thanks to Theorem \ref{thm:needle}, there exists a disintegration associated with $\dist_y$ such that :
\[
\meas|_{Y\setminus\{y\}}=\int_Q h_\alpha \haus^1\di \mk q (\alpha),
\]
and $(\gamma_\alpha, \dist, h_\alpha\haus^1|_{\gamma_\alpha})$ is an $\RCD(0,N)$ space. It is straightforward to verify that $a_{\dist_y}=\emptyset$ and $b_{\dist_y}=\{y\}$ by following their definition (see \eqref{eq:IniAndFin}). We parametrize $\gamma_\alpha:(-\infty,0]\to Y$ with $\gamma_\alpha(0)=y$ and unit speed, which is consistent with the parametrization as indicated in Remark \ref{rmk:para}. 

Under this parametrization, we have 
\begin{equation}\label{eq:meas}
r=\meas(B_r(y))=\int_Q\int_{-r}^0 h_\alpha(\gamma_\alpha(t)) \di t\di \mk q(\alpha).
\end{equation}

A crucial consequence is that $\dist_y$ is harmonic in $Y\setminus\{y\}$.

\begin{lem}\label{lem:d_y is harmonic}
We have $\dist_y\in D(\mf{\Delta}, {Y\setminus\{y\}})$ and $\mf{\Delta} \dist_y=0$ in ${Y\setminus\{y\}}$.
\end{lem}

\begin{proof}
We first observe that the map
\begin{equation*}
s\in(-\infty,0)\longmapsto
\int_Q h_\alpha(\gamma_\alpha(s))\di\mk q(\alpha)\in\R
\end{equation*} 
is locally Lipschitz; see Remark
\ref{rem:continuity_of_integral_of_h_alpha} below. Thus, differentiating \eqref{eq:meas} gives, for every $t<0$,
\begin{equation}\label{eq:const}
1=\int_Q h_\alpha(\gamma_\alpha(t)) \di \mk q(\alpha).
\end{equation} 
Since each $\gamma_\alpha$ is a ray, Lemma \ref{lem:density} implies that, for $\mk q$-a.e.\@ $\alpha\in Q$, we have $h_\alpha(\gamma_\alpha(t))\le h_\alpha(\gamma_\alpha(s))$, for every $s\le t<0$. Along with \eqref{eq:const}, we infer that, for $\mk q$-a.e.\@ $\alpha\in Q$, the function $r\mapsto h_\alpha(\gamma_\alpha(r))$ is constant on $(-\infty,0]$ (the constant may depend on $\alpha$).

By the representation formula for the distributional Laplacian \cite[Corollary 4.19]{CM_newformula} (see also \cite[Remark 4.9]{CM_newformula} for infinite transport rays), we conclude that $\dist_y\in D(\mf{\Delta},Y\setminus\{y\})$ and
\[
\mathbf{\Delta} \dist_y|_{Y\setminus\{y\}}= -(\log h_\alpha)'\meas=0.
\]
Note that we do not have any singular part. Morally, the singular part is supported on the point $y$, which is excluded, and on its cut locus, which is empty as $y$ is a pole.
\end{proof}

\begin{rem}\label{rem:unit_ray_density}
In the proof above, we showed that $t\mapsto h_\alpha(\gamma_\alpha(t))$ is constant on $(-\infty,0]$. We now identify this constant. Write $h_\alpha(\gamma_\alpha(t))=c_\alpha$ and recall that $\mk q=s_*(e^{-\dist_y}\meas|_{Y\setminus\{y\}})$ by definition. Consequently, for every Borel set $A\subset Q$, the disintegration formula and the fact that $s(x)=\alpha$ for $x\in \gamma_\alpha$ gives
\[
\begin{aligned}
\mk q(A)
&=\int_{s^{-1}(A)}e^{-\dist_y}\di\meas=\int_A\left(\int_{\gamma_\alpha}
e^{-\dist_y}\di\meas_\alpha\right)\di\mk q(\alpha).
\end{aligned}
\]
Since this holds for every Borel set $A\subset Q$, we have $\int_{\gamma_\alpha}e^{-\dist_y}\di\meas_\alpha=1$, for $\mk q$-a.e.\ $\alpha$. Since $\dist_y(\gamma_\alpha(t))=-t$ and $\meas_\alpha=h_\alpha\mathcal H^1|_{\gamma_\alpha}$, it follows that
\[
1=\int_{-\infty}^0e^t h_\alpha(\gamma_\alpha(t))\di t
=c_\alpha\int_{-\infty}^0e^t\di t=c_\alpha.
\]
Therefore $h_\alpha(\gamma_\alpha(t))=1$ for $\mk q$-a.e.\ $\alpha\in Q$ and every $t<0$, and hence $\meas|_{Y\setminus\{y\}} =\int_Q\mathcal H^1|_{\gamma_\alpha\setminus\{y\}} \,\di\mk q(\alpha)$.
\end{rem}

\begin{rem}[The function $s\mapsto \int_Q h_\alpha(\gamma{}_\alpha(s))\di \mk q(\alpha)$ is locally Lipschitz] \label{rem:continuity_of_integral_of_h_alpha}
We denote $\tilde{h}_{\alpha}(t)\coloneqq h_\alpha(\gamma{}_\alpha(t))$, for $-\infty<t\leq 0$. Let us recall that, given $\alpha\in Q$, $r>0$, and $-r\leq t\leq -r/2$ such that $\tilde{h}_{\alpha}$ is differentiable at $t$, we have the following estimate (see \cite[Lemma A.9]{Cavalletti-Milman_21} using $a=-2r$ and $b=0$):
\begin{equation*}
\lvert \tilde{h}_{\alpha}'(t)\rvert \leq \frac{2(N-1)}{r}\lVert \tilde{h}_{\alpha}\rVert_{L^{\infty}(-r,-r/2)}.
\end{equation*}
However, according to \cite[Lemma A.8]{Cavalletti-Milman_21}, the following estimate holds:
\begin{equation*}
\lVert\tilde{h}_{\alpha}\rVert_{L^{\infty}(-r,-r/2)}\leq \frac{2N}{r}\int_{-r}^{-r/2}\tilde{h}_{\alpha}(s)\ ds\leq\frac{2N}{r}\meas_{\alpha}(B_r(y)).
\end{equation*}
In particular, if we denote $\Psi(s)\coloneqq \int_Q h_\alpha(\gamma{}_\alpha(s))\di \mk q(\alpha)$ ($s\in(-\infty,0]$), then for every $-r\leq t_1,t_2\leq -r/2$, we have the following:
\begin{align*}
\lvert\Psi(t_1)-\Psi(t_2)\rvert &\leq \frac{4N(N-1)}{r^2}\lvert t_1-t_2\rvert\int_{Q}\meas_{\alpha}(B_r(y))\di\mk q(\alpha)\\
&=\frac{4N(N-1)}{r}\lvert t_1-t_2\rvert.
\end{align*} 
Thus $\Psi$ is locally Lipschitz on $(-\infty,0)$.
\end{rem}

\begin{rem}
It was brought to our attention by our reviewers that an alternative approach to proving $\mathbf\Delta \dist_y = 0$ on $Y\backslash\{y\}$ is possible using the Gauss--Green formula established by Bru\`e--Pasqualetto--Semola in the $\RCD(0,N)$ case in \cite{Brue-Pasqualletto-Semola_Green}. Indeed, thanks to \cite{CM_newformula}, since our space is polar at $y$, we have $\mathbf{\Delta}\dist_y = (\Delta \dist_y)\meas + \mu_{s}$, where $\mu_{s}$ is supported on $\{y\}$ and $\Delta \dist_y$ is radial. In particular, taking $V = \nabla \dist_y$, $E = \{\dist_y > \epsilon\}$ (for a fixed $\epsilon>0$), and $\phi = \eta\circ \dist_y$ (for some fixed $\eta\in\mathrm{Lip}_c([0,\infty))$) in the Gauss--Green formula \cite[Theorem 1.6]{Brue-Pasqualletto-Semola_Green}, we have:
\begin{align*}
-\int_{\{\dist_y = \epsilon\}}\eta(\epsilon) (\nabla \dist_y\cdot \nu_E)_{int} \di \lvert {D}\chi_E\rvert = &\int_{\{\dist_y>\epsilon\}}(\eta\circ \dist_y )\Delta \dist_y \di \meas\\
&+ \int_{\{\dist_y>\epsilon\}} (\eta'\circ \dist_y)\lvert \nabla \dist_y\rvert^2 \di \meas.
\end{align*} 
However, $(\nabla \dist_y\cdot \nu_E)_{int} = 1$ $\lvert {D}\chi_E\rvert$-a.e.\@ (taking $\phi = -\dist_y$ in \cite[Proposition 6.1]{Brue-Pasqualletto-Semola_Green}). We also have $\lvert \nabla \dist_y\rvert^2=1$ $\meas$-a.e.\@ and $\operatorname{Per}(E)=\lvert D\chi_E\rvert(Y)=1$ (since $\meas(B_r(y))=r$ for all $r\geq0$ ). In particular, the left-hand side is $-\eta(\epsilon)$.  By the coarea formula, we also have:
\begin{equation*}
\int_{\{\dist_y>\epsilon\}} (\eta'\circ \dist_y)\lvert \nabla \dist_y\rvert^2 \di \meas
=\int_{\epsilon}^{\infty}\eta'(r)\operatorname{Per}(\{\dist_y>r\})\di r
=-\eta(\epsilon),
\end{equation*} 
using again that $\mathrm{Per}(\{\dist_y>r\})\equiv 1$.  Therefore, we have:
\begin{align*}
0 = \int_{\{\dist_y>\epsilon\}}(\eta\circ \dist_y )\Delta \dist_y \di \meas &= \int_{\epsilon}^{\infty}\eta(r)\Delta \dist_y(r)\mathrm{Per}(\{\dist_y > r\})\di r\\
&= \int_{\epsilon}^{\infty}\eta(r)\Delta \dist_y(r)\di r.
\end{align*} 
Since this holds for every $\eta\in\mathrm{Lip}_c([\epsilon,\infty))$, and since $\epsilon>0$ is arbitrary, we have $\Delta \dist_y\equiv0$ on $Y\backslash\{y\}$.
\end{rem}

We next prove that the completion of $Y\setminus\{y\}$, endowed with its extended intrinsic distance, is isometric to $Q\times(-\infty,0]$, with the isometry given by the ray map introduced below.

\begin{defn}[Ray and flow maps]\label{def:flow_map}
For $\alpha\in Q$, we keep the parametrization $\gamma_\alpha:(-\infty,0]\to Y$, with $\gamma_\alpha(0)=y$ and $\dist_y(\gamma_\alpha(t))=-t$. Define the \emph{ray map by}
\[
g:Q\times(-\infty,0)\longrightarrow Y\setminus\{y\},
\qquad g(\alpha,t)=\gamma_\alpha(t),
\]
and, for $x=g(\alpha,t)$ and $\tau\in\mathbb R$ such that $t+\tau<0$, define the \emph{flow} $F_\tau$ by
\[
F_\tau x=g(\alpha,t+\tau).
\] 
Thus $\dist_y(F_\tau x)=\dist_y(x)-\tau$; positive flow time moves toward $y$.
\end{defn}

\begin{rem}
Recall that $Q$ is identified with $\{\dist_y=1\}$ through the selection map $s$, and the inverse of $g$ is $x\mapsto(s(x),-\dist_y(x))$. The map $g$ is continuous: if $(\alpha_i,t_i)\to(\alpha,t)$, properness gives convergent subsequences of $g(\alpha_i,t_i)$, and uniqueness of the ray through $\alpha$ forces every subsequential limit to be $g(\alpha,t)$. Thus $g$, its inverse, and $(\tau,x)\mapsto F_\tau x$ are continuous.  Morally, $F_\tau$ is the gradient flow of $-\dist_y$.
\end{rem}

\begin{rem}
In \cite[Section~4.2]{KKL23}, the flow is initially defined only up to a negligible set, and \cite[Proposition~4.16 and Remark~4.17]{KKL23} are used to construct continuous representatives. In our case the ray map and the flow are already defined everywhere on their natural domains. The proof of Theorem \ref{thm:splitting} below verifies directly that these canonical maps are the required representatives.
\end{rem}

\begin{defn}[Completion of the extended intrinsic distance on $Y\backslash\{y\}$]\label{def:extended_intrinsic_Y}
The intrinsic extended distance $\tilde{\dist}_{Y\backslash\{y\}}$ on $Y\backslash\{y\}$ is defined in the following way:
\begin{itemize}
\item if $x$ and $x'$ are in different connected components of $Y\backslash\{y\}$, then $\tilde{\dist}_{Y\backslash\{y\}}(x,x')\coloneqq\infty$;
\item otherwise, $\tilde{\dist}_{Y\backslash\{y\}}(x,x')$ is the infimum of the $\dist$-lengths of curves from $x$ to $x'$ in $Y\backslash\{y\}$.
\end{itemize}
We denote by $(\tilde Y,\dist_{\tilde Y})$ the componentwise completion of $(Y\backslash\{y\},\tilde{\dist}_{Y\backslash\{y\}})$ and equip it with $\meas_{\tilde Y}\coloneqq\iota_\#\meas|_{Y\backslash\{y\}}$, where $\iota\colon Y\backslash\{y\}\to\tilde{Y}$ is the inclusion map.
\end{defn}

\begin{rem}\label{rem:local_intrinsic_Y}
The distances $\dist$ and $\tilde{\dist}_{Y\backslash\{y\}}$ agree locally on $Y\backslash\{y\}$.  Indeed, fix $x\in Y\setminus\{y\}$ and choose $\rho>0$ such that $3\rho<\dist(x,y)$. Any two points of $B_\rho(x)$ are joined by an ambient minimizing geodesic contained in $B_{3\rho}(x)$, hence avoiding $y$. Therefore, $\tilde{\dist}_{Y\setminus\{y\}}=\dist$ on $B_\rho(x)$.
\end{rem}

\begin{rem}
Note that $\tilde Y$ may have a different topology than $Y$, for example, when $Y$ is a line.
\end{rem}

\begin{defn}[Extended intrinsic distance on $Q$]\label{def:extended_intrinsic_Q}
We denote by $\tilde{\dist}_{Q}$ the extended intrinsic distance on $Q$ and recall that $Q$ is equipped with the measure $\mk q=s_*(e^{-\dist_y}\meas|_{Y\setminus\{y\}})$. 
\end{defn}

The following theorem follows by combining \cite[Theorem~4.9]{Ketterer_23} with the arguments developed in \cite[Sections~4--5]{KKL23} and \cite[Section~4]{GKKO20}. Although the latter are presented for compact spaces, the parts used here are local: all measures, cutoffs, and comparison geodesics are confined to a relatively compact subset of $Y\setminus\{y\}$, and the only compactness input is properness of $(Y,\dist)$. We include a sketch for the reader's convenience and to identify the isometry with the ray map. 

\begin{thm}\label{thm:splitting}
If $(Y,y,\dist,\meas)$ satisfies the hypotheses of Proposition
\ref{prop:rigid_iso}, then the ray map $g$ extends uniquely to an
isomorphism
\[
\bar g:(Q,\tilde{\dist}_Q,\mk q)
\otimes((-\infty,0],\dist_{\mathrm E},\mathcal L^1)
\longrightarrow(\tilde Y,\dist_{\tilde Y},\meas_{\tilde Y}).
\] 
Moreover,
$\bar g(Q\times\{-r\})=\dist_y^{-1}(r)$ for every $r>0$.
\end{thm}

\begin{proof}[Sketch of proof]

Set $u=\dist_y$ and $\Omega=Y\setminus\{y\}$. Since $\Delta u=0$ and $|\nabla u|=1$ on $\Omega$, \cite[Theorem~4.9]{Ketterer_23} gives $\operatorname{Hess}u=0$. By the second-order differentiation formula of Gigli--Tamanini \cite{GigliTamanini18,GigliTamanini21}, applied locally as in \cite[Proposition~4.6]{KKL23}, the functional
\[
\mathcal U(\mu)\coloneqq\int u\,\di\mu
\]
is affine along Wasserstein geodesics whose supports remain in a
relatively compact subset of $\Omega$.

Localizing in a sufficiently small ball $U\Subset\Omega$ and following the EVI argument of \cite[Propositions~4.10 and~4.12]{KKL23}, we use the first-order differentiation formula of Gigli \cite[Theorem~3.10]{Gigli15} and the metric Brenier theorem of Ambrosio--Gigli--Savar\'e \cite{Ambrosio-Gigli-Savare_14} to obtain, for
$\mu_\tau=(F_\tau)_*\mu$,
\[
\frac{\di}{\di\tau}\frac12W_2^2(\mu_\tau,\nu)
=\mathcal U(\nu)-\mathcal U(\mu_\tau)
\]
for a.e.\ small $\tau$ and for any bounded-density probability measures $\mu,\nu$ with sufficiently small supports in $U$. Since $\mathcal U(\mu_\tau)=\mathcal U(\mu)-\tau$, integration and the approximation by Dirac masses used in the proof of \cite[Proposition~4.10]{GKKO20} yield locally \begin{equation}\label{eq:pointwise_EVI_polar}
\dist^2(p,F_\tau q)
=\dist^2(p,q)+2\tau\bigl(u(p)-u(q)\bigr)+\tau^2.
\end{equation}
Applying \eqref{eq:pointwise_EVI_polar} twice, with the roles of $p$ and $q$ interchanged, and using $u(F_\tau p)=u(p)-\tau$, gives $\dist(F_\tau p,F_\tau q)=\dist(p,q)$ whenever $p$ and $q$ are sufficiently close and $\tau$ is sufficiently small. Using a standard covering argument and the semigroup property $F_{t_1}\circ F_{t_2}=F_{t_1+t_2}$, we deduce that, for every $t\in\R$, $F_t$ is a local isometry on its domain.

Fix $x_0\in\Omega$, set $a=u(x_0)$, and define $P_a(p)\coloneqq F_{u(p)-a}p$ for $p$ near $x_0$. For nearby $p,q$, put $h=u(q)-u(p)$. Formula \eqref{eq:pointwise_EVI_polar} gives $\dist^2(p,F_hq)=\dist^2(p,q)-h^2$. Since $u(F_hq)=u(p)$, applying the local isometry $F_{u(p)-a}$ to both points and using the semigroup property yields
\[
\dist^2(P_ap,P_aq)
=\dist^2(p,q)-|u(p)-u(q)|^2.
\]
This identity immediately implies that every level set of $u$ is
locally geodesically convex.

Recall that $Q=\{u=1\}$ and that $s(p)\coloneqq F_{u(p)-1}(p)\in Q$ is our selection map. Near $x_0$ we have $s=F_{a-1}\circ P_a$. Since $F_{a-1}$ is a local isometry, the preceding identity gives
\[
\dist^2(p,q)
=\dist^2(s(p),s(q))+|u(p)-u(q)|^2.
\]
The local geodesic convexity of $Q$ implies that $\dist$ and $\tilde{\dist}_Q$ agree locally on $Q$. Together with Remark \ref{rem:local_intrinsic_Y}, the preceding identity therefore shows that
\[
p\longmapsto\bigl(s(p),-u(p)\bigr),
\]
which is the inverse of the ray map $g$, is a local isometry from $(Y\setminus\{y\},\tilde{\dist}_{Y\setminus\{y\}})$ to the  intrinsic product $(Q,\tilde{\dist}_Q)\otimes(-\infty,0)$.

We also observe that $(Q,\tilde{\dist}_Q)$ is complete componentwise. Indeed, every $\tilde{\dist}_Q$-Cauchy sequence is $\dist$-Cauchy and therefore converges with respect to $\dist$ to a point of $Q$, since $Q=\{u=1\}$ is closed in $(Y,\dist)$. The local agreement of $\dist$ and $\tilde{\dist}_Q$ then gives convergence with respect to $\tilde{\dist}_Q$.

Since $g$ is a homeomorphism, it maps the connected components of $Q\times(-\infty,0)$ bijectively onto those of $Y\setminus\{y\}$; hence, the corresponding extended intrinsic distances are infinite between distinct components. Moreover, since $g$ is a local isometry, a finite subdivision argument shows that it preserves the length of every rectifiable curve. The same holds for $g^{-1}$. Consequently, the restriction of $g$ to each connected component is a surjective global isometry onto the corresponding component of $Y\setminus\{y\}$. It therefore extends uniquely to a surjective isometry between their metric completions. Since $(Q,\tilde{\dist}_Q)$ is complete componentwise, the componentwise metric completion of $Q\times(-\infty,0)$ is $Q\times(-\infty,0]$. Combining the componentwise extensions therefore gives the required isometry
\[
\bar g\colon
(Q,\tilde{\dist}_Q)\otimes(-\infty,0]
\longrightarrow
(\tilde Y,\dist_{\tilde Y}).
\]
Finally, Remark \ref{rem:unit_ray_density} gives the measure identity, while $u(g(\alpha,t))=-t$ identifies the images of the level sets.
\end{proof}

\begin{rem}\label{rem:Q_locally_geodesically_convex}
As a consequence of the proof above (compare \cite[Corollary~4.7]{KKL23}), $Q=\{\dist_y=1\}$ is locally geodesically convex. Hence $\dist$ and $\tilde{\dist}_Q$ agree locally on $Q$.
\end{rem}

We will then control the diameter of any connected component of $({Q},\tilde{\dist}_{Q})$, aiming to show the diameter is $0$.

\begin{lem}\label{lem:level_sets_are_convex}
The section $(Q,\tilde{\dist}_Q)$ is compact and, for every $r>0$, $\dist_y^{-1}(r)$ is geodesically convex in $(\tilde Y,\dist_{\tilde Y})$ and isometric to $(Q,\tilde{\dist}_Q)$. 
\end{lem}

\begin{proof}
We first show that $({Q},\tilde{\dist}_{Q})$ is compact. Let $\{x_i\}$ be a sequence in $Q=\{\dist_y=1\}$. Since $(Y,\dist)$ is proper, $(Q,\dist)$ is compact, and there exists a subsequence $\{x_{k_i}\}$ converging to $x\in Q$ w.r.t.\@ the extrinsic distance $\dist$. However, $Q$ is locally geodesically convex in $(Y,\dist)$ (see Remark \ref{rem:Q_locally_geodesically_convex}). Consequently, the extrinsic distance $\dist$ and intrinsic distance $\tilde{\dist}_Q$ coincide on $B_{\delta}(x)$, for any $\delta>0$ small enough. In particular, $\{x_{k_i}\}_i$ is converging to $x\in Q$ w.r.t.\@ the intrinsic distance $\tilde{\dist}_Q$. Hence, $({Q},\tilde{\dist}_{Q})$ is compact.

Since $\bar g$ extends the ray map, Theorem \ref{thm:splitting} gives
\[
\bar g(Q\times\{-r\})=\dist_y^{-1}(r)
\]
for every $r>0$.  The slice $Q\times\{-r\}$ is geodesically convex in $(Q,\tilde{\dist}_Q)\otimes(-\infty,0]$. The slice is also isometric to $(Q,\tilde{\dist}_Q)$. Therefore, $\dist_y^{-1}(r)$ is geodesically convex in $(\tilde Y,\dist_{\tilde Y})$ and isometric to $(Q,\tilde{\dist}_Q)$. 
\end{proof}

\begin{rem}\label{rem:Q_has_finitely_many_connected_components}
By Remark \ref{rem:Q_locally_geodesically_convex}, \(Q\) is locally geodesically convex and hence locally path-connected. Thus its connected components are open; since \(Q\) is compact by Lemma \ref{lem:level_sets_are_convex}, it has only finitely many connected components.
\end{rem}

Let us introduce the following quantity:
\begin{equation}\label{eq:max_diam}
D\coloneqq\sup_{\mathcal{C}}\{\mathrm{Diam}(\mathcal{C},\tilde{\dist}_Q)\},
\end{equation}
where the supremum ranges over connected components of $Q$ (note that Remark \ref{rem:Q_has_finitely_many_connected_components} together with the compactness of $Q$ implies that $D$ is finite).

\begin{lem}\label{lem:level_sets_are_convex_for_the_extrinsic_distance}
For every $R>D/2$, any connected component of $\dist_y^{-1}(R)$ is convex in $(Y,\dist)$.
\end{lem}

\begin{proof}
Let $\mathcal C$ be a connected component of $\dist_y^{-1}(R)$ and let $x_1,x_2\in\mathcal C$. Since the intrinsic diameter of $\mathcal C$ is at most $D$, we have
\[
\dist(x_1,x_2)\leq D.
\]
Let $\gamma\colon[0,1]\to Y$ be a constant-speed minimizing geodesic from $x_1$ to $x_2$.  For every $t\in[0,1]$,
\begin{align*}
2R
=\dist(y,x_1)+\dist(y,x_2)&\leq
2\dist(y,\gamma_t)
+\dist(x_1,\gamma_t)+\dist(x_2,\gamma_t)\\
&=
2\dist(y,\gamma_t)+\dist(x_1,x_2)\\
&\leq
2\dist(y,\gamma_t)+D.
\end{align*} 
Hence
\[
\dist(y,\gamma_t)\geq R-\frac D2>0,
\]
so $\gamma$ is contained in $Y\setminus\{y\}$. It is therefore also a minimizing geodesic in $(\tilde Y,\dist_{\tilde Y})$. By the geodesic convexity of $\dist_y^{-1}(R)$ established in Lemma \ref{lem:level_sets_are_convex}, $\gamma$ is contained in $\dist_y^{-1}(R)$. Since $\gamma$ is connected and contains $x_1$, it is in fact contained in $\mathcal C$. 
\end{proof}

\begin{rem}[Components of level and superlevel sets]
\label{rem:components_under_g}
\label{rem:super_level_sets_are_convex}
Let $Q_1,\ldots,Q_k$ be the connected components of $Q$. For every $R>0$, the connected components of $\dist_y^{-1}(R)$ and $\{\dist_y\geq R\}$ are respectively
\[
L_i^R\coloneqq\bar g\bigl(Q_i\times\{-R\}\bigr)
\quad\text{and}\quad
C_i^R\coloneqq\bar g\bigl(Q_i\times(-\infty,-R]\bigr),
\qquad i=1,\ldots,k.
\]
Moreover, if $R>D/2$, then every $C_i^R$ is convex in $(Y,\dist)$. Indeed, suppose that a minimizing geodesic $\gamma$ with endpoints in $C_i^R$ leaves $\{\dist_y\geq R\}$. The first and last points at which $\gamma$ meets $\dist_y^{-1}(R)$ belong to $L_i^R$. By Lemma \ref{lem:level_sets_are_convex_for_the_extrinsic_distance}, the portion of $\gamma$ between these points must remain in $L_i^R$, which is a contradiction.
\end{rem}

\begin{lem}\label{lem:D=0}
$D=0$ (see \eqref{eq:max_diam}).
\end{lem}

\begin{proof}
Looking for a contradiction, we assume that $D>0$, and we fix a connected component $\mathcal{C}$ of $Q$ with diameter $D$ w.r.t.\@ $\tilde{\dist}_Q$. Let us also fix a sequence $r_i\to 0$, and, for every $i\in\N$, denote by $\mathcal{C}_{r_i}\coloneqq \bar g(\mathcal{C}\times\{-2r_i\})$ the associated connected component of $\dist_y^{-1}(2r_i)$.  In addition, let us recall that, thanks to Lemma \ref{lem:level_sets_are_convex}, we have $\mathrm{Diam}(\mathcal{C}_{r_i},\dist_i)=\mathrm{Diam}(\mathcal{C},\tilde{\dist}_Q)=D$, where $\dist_i$ is the intrinsic metric on $\mathcal{C}_{r_i}$ (in particular, $\dist_i$ is simply the restriction of ${\dist_{\tilde{Y}}}$ to $\mathcal{C}_{r_i}$). Therefore, the following holds:
\begin{equation*}
\mathrm{Diam}(\mathcal{C}_{r_i},r_i^{-1}\dist_i)=r_i^{-1}D\to \infty.
\end{equation*}

Let us now fix $\epsilon>0$, and, for each $i\in\N$, denote:
\begin{align*}
\mathcal N_i(\epsilon)
\coloneqq\max\bigl\{k\in\mathbb N :\;&
    \exists\,x_1,\dots,x_k\in\mathcal C_{r_i}
    \text{ such that}\\
& r_i^{-1}\dist_i(x_j,x_\ell)\ge\epsilon
    \quad\text{for all }j\ne\ell
\bigr\}.
\end{align*}
Given that $\{(\mathcal{C}_{r_i},r_i^{-1}\dist_i)\}$ is a sequence of length metric spaces with diameter going to $\infty$, we must have $\mathcal{N}_i(\epsilon)\to \infty$ as $i\to\infty$.

However, for any $\epsilon\leq2$, the proof of Lemma \ref{lem:level_sets_are_convex_for_the_extrinsic_distance} shows that, for all $x,x'\in \mathcal{C}_{r_i}$ with $\dist(x,x')< 2\epsilon r_i$, the following equality is satisfied:
\begin{equation*}
\dist(x,x')=\dist_i(x,x').
\end{equation*} 
Indeed, the ambient geodesic stays at distance at least $2r_i-\dist(x,x')/2>0$ from $y$, and hence is a geodesic in $(\tilde Y,\dist_{\tilde Y})$; the product splitting then forces it to remain in the level set. It follows that every $\epsilon$-separated subset of $(\mathcal C_{r_i},r_i^{-1}\dist_i)$ is also $\epsilon$-separated in $(Y,r_i^{-1}\dist)$. The disjoint balls of radius $\epsilon/2$ centered at these points lie in $B_3(y)$ in the rescaled space. Bishop--Gromov comparison, applied also to
$B_3(y)\subset B_5(x_j)$, gives
$$\mathcal{N}_i(\epsilon) \le (10/\epsilon)^N,$$ 
a contradiction to $\mathcal{N}_i(\epsilon)\to \infty$.
\end{proof}

Finally, we can finish the proof of Proposition \ref{prop:rigid_iso}.

\begin{proof}[Proof of Proposition \ref{prop:rigid_iso}] Replacing $\meas$ by $c^{-1}\meas$ if necessary, we may assume that $c=1$, i.e., $\meas(B_r(y))=r$ (for all $r>0$). Lemma \ref{lem:D=0} implies that the connected components $Q_1,\ldots,Q_k$ of $Q$ are singletons. Let us fix $R>0$ and observe that thanks to Remark \ref{rem:super_level_sets_are_convex} and Lemma \ref{lem:D=0}, $C_1^R,\ldots,C_k^R$ (defined as in Remark \ref{rem:components_under_g}) are geodesically convex in $(Y,\dist)$. Therefore, Theorem \ref{thm:splitting}, restricted to $Q_i\times(-\infty,-R]$, gives an isomorphism
\[
\bigl(C_i^R,\dist,\meas\llcorner C_i^R\bigr)
\simeq
\bigl((-\infty,-R],\dist_{\mathrm E},
\mk q(Q_i)\mathcal L^1\bigr).
\]
In particular, the $1$-regular set $\mathcal R_1$ of $(Y,\dist,\meas)$ is nonempty.  Thus, as a result of \cite[Theorem~1.1]{Kitabeppu-Lakzian_16}, and since $Y$ is noncompact, $(Y,\dist,\meas)$ is necessarily isomorphic to a Euclidean line or ray equipped with a multiple of the Lebesgue measure.
\end{proof}

\subsection{Rigidity with a free \texorpdfstring{$\mathbb{R}$}{R}-action}
We suppose that the hypotheses of Proposition \ref{prop:rigid_free} hold. Replacing $\meas$ by $c^{-1}\meas$, which preserves the $\RCD(0,N)$ condition, we assume throughout this subsection that
\[
\meas(\Omega_r(v))=rv
\qquad\text{for every }r,v\geq0.
\]
Moreover, by Remark \ref{rem:N=1}, we may assume that $N>1$. We take the distance function to the orbit $Gy$, $\dist_\partial\defeq \dist (\cdot, Gy)$, and consider the corresponding disintegration. For each $z\in Y$ there is a unique horizontal ray $\gamma_z$ passing through $z$, which is a lift of $Y/G=[0,\infty)$. 
 \begin{lem}
     We have $A_+\cup A_-=\emptyset$, $\tran_{\dist_\partial}=Y$, $a_{\dist_\partial}=\emptyset$ and $b_{\dist_\partial}=Gy$. In particular, for any $x\in Y$, $R_{\dist_\partial}(x)$ is the horizontal ray through $x$.
 \end{lem}

 \begin{proof}
     Let $(x,z)\in \Gamma_{\dist_\partial}$, then we have $\dist_\partial (x)-\dist_\partial(z)=\dist(x,z)$. We claim that $x,z$ are on the same horizontal ray. Let $x',z'\in Gy$ such that $\dist_\partial(x)=\dist(x,x')$ and $\dist_\partial(z)=\dist(z,z')$, it follows 
    \begin{equation*}
         \dist(x,z')\le \dist(x,z)+\dist(z,z')=\dist(x,x')-\dist(z,z')+\dist(z,z')=\dist(x,x')=\dist(x,Gy).
    \end{equation*}
    Meanwhile, it holds $\dist(x,Gy)\le\dist(x,z')$. This forces $ \dist(x,z')= \dist(x,z)+\dist(z,z')$. If $z=z'$, then $\dist(x,z)=\dist(x,Gy)$, by Lemma \ref{lem:horizontal_rays}(1), $x$ is on the horizontal ray emanating from $z$. If $z\neq z'$,  then $x,z,z'$ are on a geodesic. By the non-branching property $x$ is on the horizontal ray through $z, z'$. Clearly the same argument applies to show if $(x,z)\in \Gamma^{-1}_{\dist_\partial}$ then $x,z$ are on the same horizontal ray as well. It follows that for any $x\in Y$, $R_{\dist_\partial}(x)$ is the unique horizontal ray through $x$. As a consequence $A_+\cup A_-=\emptyset$ and $\tran_{\dist_\partial}=Y$. Finally, we see that $a_{\dist_\partial}=\emptyset$ since, for every $x\in Y$, $\Gamma^{-1}_{\dist_\partial}(x)$ is a ray. We also see that $Gy\subset b_{\dist_\partial}$ by definition, and that if $x\notin Gy$, then $\Gamma_{\dist_\partial}(x)$ contains the geodesic segment from $x$ to $Gy$, so $x\notin b_{\dist_\partial}$. We get $Gy= b_{\dist_\partial}$.
 \end{proof}

 Given $x\in Y$, we denote by $s(x)$ the unique intersection point between the horizontal ray through $x$ and the set $\{\dist_{\partial}=1\}$. As a result, we obtain a surjective map:
 \begin{equation}\label{eq:selection_free_case}
     s:Y\to \{\dist_\partial=1\}
 \end{equation}
 such that $(x,y)\in R_{\dist_{\partial}}$ if and only if $s(x)=s(y)$.
 
 \begin{lem}\label{lem:s_is_continuous}
     The map $s$ in \eqref{eq:selection_free_case} is continuous. In particular, it is a selection map for the equivalence relation induced by $R_{\dist_{\partial}}$.
 \end{lem}

 \begin{proof}
     Let us assume that $x_n\to x_{\infty}$ in $Y$ and let us show that $s(x_n)\to s(x_{\infty})$. Observe that, for every $n\in\N\cup\{\infty\}$, there exists a unique horizontal ray $\gamma_n\colon[0,\infty)\to Y$ through $x_n$. In particular, denoting $t_n\coloneqq \dist_{\partial}(x_n)$, we have $x_n=\gamma_n(t_n)$. Since $\dist_{\partial}$ is continuous, we have $t_n\to t_{\infty}$. In addition, $\{\gamma_n(0)\}$ is a bounded sequence; hence, thanks to Arzel\`{a}--Ascoli's theorem, $\{\gamma_n\}$ subsequentially converges to a horizontal ray $\gamma$ w.r.t. the topology of uniform convergence on compact subsets of $Y$. In addition, observe that $\gamma(t_\infty)=\lim\gamma_n(t_n)=\lim x_n=x_{\infty}$, i.e. $\gamma$ is the horizontal ray through $x_{\infty}$; in other words, $\gamma$ is equal to $\gamma_{\infty}$. As a result, $s(x_n)=\gamma_n(1)\to\gamma_{\infty}(1)=s(x_\infty)$, which concludes the proof.
 \end{proof}
 
 Since $s$ is a continuous selection (in particular, measurable), we can identify the index set $Q$ with $\{\dist_\partial=1\}$. With this identification, for each $\alpha\in Q$, there is a unique horizontal ray passing through $\alpha$ which we denote by $\gamma_\alpha$. Again, we use the convention of parametrization  from Remark \ref{rmk:para}, i.e. $\gamma_\alpha\colon(-\infty,0]\to Y$, so that $\gamma_\alpha(0)\in Gy$. By Theorem \ref{thm:needle}, there exists a disintegration associated with $\dist_\partial$ such that:
 \[
 \meas=\int_Q h_\alpha\haus^1\di \mk q(\alpha),
 \]
 and $(\gamma_\alpha, \dist, h_\alpha\haus^1|_{\gamma_\alpha})$ is an $\RCD(0,N)$ space. Together with \cite[Corollary 4.16]{CM_newformula}, the arguments used in the proof of Lemma \ref{lem:d_y is harmonic} carry over verbatim; hence, the following lemma holds.
 
 \begin{lem}
    We have $ \dist_\partial\in D(\mf \Delta, Y\setminus Gy)$ and $\mf \Delta \dist_\partial=0$ in $Y\setminus Gy$.
 \end{lem}

 \begin{proof}
      First notice that $Q$ is identified with $\{\dist_\partial=1\}$, which is an orbit of $G$, so we can further identify $(Q,\mk q)$ as $\R$ equipped with a probability measure. With such an identification, for any $v\in G=\R$, $v\ge 0$, we have
     \[
    rv= \meas(\Omega_r(v))=\int_{-r}^0\int_{-v}^v h_\alpha(\gamma_\alpha(t))\di\mk q(\alpha)\di t.
     \]
      By differentiating the above equation with respect to $r$ we have 
      \[
    v= \int_{-v}^v h_\alpha(\gamma_\alpha(-r))\di\mk q(\alpha).
     \]
     Since each $\gamma_\alpha$ is a ray, by Lemma \ref{lem:density}, we have that $h_\alpha(\gamma_\alpha(t))\le h_\alpha(\gamma_\alpha(s))$ for $s\le t<0$ $\mk q$-a.e.\@ $\alpha\in [-v,v]\subset Q(=\R)$. It follows that $h_\alpha(\gamma_\alpha(t))$ is a constant possibly depending on $\alpha$ for any $t\le 0$ and $\mk q$-a.e.\@ $\alpha\in [-v,v]$. Since $v\ge 0$ is arbitrary, we infer that $h_\alpha(\gamma_\alpha (t))$ is a constant for any $t\le 0$ and $\mk q$-a.e.\@ $\alpha\in Q$.

 By the representation formula of the distributional Laplacian from \cite[Corollary 4.16, Remark 4.9]{CM_newformula}, we have that 
 \[
 \mathbf{\Delta} \dist_{\partial}|_{Y\setminus Gy}= -(\log h_\alpha)'\meas=0,
 \] 
 as expected.
 \end{proof}

The next lemma is a direct consequence of the functional splitting theorem \ref{thm:ketterer}.

\begin{lem}\label{lem:halfplane_rigid}The following holds.
  \begin{enumerate}
      \item\label{lem:halfplane_rigid:item:1} $\dist_\partial^{-1}([0,r])$ is geodesically convex for every $r>0$.

      \item\label{lem:halfplane_rigid:item:2} The orbit $Gy$ is a line.
  \end{enumerate} 
\end{lem}

\begin{proof}
   \eqref{lem:halfplane_rigid:item:1} Let $r>0$. Suppose that there are two points in $\dist_\partial^{-1}[0,r]$ such that there is a geodesic $c$ between them not contained in $\dist_\partial^{-1}[0,r]$. Then there is an open interval $I=(a,b)$ such that $c(I)$ is contained in $\dist_\partial^{-1}(r,\infty)$ with $c(a),c(b)\in \dist_\partial^{-1}(r)$. By Theorem \ref{thm:ketterer}, $\dist_\partial^{-1}([r/2,\infty))$ is a metric product $[r/2,\infty)\times Z$. Observe that $c(\overline{I})$ is an ambient geodesic hence also intrinsic geodesic, i.e., it is a geodesic in the product space $[r/2,\infty)\times Z$ connecting two points in the cross section $Z$, so it should stay in $\dist_\partial^{-1}(r)$. This is a contradiction.

   \eqref{lem:halfplane_rigid:item:2} Let $\sigma$ be the unit speed horizontal ray emanating at $y$. Observe that each $j\in \mathbb{N}$ gives rise to a geodesic $c_j$ from $\sigma(1/j)$ to $j\cdot \sigma(1/j)\in G\cdot \sigma(1/j)$. By \eqref{lem:halfplane_rigid:item:1}, $c_j$ is contained in $\dist_\partial^{-1}([0,1/j])$. After composing $c_j$ with a suitable group element $g_j\in G$, the midpoint of $g_j\circ c_j$ belongs to $\sigma([0,1/j])$. Let $j\to\infty$, $g_j\circ c_j$ converges to a line in $Y$, which by construction coincides with $Gy=\dist_\partial^{-1}(0)$.  
\end{proof}

We are in a position to complete the last step of Proposition \ref{prop:rigid_free}.

\begin{proof}[Proof of Proposition \ref{prop:rigid_free}]
  By Lemma \ref{lem:halfplane_rigid} and Gigli's splitting theorem (Theorem \ref{thm:meas_split}), $(Y,y,d_Y,\meas)$ is isomorphic to $(\mathbb{R},0,d_E,\mc{L}^1)\otimes (Z,z,d_Z,\mu)$ with the line $Gy$ corresponding to $\mathbb{R}\times \{z\}$. For any $z'\in Z$, the line $L(t)=(t,z')\in \mathbb{R}\times Z$ is parallel to the line $Gy=\mathbb{R}\times \{z\}$ in the sense that $d(L(t),Gy)$ is constant in $t$. Recall that $G=\mathbb{R}$. Hence for $r=\dist_Z(z,z')$, we have
  $$\mathbb{R}\times \{z'\}= \mathrm{im} L = \dist_\partial^{-1}(r) = G\cdot \sigma(r).$$
  Consequently, every $G$-orbit in $Y=\mathbb{R}\times Z$ is of the form $\mathbb{R}\times \{z'\}$. In particular, $G$ acts trivially on the $Z$-factor. Therefore,
  $$[0,\infty)=Y/G = (\R \times Z)/G=Z.$$
  We infer that $\Omega_r(v)=[-v,v] \times [0,r] \subset\R\times Z$, so the measure $\mu$ on $Z=[0,\infty)$ satisfies $\mu([0,r])=r/2$ for any $r\ge 0$. Hence $\mu=\mathcal L^1/2$ on $[0,\infty)$. Restoring the original normalization $c$, $\mu$ is a multiple of the Lebesgue measure, and the result follows.
\end{proof}

\appendix

\section{The limit ratio of linear volume growth}\label{sec:app}
In this appendix, we prove Propositions \ref{prop:linear1} and \ref{prop:linear2} below on linear volume growth. The methods are due to a recent work of Zhou-Zhu \cite[Remark 2.1]{ZZ}. For readers' convenience, here we follow their method and provide the detailed proofs.

  \begin{prop}\label{prop:linear1}
    If $M$ is an open manifold with $\mathrm{Ric}\ge 0$ and 
  	$$\liminf_{r\to\infty} \dfrac{\vol B_r(p)}{r}=\theta \in (0,\infty),$$
  	then $M$ has linear volume growth.
  \end{prop}
  
  \begin{prop}\label{prop:linear2}
  	If $M$ is an open manifold with $\mathrm{Ric}\ge 0$ and linear volume growth, then the limit 
  	$\lim_{r\to\infty}\vol B_r(p)/r$ exists.
  \end{prop}

  We remark that Proposition \ref{prop:linear1} is not used in this paper. We include it here because the proofs of the above two propositions are similar. The proof of Proposition \ref{prop:linear2} requires some additional help from the properness of Busemann functions \cite[Corollary 23]{Sor98}.
  
  We prove Proposition \ref{prop:linear1} first.
  
  Let $p\in M$ and let $\gamma:[0,\infty)\to M$ be a ray in $M$ starting at $p$. We estimate the volume of $B_R(p)$, where $R\ge 10$. Let $\epsilon\in (0,1/10)$. We choose $r$ large such that
  $$r> \epsilon^{-1} R,\quad \mathrm{vol}B_r(p)\le (\theta + \epsilon) r.$$
  Let $q=\gamma(t)$, where $t\gg r$ is large. We denote $\overline{A}_{r_1}^{r_2}(q)$ the closed annulus centered at $q$, that is,
  $$\overline{A}_{r_1}^{r_2}(q)=\{ x\in M | r_1 \le d(x,q)\le r_2 \}.$$
  By triangle inequality,
  $$B_R(p)\subseteq \overline{A}_{t-R}^{t+R} (q),\quad B_{r}(p) \subseteq \overline{A}_{t-r}^{t+r} (q).$$
  
  \begin{lem}\label{lem:linear1_level_set}
  	There is a subset $\mathcal{A}\subseteq[t-r,t-R]$ of measure at least $4\epsilon r$ such that
  	$$\mathcal{H}^{n-1}(\partial B_\rho (q)\cap {B}_{r}(p))\le \dfrac{\theta +\epsilon}{1-5\epsilon} \quad \text{for all } \rho\in \mathcal{A}.$$
  \end{lem}
  
  \begin{proof}
  	$\partial B_\rho (q)$ is $\mathcal{H}^{n-1}$-rectifiable for a.e. $\rho$. By co-area formula, 
  	$$\int_{t-r}^{t-R} \mathcal{H}^{n-1}(\partial B_\rho (q)\cap B_{r}(p)) d\rho \le \mathrm{vol}B_{r}(p)\le (\theta+\epsilon)r.$$
  	Let 
  	$$I=\left\{ \rho \in[t-r,t-R]\ |\  \mathcal{H}^{n-1}(\partial B_\rho (q)\cap B_{r}(p))\ge \dfrac{\theta+\epsilon}{1-5\epsilon} \right\}.$$ Then
  	$$\int_{t-r}^{t-R} \mathcal{H}^{n-1}(\partial B_\rho (q)\cap B_{r}(p)) d\rho \ge |I|\dfrac{\theta+\epsilon}{1-5\epsilon}. $$
  	Hence 
  	$$|I|\le (1-5\epsilon)r < (1-\epsilon)r \le r-R =|[t-r,t-R]|.$$
  	In particular, $\mathcal{A}:=[t-r,t-R]-I$ has measure at least $4\epsilon r$.
  \end{proof}
  
  \begin{proof}[Proof of Proposition \ref{prop:linear1}]
  	Let $T_s=\partial B_s(q) \cap (B_R(p)-\mathcal{C}_q)$, where $s\in [t-R,t+R]$ and $\mathcal{C}_q$ is the cut locus of $q$. By Lemma \ref{lem:linear1_level_set}, we can choose $\rho\in \mathcal{A} \subseteq [t-r,t-R]$ such that $\Sigma= \partial B_\rho(q)\cap \overline{B}_{r}(p)$ satisfies 
  	$$\rho\ge t-r+ 3\epsilon r ,\quad\mathcal{H}^{n-1}(\Sigma)\le \dfrac{\theta +\epsilon}{1-5\epsilon},\quad  \mathcal{H}^{n-1}(\Sigma \cap \mathcal{C}_q)=0.$$ 
  	We note that every minimal geodesic $\sigma$ from $q$ to $z\in T_s$ must intersect $\Sigma$. In fact, let $z^*$ be the intersection of $\sigma$ and $\partial B_\rho (q)$; then by triangle inequality
  	$$d(p,z^*)\le d(p,z)+d(z,z^*)\le R+(s-\rho)
  		\le R + (t+R) - (t-r+3\epsilon r)<r.$$
  	By relative volume comparison, we have 
   \begin{equation}\label{eq:volume_sphere}
    \mathcal{H}^{n-1}(T_s)\le \dfrac{s^{n-1}}{\rho^{n-1}}\mathcal{H}^{n-1}(\Sigma).
   \end{equation}
  	Next, we estimate the volume of $B_R(p)$ by integrating $\mathcal{H}^{n-1}(T_s)$.
  	$$\dfrac{\mathrm{vol}(B_R(p))}{R} =\frac{1}{R} \int_{t-R}^{t+R} \mathcal{H}^{n-1}(T_s) ds
  	\le  \frac{2R}{R} \left( \dfrac{t+R}{t-r} \right)^{n-1} \dfrac{\theta +\epsilon}{1-5\epsilon} 
  	\to  2 \dfrac{\theta+\epsilon}{1-5\epsilon}$$
  	as $t\to\infty$. The result follows.
  \end{proof}

  As indicated, the proof of Proposition \ref{prop:linear2} is similar to that of Proposition \ref{prop:linear1}. A key improvement is that we can pick $p$ attaining the minimum of a Busemann function so that the annulus $\overline{A}_{t-R}^{t+R}(q)$ of width $2R$ in the above proof will be improved to an annulus of width $\sim R$.

  \begin{proof}[Proof of Proposition \ref{prop:linear2}]
    We assume that $M$ is not a metric product $\mathbb{R}\times N$ (see Theorem \ref{thm:sormani}); otherwise, the result holds trivially. Let $\gamma:[0,\infty)\to M$ be a ray and let $$b(x)=\lim_{t\to\infty} t-d(x,\gamma(t))$$ be the Busemann function associated to $\gamma$. Because $M$ has linear volume growth and only one end, by \cite[Corollary 23]{Sor98}, $b$ is a proper function with a minimum $\min b>-\infty$. Let $p\in M$ be a point attaining the minimum of $b$. 
  	
    We continue to write $\theta$ as the liminf of $\vol(B_r(p))/r$ as $r\to\infty$. We estimate the volume of $B_R(p)$, where $R\ge 10$. Let $\epsilon\in(0,1/10)$ and let $r$ large such that
  	$$r> \epsilon^{-1} R,\quad \mathrm{vol}B_r(p)\le (\theta + \epsilon) r.$$ 
  	As $t\to\infty$, the partial Busemann function
  	$b_t(x)=t-d(x,\gamma(t))$
  	converges uniformly to $b$ on every compact set. In particular, on $\overline{B_{r}}(p)$ we have $|b_t-b|\le 1$ for all $t$ sufficiently large. Then
  	$$\overline{B_r}(p)\subseteq b^{-1}[\min b,\min b+r]\subseteq b_t^{-1}[\min b-1, \min b +r +1].$$
  	Setting $q=\gamma(t)$, where $t\gg r$ is large, then
  	$$\overline{B_r}(p) \subseteq \overline{A}_{t-\min b- r -1}^{t-\min b +1} (q),\quad 
  	\overline{B_R}(p) \subseteq \overline{A}_{t-\min b- R -1}^{t-\min b +1} (q).$$
  	For convenience, we denote
  	$$\eta_r= t-\min b -r-1,\quad \eta_R = t-\min b -R-1.$$
  	
  	Let $T_s=\partial B_s(q) \cap (B_R(p)-\mathcal{C}_q)$, where $s\in [\eta_R,t-\min b +1]$. By the proof of Lemma \ref{lem:linear1_level_set}, we choose $\rho\in \mathcal{A} \subseteq [\eta_r,\eta_R]$ such that $\Sigma= \partial B_\rho(q)\cap \overline{B}_{r}(p)$ satisfies 
  	$$\rho\ge \eta_r + 3\epsilon r,\quad \mathcal{H}^{n-1}(\Sigma)\le \dfrac{\theta +\epsilon}{1-5\epsilon},\quad  \mathcal{H}^{n-1}(\Sigma \cap \mathcal{C}_q)=0.$$ 
  	Then every minimal geodesic $\sigma$ from $q$ to $z\in T_s$ must intersect $\Sigma$. Following a similar volume estimate as in the proof of Proposition \ref{prop:linear1}, we have
  	\begin{align*}
  		\dfrac{\mathrm{vol}(B_R(p))}{R} =& \frac{1}{R}\int_{\eta_R}^{t - \min b +1} \mathcal{H}^{n-1}(T_s) ds \\
  		\le & \frac{R+2}{R} \left( \dfrac{t - \min b+1}{ t-\min b-r-1} \right)^{n-1} \dfrac{\theta +\epsilon}{1-5\epsilon} \\
  		\to & \frac{R+2}{R} \cdot \dfrac{\theta +\epsilon}{1-5\epsilon}.
  	\end{align*}
  	Taking $\epsilon\to 0$ and $R\to\infty$, we complete the proof.
  \end{proof}

  Lastly, we extend Proposition \ref{prop:linear2} to $\RCD(0,N)$ spaces with linear volume growth by modifying the proof of Proposition \ref{prop:linear2}. This result is needed in the proof of Theorem \ref{thm:rcd_abel} and may have some independent interest. This property was pointed out to us by Zhu Ye.

  \begin{prop}\label{prop:rcd_vol_ratio}
     If $(X,d,\meas)$ is a non-compact $\RCD(0,N)$ space with linear volume growth, where $N\in[1,\infty)$, then the limit $\lim_{r\to\infty} \meas (B_r(p))/r$ exists for some (and thus for all) $p\in X$.
  \end{prop}

  \begin{proof}
     We continue to use the same notations as in the proof of Proposition \ref{prop:linear2} and assume that $X$ is not a metric product $\mathbb{R}\times X'$. By the work of Huang \cite[Page 1034]{Huang18}, the Busemann function $b$ associated to a ray is proper and has a minimum $\min b>-\infty$. We choose $p\in X$ that attains the minimum of $b$ and set $q=\gamma(t)$, where $t\gg r$.

     Instead of foliating $B_R(p)$ by $\partial B_s(q)$ as we did in the proof of Proposition \ref{prop:linear2}, we intersect $B_R(p)$ with disjoint and thin annuli centered at $q$. Specifically, we first divide the interval $[\eta_R,t-\min b+1]$ into $L$ many subintervals of equal length $\delta=(R+2)/L$ with $\delta\le \epsilon/10$. For convenience, we write the endpoints of subintervals as
     $$s_0=\eta_R,\quad s_1=\eta_R+\delta,\quad ...,\quad s_L=t-\min b+1.$$
     Then we set 
     $$T_i:= A_{s_{i-1}}^{s_i}(q) \cap B_R(p),\quad i=1,...,L.$$

     Next, we foliate $B_r(p)$ similarly by a thin annulus of width $\delta$ and centered at $q$ and choose a suitable comparison annulus. We set 
     $$\rho_0=\eta_r,\quad \rho_1=\eta_r+\delta,\quad ...,\quad \rho_{L'}=\eta_r+L'\delta \le \eta_R,$$
     where $L'$ is the largest integer such that 
     $$L'\cdot \delta \le \eta_R-\eta_r=r-R.$$ We denote 
     $$\Sigma_j= A_{\rho_{j-1}}^{\rho_j}(q)\cap B_r(p),\quad j=1,...,L' $$ 
     and note
     $$\sum_{j=1}^{L'} \meas(\Sigma_j)\le \meas (B_r(p))\le (\theta+\epsilon)r.$$
     We claim that there exists some $\Sigma_j$ such that
     $$\rho_{j-1}\ge \eta_r+3\epsilon r,\quad \meas(\Sigma_j)\le \dfrac{\theta+\epsilon}{1-5\epsilon}\cdot \delta.$$
     In fact, the index set 
     $$J=\left\{j=1,...,L'\mid \meas(\Sigma_j)\ge \dfrac{\theta+\epsilon}{1-5\epsilon}\cdot \delta\right\}$$
     satisfies
     $$\delta(L'- |J|)\ge r-R-\delta -r(1-5\epsilon)\ge 4\epsilon r-\delta.$$
     Hence, we can choose some $j\not\in J$ with $\rho_{j-1}\ge \eta_r+3\epsilon r$. In the following, we denote this $\Sigma_j$ by $\Sigma$ and $\rho_j$ as $\rho$, respectively.

     For each $i$, note that every geodesic from $q$ to $T_i$ must intersect $\Sigma$. In fact, let $\sigma$ be a geodesic from $q$ to a point $z\in T_i$, and $z^*$ any intersection point of $\sigma$ and $A_{\rho-\delta}^\rho(q)$. Then
     \begin{align*}
     d(p,z^*)&\le d(p,z)+d(z,z^*)\le R+(s_i-\rho-\delta)\\
     &\le R+(t-\min b+1) - \eta_r -3\epsilon r\\
     &= r+R-2-3\epsilon r \le r-\epsilon r.
     \end{align*}
     This shows $z^*\in B_r(p)$. Moreover, the $(\rho/s_i)$-midpoint of the geodesic $\sigma$ from $q$ to $z$ is contained in $\Sigma$ because
     $$\dfrac{\rho}{s_i}[s_i-\delta,s_i]\subseteq [\rho-\delta,\rho].$$
     By Brunn-Minkowski inequality \cite[Proposition 2.1]{Sturm_II_06}, we have
     $$\meas(T_i)\le \left(\dfrac{s_i}{\rho}\right)^N \meas(\Sigma).$$

     Now, we can estimate the measure of $B_R(p)$ by summing up that of $T_i$.
     \begin{align*}
        \dfrac{\meas (B_R(p))}{R}&=\dfrac{1}{R}\sum_{i=1}^L \meas(T_i) \\
        &\le \dfrac{1}{R} \sum_{i=1}^L \left(\dfrac{s_i}{\rho}\right)^N \dfrac{\theta+\epsilon}{1-5\epsilon}\cdot \delta \\
        &\le \dfrac{R+2}{R} \left(\dfrac{t-\min b+1}{t-\min b-r-1}\right)^N \dfrac{\theta+\epsilon}{1-5\epsilon}\\
        &\to  \dfrac{R+2}{R}\dfrac{\theta+\epsilon}{1-5\epsilon}.
     \end{align*}
     Letting $\epsilon\to 0$ and $R\to\infty$, we conclude the result.
  \end{proof}

%(see, for example, \cite[Remark 3.3]{Ye23}). 

%-------Bibliography---------
\bibliographystyle{emsjems}
\bibliography{refv2.bib}

\end{document}